\documentclass[11pt]{amsart}
\usepackage{amssymb}
\usepackage{amsthm}
\usepackage{graphicx}
\usepackage[all,cmtip]{xy}
\usepackage{hyperref}
\newcommand{\SA}{{\mathcal{A}}}
\newcommand{\SB}{{\mathcal{B}}}
\newcommand{\SC}{{\mathcal{C}}}
\newcommand{\SD}{{\mathcal{D}}}

\newcommand{\SF}{{\mathcal{F}}}

\newcommand{\SH}{{\mathcal{H}}}

\newcommand{\SL}{{\mathcal{L}}}

\newcommand{\SN}{{\mathcal{N}}}

\renewcommand{\SS}{{\mathcal{S}}}

\newcommand{\SV}{{\mathcal{V}}}

\newcommand{\F}{\mathbb{F}}

\newcommand{\PP}{\mathbb{P}}
\newcommand{\Z}{\mathbb{Z}}
\newcommand{\C}{\mathbb{C}}
\newcommand{\OO}{{\mathbb{O}}}

\newcommand{\N}{\mathbb{N}}
\newcommand{\R}{\mathbb{R}}

\newcommand{\D}{\mathbb{D}}
\newcommand{\CP}{\mathbb{CP}}
\newcommand{\sS}{\mathbb{S}}

\newcommand{\rk}{\operatorname{rk}}

\newtheorem{proposition}{Proposition}[section]
\newtheorem{theorem}[proposition]{Theorem}
\newtheorem{definition}[proposition]{Definition}
\newtheorem{lemma}[proposition]{Lemma}

\newtheorem{corollary}[proposition]{Corollary}
\newtheorem{remark}[proposition]{Remark}

\newtheorem*{rks}{Remark}

\begin{document}

\title{Almost contact $5$--manifolds are contact}

\date{March, 2012}

\thanks{RC is financially supported by a research grant from La Caixa. DP would like to thank ICTP for offering a visiting position in 2011 that allowed him to develop this article. RC and FP are supported by the Spanish National Research Project MTM2010--17389 and ICMAT Severo Ochoa Project SEV--2011--0087.}

\author{Roger Casals}
\address{Instituto de Ciencias Matem\'aticas CSIC-UAM-UC3M-UCM,
C. Nicol\'as Cabrera, 13-15, 28049, Madrid, Spain}
\email{casals.roger@icmat.es}

\author{Dishant M. Pancholi}
\address{Chennai Mathematical Institute,
H1 SIPCOT IT Park, Kelambakkam,
Siruseri Pincode:603 103, TN, India.}
\email{dishant@cmi.ac.in}

\author{Francisco Presas}
\address{Instituto de Ciencias Matem\'aticas CSIC-UAM-UC3M-UCM,
C. Nicol\'as Cabrera, 13-15, 28049, Madrid, Spain}
\email{fpresas@icmat.es}


\renewcommand{\theenumi}{\roman{enumi}}

\begin{abstract}
The existence of a contact structure is proved in any homotopy class of almost contact structures on a closed $5$--dimensional manifold.
\end{abstract}

\maketitle

\section{Introduction} \label{sec:introduction}

Let $(M^{2n+1},\xi)$ be a cooriented contact manifold with associated contact form
$\alpha$, i.e. $\xi=\ker\alpha$. This structure determines a symplectic distribution
$(\xi,d\alpha_{|\xi}) \subset TM$. Any change of the associated contact form $\alpha$ does
not change the conformal symplectic class of $d \alpha$ restricted to $\xi$. 
This  allows us to choose a compatible almost complex structure $J\in End(\xi).$
Thus given a cooriented contact structure we obtain in a natural way a reduction of
the structure group $Gl(2n+1,\R)$ of the tangent bundle $TM$ to the group $U(n)\times \{1\}$,
which is unique up to homotopy, see \cite[Prop. 2.4.8]{Ge}. A manifold $M$ is said to be an \emph{almost contact
manifold} if the structure group of its tangent bundle can be reduced to $U(n)
\times \{ 1 \}$. In particular, cooriented contact manifolds are
almost contact manifolds and such a reduction of the structure group of the
tangent bundle of a manifold $M$ is a necessary condition for the existence of a
cooriented contact structure on $M$. It is unknown whether this condition is in general
sufficient. See however the recent development \cite{BEM}.\\

\noindent Nevertheless there are cases in which the existence of an almost contact structure is sufficient for the manifold to admit a contact structure. For example, if the manifold $M$ is open then one can apply Gromov's
$h$--principle techniques to conclude that the condition is sufficient. See the result 10.3.2 in \cite{EM}. The scenario is quite different for closed almost contact manifolds. Using results of Lutz \cite{Lu1} and Martinet \cite{Ma} one can show that every cooriented tangent $2$--plane field on a closed oriented $3$--manifold is
homotopic to a contact structure. A good account of this result from a modern perspective is given in~\cite{Ge}. For manifolds of higher dimensions there are various results
establishing the sufficiency of the condition. Important instances of these are the construction of contact structures on certain principal $\sS^1$--bundles over closed symplectic manifolds due to Boothby and
Wang~\cite{BW}, the existence of a contact structure on the product of a contact manifold with a
surface of genus greater than zero following Bourgeois~\cite{Bo} and the existence of contact
structures on simply connected $5$--dimensional closed orientable manifolds obtained by Geiges
\cite{Ge1} and its higher dimensional analogue ~\cite{Ge2}.\\

\noindent Let us turn our attention to $5$--manifolds since the main goal of this article is to show that any orientable almost contact $5$--manifold is contact. In this case H. Geiges has been studying existence results in other situations apart from the simply connected one. In \cite{GT1} a positive result is also given for spin closed manifolds with $\pi_1=\Z_2$, and spin closed manifolds with finite fundamental group of odd order are studied in ~\cite{GT2}. On the other hand there is also a construction of contact structures on an orientable $5$--manifold occurring as a product of two lower dimensional manifolds by Geiges and Stipsicz~\cite{GS}. While Geiges used the topological classification of simply connected manifolds for his results in ~\cite{Ge1}, one of the ingredients in \cite{GS} is a decomposition result of a $4$--manifold into two Stein manifolds with common contact boundary ~\cite{AM}, ~\cite{Bk}.\\

Being an almost contact manifold is a purely topological condition. In fact, the reduction of the structure group can be studied via obstruction theory. For example, in the $5$--dimensional situation a manifold $M$ is almost contact if and only if the third integral Steifel--Whitney class $W_3(M)$ vanishes. Actually, using this hypothesis and the classification of simply connected manifolds due to D. Barden ~\cite{Ba}, H. Geiges deduces that any manifold with $W_3(M)=0$ can be obtained by Legendrian surgery from certain model contact manifolds. Though this approach is elegant, it seems quite difficult to extend these ideas to produce contact structures on any almost contact $5$--manifold. We therefore propose a different approach: the existence of an \emph{almost contact pencil} structure on the given almost contact manifold is the required topological property to produce a contact structure. The tools appearing in our proof use techniques from three different sources:\\
\begin{itemize}
\item[-] The approximately holomorphic techniques developed by Donaldson in the
symplectic setting \cite{Do1,Do2} and adapted in \cite{IMP,Na} to the contact
setting to produce the so--called \emph{quasi contact pencil}.\\

\item[-] Eliashberg's classification of overtwisted $3$--dimensional manifolds \cite{El} to produce overtwisted contact structures on the fibres of the pencil.\\

\item[-] The canonical structure of the space of contact elements in a 3--manifold. See \cite{Lu2}.\\
\end{itemize}

\noindent Let us state the main result.

\begin{theorem} \label{main}
Let $M$ be a closed oriented $5$--dimensional manifold. There exists a contact structure in every homotopy class of almost contact structures.
\end{theorem}

\noindent In particular closed oriented almost contact $5$--manifolds are contact. It is important to emphasize that using the techniques developed in this article, it is not possible to conclude anything about the number of distinct contact distributions that may occur in a given homotopy class of almost contact distributions. The result states that there is at least one, the article \cite{Pr2} provides examples with more. It follows from the construction that the contact structure is PS--overtwisted \cite{Ni,NP} and therefore it is non--fillable.

\begin{remark}
The data given by an almost contact structure is tantamount to that of a hyperplane subbundle of the tangent bundle endowed with a complex structure \cite{Ge}. An almost contact structure will refer to either the reduction of the structure group or to such distribution. In the course of the article the distributions are supposed to be coorientable and Section \ref{sec:non-coorientable} contains the corresponding results for non--coorientable distributions.
\end{remark}

\noindent The proof of Theorem \ref{main} consists of a constructive argument in which we obtain the contact condition step by step. These steps correspond to the sections of the paper as follows:\\

\begin{itemize}
\item[-] To begin with, we explain how to produce over any almost contact $5$--manifold $(M, \xi)$ an almost contact fibration over $\sS^2$ with singularities of some standard type. It is defined on the complement of a link. The definition and properties of this almost contact fibration -- in fact, an almost contact pencil -- is the content of Sections  \ref{sec:quasi} and \ref{sec:pencils}. The details of the actual construction are not provided and the reader is referred to \cite{IbortDMT2, MT, Na} for the proofs. The existence of such a pencil is the input data of this article. \\

\item[-] In Section \ref{sec:defor_local}, we produce a first deformation of the almost contact structure $\xi$ to obtain a contact structure in a neighborhood of the singularities of the fibration and in a neighborhood of the link.\\

\item[-] The neighborhood of the link has the structure of a base locus of a pencil occurring in algebraic or symplectic geometry. In order to provide a Lefschetz type fibration we blow--up the base locus. This requires the notion of a contact blow--up. For the purposes of the article, it will be enough to define an appropriate contact surgery of the $5$--manifold along a transverse $\sS^1$. This is the content of Section \ref{sec:blowup}.\\

\item[-] Away from the critical points the distribution splits as $\xi= \xi_v \oplus \SH$, where $\xi_v$ is the restriction of the distribution to the fibres and $\SH$ is the symplectic orthogonal. Section \ref{sec:vertical} deals with a deformation of $\xi_v$ to produce a contact structure in the fibres. It strongly uses the classification of overtwisted contact manifolds due to Eliashberg \cite{El}.\\

\item[-] In Section \ref{sec:skeleton} we begin to deform the horizontal direction $\SH$. This is done in two steps. Given a suitable cell decomposition of the base $\sS^2$, we first deform $\SH$ in the pre--image of a neighborhood of the $1$--skeleton. Section \ref{sec:skeleton} contains this first step.\\

\item[-] The contact condition still has to be achieved in the pre--image of the $2$--cells. This is the second step. The contact structure used in order to fill the pre--image of the $2$--cells is constructed in Section \ref{sec:bands}. This construction uses the contact structure of the space of contact elements of the 3--dimensional fibre. \\

\item[-] In Section \ref{sec:end} we obtain a contact structure on the surgered 5--manifold using the results obtained in Section \ref{sec:bands}. Then we reverse the blow--up surgery and construct the contact structure on the initial $5$--manifold. Theorem \ref{main} is concluded.\\

\item[-] In Section \ref{sec:non-coorientable} we deal with the case of non--coorientable distributions. We introduce the suitable definitions and explain the non--coorientable version of Theorem \ref{main}.\\
\end{itemize}

\noindent The more technical results on this article are contained on Sections 5, 6 and 8. Section 7 (resp. Section 9) is also essential but the exposition can be made less technical and the reader should be able to readily comprehend it once Sections 5 and 6 (resp. Section 8) are understood. Section 6 and 7 can be understood without Section 5 and Section 8 can be read almost independently.\\

\noindent The work in this article was presented in the Spring 2012 AIM Workshop on higher dimensional contact geometry. In its course, J. Etnyre commented on a possible alternative approach in the framework of Giroux's program using an open book decomposition. The argument has been subsequently written and it is the content of the article \cite{Et}.\\

\noindent{\bf Acknowledgements.} The authors are grateful to Y. Eliashberg, J. Etnyre, E. Giroux and H. Geiges for valuable conversations. We are also indebted to the referee for meaningful suggestions. The second author is also grateful to M.S. Narasimhan and T.N. Ramadas for their constant support and encouragement. The proof of Theorem \ref{thm:dishant} was outlined to us by Y. Eliashberg. The original work lacked the construction of the homotopy in the case that $2$--torsion existed in $H^2(M,\Z)$. This case was proven after a useful discussion with J. Etnyre at the AIM Workshop. The present work is part of the authors activities within CAST, a Research Network Program of the European Science Foundation.

\section{Preliminaries.} \label{sec:quasi}

\subsection{Quasi--contact structures} Let $M$ be an almost contact manifold. There exists a
choice of a symplectic distribution $(\xi, \omega) \subset TM$ for such a
manifold. Namely, we can
find a $2$--form $\eta$ on $\xi$ with the property that $\eta$ is
non--degenerate and
compatible with the almost complex structure $J$ defined on $\xi.$  By extending
$\eta$ to a form on $M$  we can find a $2$--form
$\omega$ on $M$ such that $(\xi, \omega_{|\xi})$ becomes a symplectic
vector bundle. This form $\omega$ is not necessarily closed. The triple $(M, \xi, \omega)$ is also said to be an almost contact manifold. In other words, an almost contact structure is meant to be a triple $(\xi,J,\omega)$ for some $\omega$ as discussed. The choice of almost complex structure $J$ is homotopically unique and it might be omitted. An almost contact manifold is subsequently described by a triple $(M,\xi,\omega)$.\\

\noindent In order to construct a contact structure out of an almost contact one, the
first step is to provide a better $2$--form on $M.$ That is, we replace $\omega$ by a closed
$2$--form.

\begin{definition}
A manifold $M^{2n+1}$ admits  a {\it quasi--contact structure} if there exists
a pair $(\xi, \omega)$ such that $\xi$ is a codimension $1$--distribution and
$\omega$ is a closed $2$--form 
on $M$ which is non--degenerate when restricted to $\xi.$
\end{definition}

\noindent Notice that a quasi--contact pair $(\xi,\omega)$ admits a compatible almost contact
structure, i.e. there exists a $J$ which makes $(\xi,J,\omega)$ into an almost contact
structure. These manifolds have also been called {\it
$2$--calibrated}~\cite{IbortDMT} in the literature. The following lemma justifies the appearance of the previous definition:

\begin{lemma} \label{lem:almost-quasi}
Every almost contact manifold $(M, \xi_0, \omega_0)$ admits a quasi--contact
structure $(\xi_1, \omega_1)$ homotopic to $(\xi_0, \omega_0)$ through
symplectic distributions and the class $[\omega_1]$ can be fixed to be any
prescribed cohomology class $a\in H^2(M, \R)$.
\end{lemma}
\begin{proof}
Let $j: M \longrightarrow M \times \R$ be the inclusion as the zero section. Consider a not--necessarily closed $2$--form $\widetilde{\omega}_0$, such that
$\omega_0= j^* \widetilde{\omega}_0$. Fix a Riemannian metric $g$ over $M$ such that
$\xi_0$ and $\ker \omega_0$ are $g$--orthogonal.\\

\noindent Apply Gromov's classification result of open symplectic manifolds
to produce a $1$--parametric family $\{ \widetilde{\omega}_t \}_{t=0}^1$ of
{\it symplectic} forms  such that for $t=1$ the form is closed. See \cite{EM}, Corollary 10.2.2. Let $\pi: M \times \R \longrightarrow M$ be the projection and choose the cohomology class defined by $\widetilde{\omega}_1$ to be $\pi^* a$. Consider  the family of $2$--forms $\omega_t= j^* \widetilde{\omega}_t$ on $M.$
Since $\widetilde{\omega_t}$ is non--degenerate on $M \times \mathbb{R}$ for each $t$,
the form $\omega_t$ has $1$--dimensional kernel $\ker \omega_t$. Define $\xi_t=
(\ker \omega_t)^{\perp g}$. Then $(\xi_t, \omega_t)$ provides the required
family.
\end{proof}

This is the farthest one can reach by the standard $h$--principle argument in
order to find contact structures on a closed manifold. One can start
with the almost contact bundle $\xi= \ker \alpha$ and use Lemma \ref{lem:almost-quasi} to find a $2$--form $d\beta$
such that $(\xi,d\beta)$ is a symplectic bundle, but there is in general no way to relate $\alpha$ and $\beta$. This is the aim of the article.

\subsection{Obstruction theory} The content of Theorem \ref{main} has two parts. The statement implies the existence of a contact structure in an almost contact manifold. This is a result in itself, regardless of the homotopy type of the resulting almost contact distribution. The construction we provide in this article also concludes that the obtained contact distribution lies in the same homotopy class of almost contact distributions as the original almost contact structure. This is achieved via the study of an obstruction class. Let us review some well--known facts. \\

\noindent Let $M$ be a smooth oriented $5$--manifold and $\pi:TM\longrightarrow M$ its tangent bundle. The projection $\pi$ is considered to be an $SO(5)$--principal frame bundle. An almost contact structure is a reduction of the structure group $G=SO(5)$ to a subgroup $H\cong U(2)\times\{1\}\cong U(2)$. The isomorphism classes of almost contact structures are parametrized by the homotopy classes of such reductions. A reduction of the structure group $G$ to a subgroup $H$ is tantamount to a section of a $G/H$--bundle over $M$. Hence the classification of almost contact structures on $M$ is reduced to the study of homotopy classes of sections of a $SO(5)/U(2)$--bundle over $M$.

\begin{lemma}
There exists a diffeomorphism $SO(5)/U(2)\cong\CP^3$.
\end{lemma}
\noindent See \cite[Prop. 8.1.3]{Ge} for the proof of this Lemma. \\

\noindent The homotopy groups $\pi_i(\CP^3)=0$ for $1\leq i\leq6$, $i\neq2$, hence the existence of sections of a fibre bundle with typical fibre $\CP^3$ over the $5$--manifold $M$ is controlled by the primary obstruction class $d=W_3(M)\in H^3(M,\pi_2(\CP^3))\cong H^3(M,\Z)$. The hypothesis of Theorem \ref{main} is $d=0$.\\

\noindent Let $s_\xi$ and $s_{\xi'}$ be two sections of this $\CP^3$--bundle. The obstruction class dictating the existence (or the lack thereof) of a homotopy between them is the primary obstruction $d(\xi,\xi')\in H^2(M,\Z)$. The obstruction theory argument can be made relative to a submanifold $A\subset M$. Given a self--indexing Morse function for the pair $(M,A)$, we consider the relative $j$--skeleton $M_j$ defined as the union of $A$ and the cores of the handles of the critical points of index less or equal than $j$. We have the following
\begin{lemma}\label{lem:2sk}
Consider a relative 2--skeleton $M_2$ for the pair $(M,A)$ and let $s_\xi$, $s_{\xi'}$ be two sections of a $\CP^3$--bundle over $M$ that are homotopic over $M_2$. Then $s_\xi$ and $s_{\xi'}$ are also homotopic over $(M,A)$.
\end{lemma}
\noindent Let $(M,\xi)$ be an almost contact structure, the construction of the contact structure $\xi'$ obtained in Theorem \ref{main} does not modify the homotopy class of the given section, i.e. $s_\xi\sim s_{\xi'}$. In Section \ref{sec:bands} we provide a detailed account on the modification of the obstruction class $d(\xi,\xi')$ in the $2$--skeleton of certain pieces of $M$ where $\xi'$ has been constructed. This is enough to conclude that $d(\xi,\xi')=0$ once $\xi'$ is extended to $M$ in Section \ref{sec:end}.

\subsection{Homotopy of vector bundles} The argument constructing the homotopy between the initial almost contact structure and the resulting contact distribution in Theorem \ref{main} uses the following lemma. It is used in several parts of Sections \ref{sec:defor_local} to \ref{sec:end}.\\

\noindent Let $(V,\omega)$ be an oriented vector space of dimension $\dim_\R V=4$. Consider an splitting $V=V_0\oplus V_1$ with $V_0,V_1$ two oriented $2$--dimensional vector subspaces. Since $Sp(2,\R)/SO(2)$ is contractible, the space of symplectic structures on $V$ such that $V_0$ and $V_1$ are symplectic orthogonal subspaces is contractible. This essentially implies the following

\begin{lemma}\label{lem:split}
Let $M$ be an almost contact $5$--manifold, $A$ an open submanifold of $M$, and $(\xi_0,\omega_0),(\xi_1,\omega_1)$ two almost contact structures on $M$ such that there exists a homotopy $\{\xi_t\}$ of oriented distributions on $(M,A)$ connecting $\xi_0$ and $\xi_1$. Suppose that there exist $L_0$ and $L_1$ two rank--$2$ symplectic subbundles of $\xi_0$ and $\xi_1$ and a homotopy $\{L_t\}\subset\{\xi_t\}$ of oriented distributions connecting $L_0$ and $L_1$ on $(M,A)$. Then there is a path $\{\omega_t\}$ of symplectic structures on $\{\xi_t\}$ such that $\{(\xi_t,\omega_t)\}$ is a path of almost contact structures connecting
 $(\xi_0, \omega_0)$ and $(\xi_1, \omega_1)$ on $(M,A)$.
\end{lemma}
\begin{proof}
Consider $J_0$ and $J_1$ two compatible complex structures on the symplectic distributions $\xi_0$ and $\xi_1$ respectively. These define two fibrewise scalar--product structures
$$g_0=\omega_0(\cdot,J_0\cdot)\mbox{ and }g_1=\omega_1(\cdot,J_1\cdot)$$
on $\xi_0$ and $\xi_1$. The space of fibrewise scalar--product structures has contractible fibre, namely $Gl^+(4,\R)/SO(4)$, and thus it is contractible. Hence, there exists a homotopy $\{g_t\}$ of fibrewise scalar--products connecting $g_0$ and $g_1$. The scalar--product $g_t$ provides an orthogonal decomposition $\xi_t=L_t\oplus L_t^{\perp_{g_t}}$. The homotopy of oriented bundles $\{L_t\}$ induces a homotopy of oriented bundles $\{L_t^{\perp_{g_t}}\}$ respecting the symplectic splitting given by $\omega_0$ and $\omega_1$ on $\xi_0$ and $\xi_1$.
\end{proof}

\subsection{Notation.} Let $\R^{2n}$ be Euclidean space, $B^{2n}(r)=\{p\in\R^{2n}:\|p\|\leq r\}$ denotes the closed ball of radius $r$ centered at the origin. The $2$--dimensional balls are also referred to as disks and denoted by $\D^2(r)$. In case the radius is omitted $B^{2n}$ and $\D^2$ denote the ball and disk of radius $1$ respectively.
\section{Quasi--contact pencils.} \label{sec:pencils} 

Approximately holomorphic techniques have been extremely useful in symplectic
geometry. Their main application in contact geometry -- due to E. Giroux -- is
to establish the existence of a compatible open book for a contact manifold in
higher dimensions. See \cite{Co,Gi,Na}. An open book decomposition is a way of trivializing a
contact manifold by fibering it over $\sS^1$.  Such objects have also
been studied in the almost contact case, see ~\cite{DMTMPresas}.\\

\noindent There exists a construction \cite{Fran1} in the contact case analogous to the Lefschetz pencil decomposition introduced by Donaldson over a symplectic manifold \cite{Do2}. It is called a contact pencil and it allows us to express a contact manifold as a singular fibration over $\sS^2$. It has been extended in \cite{IbortDMT2, MT, Na} to the quasi--contact setting. Theorem \ref{thm:exist_pencil} and Corollary \ref{coro:starting_pencil} in this Section provide the existence of a quasi--contact pencil with suitable properties. Let us begin with the appropriate definitions.

\begin{definition}
An almost contact submanifold of an almost contact manifold $(M, \xi, \omega)$ is an embedded
submanifold $j: S \longrightarrow M$ such that the induced pair $(j^* \xi, j^* \omega)$ is an almost contact structure on $S$.
\end{definition}

\noindent A quasi--contact submanifold of a quasi--contact manifold is defined analogously. In particular this implies in both cases that the submanifold $S$ is transverse to the distribution $\xi$.\\

\noindent A chart $\phi :(U,p)  \longrightarrow V \subset (\C^n \times \R, 0)$ of an atlas of $M$ is compatible with the almost
contact structure $(\xi,\omega)$ at a point $p \in U \subset M$ if the push--forward at $p$ of $\xi_p$ by $\phi$ is $\C^n \times \{ 0\}$ and the $2$--form $\phi_*\omega(p)$ is a positive $(1,1)$--form with respect to the canonical almost complex structure.

\begin{definition} \label{def:almost_pencil}
An almost contact pencil on a closed almost contact manifold $(M^{2n+1}, \xi,
\omega)$ is a triple $(f,B,C)$ consisting of a codimension--$4$ almost contact
submanifold $B$, called the base locus, a finite set $C$ of smooth
transverse curves and a map $f:M\backslash B\longrightarrow \C\mathbb{P}^1$
conforming the following conditions:
\begin{itemize}
\item[(1)] The map $f$ is a submersion on the complement of $C$ and the fibres $f^{-1}(p)$, for any $p\in \CP^1$, are almost contact
submanifolds at the regular points.
\item[(2)] The set $f(C)$ is a finite union of locally smooth curves with transverse self--intersections.
\item[(3)] At a critical point $p\in C\subset M$ there exists a compatible chart $\phi_p$ such that $$(f\circ \phi_p^{-1})
(z_1,\ldots,z_n,s)=f(p)+z_1^2+\ldots+z_n^2+g(s)$$ where $g:(\R,0)\longrightarrow (\C,0)$ is
an immersion at the origin.
\item[(4)] Each $b\in B$ has a compatible chart to $(\C^n \times
\R,0)$ under which $B$ is locally cut out by $\{z_1=z_2=0\}$ and $f$ corresponds
to the projectivization of the first two coordinates, i.e. locally $\displaystyle f(z_1, \ldots, z_n,
t)= \frac{z_2}{z_1}$. 

\end{itemize}
\end{definition}

\begin{remark}
Quasi--contact pencils for quasi--contact manifolds and
contact pencils for contact manifolds are defined by replacing the expression almost
contact by the suitable one in each case.
\end{remark}

\noindent The generic fibres of $f$ are open almost contact submanifolds and the closures
of the fibres at the base locus are smooth. This is because the local model
$(4)$ in the Definition \ref{def:almost_pencil} is a parametrized elliptic
singularity and the fibres come in complex lines $\{\displaystyle z_2=const\cdot z_1\}$
joining at the origin. We refer to the compactified
fibres so constructed as the fibres of the pencil. See Figure \ref{fig:base_points}.\\

\noindent In dimension $5$, each compactified smooth fibre is a smooth 3--manifold containing $B$ as a link and any two different compactified fibres intersect transversely along $B$. Note that if we remove a tubular neighborhood of $C$ in $M$ the compactified fibre over a neighborhood of a point in $f(C)$ becomes a smooth manifold whose boundary is a (union of) 2--tori. This boundary components can be filled by solid tori at any regular fibre.

\begin{figure}[ht]
\includegraphics[scale=0.6]{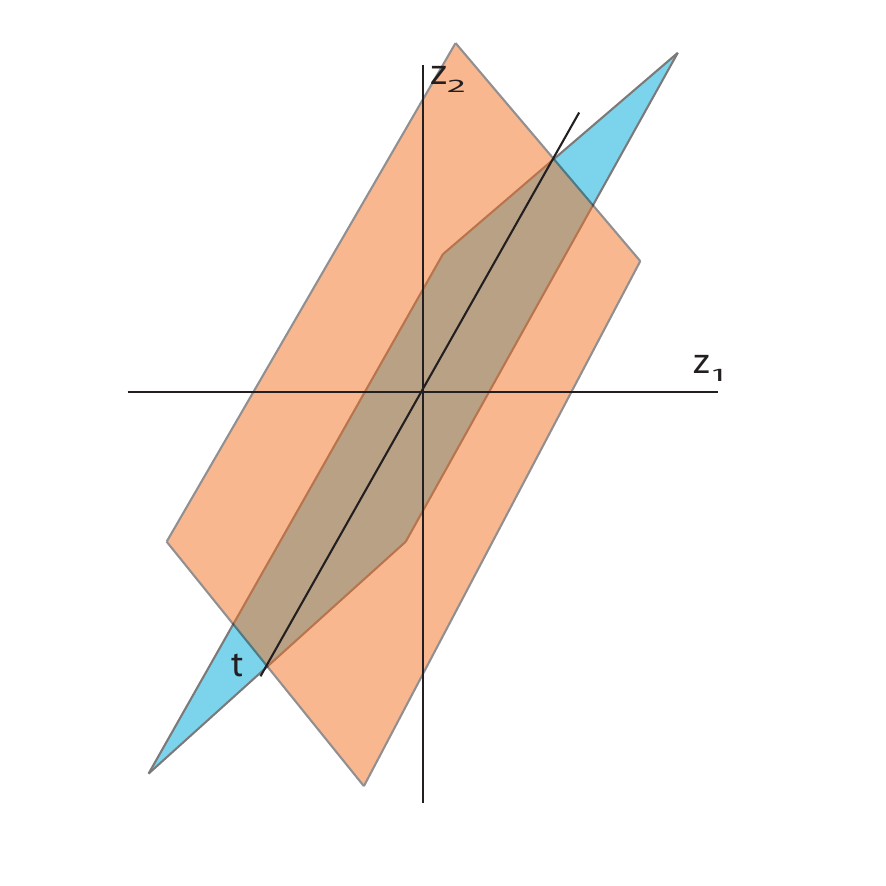}
\caption{Fibres close to the base locus $B=\{z_1=z_2=0\}$.}\label{fig:base_points}
\end{figure}

Notice that the set of critical values $\Delta= f(C)$ are no longer points, as
in the symplectic case, but immersed curves. This is because of Condition $(3)$
in the Definition \ref{def:almost_pencil}. In particular, the usual isotopy
argument between two fibres does not apply unless their images are in the same
connected component of $\CP^1\backslash \Delta$. This has been studied in the contact and quasi--contact cases. The set $C$ is a positive link and therefore $\Delta$ is also oriented. There is a partial order in
the complement of $\Delta$: a connected component $P_0$ is less or equal than
a connected component $P_1$ if $P_0$ and $P_1$ can be connected by an oriented path $\gamma \subset \CP^1$ intersecting $\Delta$ only with positive crossings. The proposition that follows has only been proved for the contact and quasi--contact cases. An analogous statement probably remains true in the almost contact setting. It is provided to offer some geometric insight about contact and quasi--contact pencils, it is not used in the rest of the article.

\begin{proposition} [Proposition $6.1$ of ~\cite{Fran1}] \label{prop:crossing}
Let $M$ be a quasi--contact manifold equipped with a quasi--contact pencil $(f, B,
C)$. Then if two regular values of $f$, $P_0$ and $P_1$, are separated by a unique
curve of $\Delta$ then the two corresponding fibres $F_0=\overline{f^{-1}(P_0)}$
and $F_1=\overline{f^{-1}(P_1)}$ are related by an index $n-1$ surgery.\\

\noindent Suppose that the manifold and the pencil are contact, then the surgery is Legendrian and it attaches a Legendrian sphere to $F_0$ if $P_0$ is
smaller than $P_1$. See Figure \ref{fig:crossing}.
\end{proposition}

\begin{figure}[ht]
\includegraphics[scale=0.6]{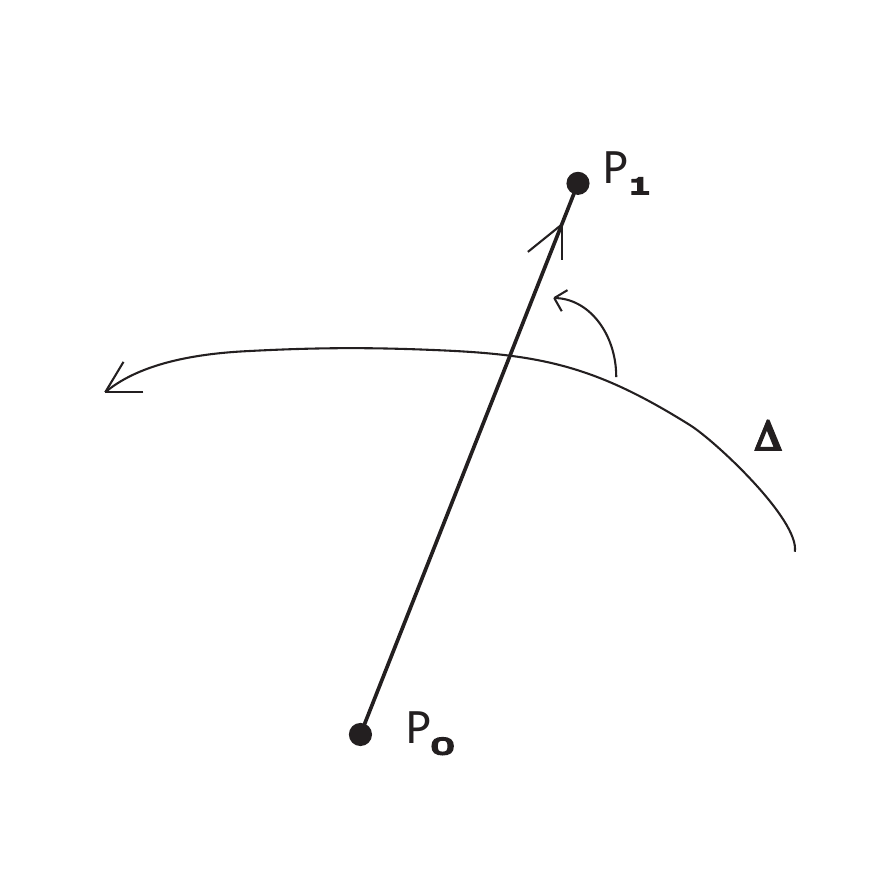}
\caption{According to the orientations, the fibre $F_1=\overline{f^{-1}(P_1)}$ is obtained via a Legendrian
surgery on the fibre $F_0=\overline{f^{-1}(P_0)}$.}\label{fig:crossing}
\end{figure}

\noindent In the contact case it implies that the crossing of a singular curve in
the fibration amounts to a directed Weinstein cobordism. In the quasi--contact
case no such orientation appears. For instance, the case in which the
quasi--contact distribution is a foliation -- in dimension $3$ this is a taut foliation -- becomes absolutely symmetric and there is no difference in crossing one way or the other.\\

\noindent {\bf Examples.} The following two constructions yield simple instances of contact pencils.
\begin{itemize}
\item[1.] Consider a closed symplectic manifold $(M, \omega)$ with $[\omega]$ of integral class and a symplectic Lefschetz pencil $(f,B,C)$ on $(M, \omega)$ as constructed in \cite{Do2}. Consider the circle bundle $\sS(L)$ associated to $\omega$ with its Boothby--Wang contact structure $(\sS(L),\xi_{\omega})$, defined in \cite{BW}, and the projection $\pi: \sS(L)\longrightarrow M$. Then the triple
$$(\pi^*f, \pi^{-1}(B), \pi^{-1}(C))$$
is, after a small perturbation of $\pi^*f$, a contact pencil for $(\sS(L),\xi_{\omega})$.\\

\item[2.] Given two generic complex polynomials in $\C^n$ of high enough degree, we can construct the associated complex pencil $(f,B,C)$. Suppose that the base points set $B$ contains the origin and denote the standard embedding of the radius $r$ sphere by $e_{r}: \sS^{2n-1}\longrightarrow \C^n$. Then for a generic radius $\rho>0$, the triple $(e_{\rho}^*f, e_{\rho}^{-1}(B), \mbox{Crit}(e_{\rho}^*(f)))$ is a contact pencil for $(\sS^{2n-1},\xi_{st})$.\\
\end{itemize}

\noindent Consider a quasi--contact structure $(M, \xi, \omega)$. The main existence result \cite{IbortDMT2, MT, Na} can be stated as

\begin{theorem} \label{thm:exist_pencil}
Let $(M, \xi, \omega)$ be a quasi--contact manifold with $[\omega]$ rational. Given an integral cohomology class $a\in
H^2(M, \Z)$, there exists a quasi--contact pencil $(f, B, C)$ such that the
fibres are Poincar\' e dual to the class $a+k[\omega]$, for any $k\in \N$ large enough. 
\end{theorem}
\noindent The basic construction goes as follows. Consider a line bundle $V$ whose first Chern class equals $a$ and denote by $L$ a Hermitian line bundle over $M$ whose curvature is $-i \omega$. The pencil is constructed using a suitable approximately holomorphic section $\sigma_1^k\oplus \sigma_2^k: M\longrightarrow\C^2\otimes (L^{k} \otimes V)$, this requires $k\in \N$ to be large enough. The pencil map is $f_k=[\sigma_1^k: \sigma_2^k]: M \setminus B_k \longrightarrow \CP^1$ and the base locus is $B_k=\{p \in M:\sigma_1^k(p)=\sigma_2^k(p)=0\}$. A point $p\in M$ maps to $[\sigma_1^k(p):\sigma_2^k(p)]\in\CP^1$. This is well--defined if $p$ is not contained in the base locus $B_k$. The construction is detailed in ~\cite{Na}.\\

\noindent The proof of this result does not work in the almost contact setting. In order to construct the pencil, the approximately holomorphic techniques are essential and for them to work we need the closedness of the $2$--form $\omega$ (so as to be able to construct the line bundle $L$). In general, a quasi--contact pencil may have empty base locus. Nevertheless a pencil obtained through approximately holomorphic sections on a higher dimensional manifold does not.\\

\noindent The following lemma will be useful.
\begin{lemma} \label{lem:zero}
Let $(M,\xi,\omega)$ be an almost contact 5--manifold, $(f, B, C)$ an almost contact pencil adapted to it and obtained from a section $s_1\oplus s_2$ of the bundle $\C^2\otimes\det(\xi)$, and so the base locus is defined as $B=Z(s_1\oplus s_2)$ and the pencil map is $f:= [s_1:s_2]: M \setminus B \to \CP^1$. Then the Chern class of $\xi_F$ vanishes for any regular fibre $(F, \xi_F)$.
\end{lemma}
\begin{proof}
Let $F$ be a regular fibre of $f$, this fibre is defined as the zero set of the section $s_{\lambda}=\lambda_1 s_1 +\lambda_2 s_2$, for a fixed $[\lambda_1: \lambda_2]\in \CP^1$. This is a section of the bundle $\det(\xi)$. Along this fibre $F$, the distribution $\xi$ satisfies
\begin{equation*}
c_1(\xi)|_{F}=c_1(\xi_{F})+c_1(\nu_{F}). 
\end{equation*}
The statement follows from $c_1(\nu_{F})= c_1(\det \xi)|_{F}= c_1(\xi)|_{F}$ inserted in the previous equation.
\end{proof}

\noindent In case the form $\omega$ of the quasi--contact structure is exact -- then called an exact quasi-contact structure -- we obtain the following

\begin{corollary} \label{coro:starting_pencil}
Let $(M,\xi,\omega)$ be an exact quasi--contact closed manifold. Then it admits a
quasi--contact pencil such that any smooth fibre $F$ satisfies
$c_1(\xi_F)=0$. Further, the base locus $B$ is non--empty if $\dim M$
is greater than $3$.
\end{corollary}
\begin{proof}
We use Theorem \ref{thm:exist_pencil} to construct a pencil such that the cohomology class $a\in H^2(M,\Z)$ is fixed to be $a = c_1(\xi) = c_1(\det \xi)$. Since $\omega$ is exact, thus $L\cong\C$, we obtain that the section defining the pencil $s_1\oplus s_2$ is a section of the bundle $\C^2 \otimes(\det \xi \otimes L^k)= \C^2 \otimes\det(\xi)$. Lemma \ref{lem:zero} implies that the almost contact structure induced in the regular fibres of the pencil has vanishing first Chern class.\\

\noindent Let us prove the non--emptiness of the set $B$. It is explained in \cite{IbortDMT2,IMP} that the submanifold $B=Z(\sigma_1^k\oplus \sigma_2^k)$ satisfies a Lefschetz hyperplane theorem (this follows from the fact that it is asymptotically holomorphic). It implies that whenever the dimension
of $M$ is greater than $3$, the morphism
$$ H_0(B) \longrightarrow H_0(M)$$ is surjective. Hence we conclude that $B$ is not the empty set.
\end{proof}

\noindent The triviality of the Chern class of the quasi--contact structures on the fibres and the non--emptiness of $B$ are used in the construction of the contact structure.

\section{Base locus and Critical loops.} \label{sec:defor_local}

\noindent Let $(M,\xi,\omega)$ be an exact quasi--contact $5$--manifold and $(f,B,C)$ a quasi--contact pencil on it. Assume that $B\neq\emptyset$ and $c_1(\xi_F)=0$ for a regular fibre $F$ of $f$. Such a pencil is provided in Corollary \ref{coro:starting_pencil}. A fair amount of control on the almost--contact structure can be achieved in the neighborhood of the base locus and the critical loops.

\begin{definition}
A submanifold $i:S \longrightarrow M$
of an almost contact manifold $(M, \xi, \omega)$ is said to be contact if it is an almost contact submanifold and there is a choice of adapted form $\alpha$ for $\xi$ in a neighborhood $U$ of $S$, i.e. $\xi_{|U} = \ker \alpha$, such that $(d\alpha)_{|U}= \omega_{|U}$.
\end{definition}

\noindent An additional property in our almost contact pencil can then be required.

\begin{definition}
An almost contact pencil $(f,B,C)$ on $(M,\xi, \omega)$ is called good if $B\neq\emptyset$, any smooth fibre $F$ satisfies $c_1(\xi_F)=0$ and $B$ and $C$ are contact submanifolds of $(M,\xi,\omega)$.
\end{definition}

\noindent The following lemma provides a perturbation achieving a suitable almost contact pencil.

\begin{lemma} \label{coro:perturb_pencil}
Let $(M,\xi,\omega)$ be a quasi--contact closed $5$--dimensional manifold and let $(f,B,C)$ be a quasi--contact pencil. There
exists a $C^0$--small perturbation $\{(\xi_t, \omega)\}$ of almost contact
structures such that:
\begin{enumerate}
\item $(\xi_t,\omega)$ is an almost contact structure $\forall t\in[0,1]$, and $(\xi_0,\omega)=(\xi,\omega)$.
\item $B$ and $C$ are contact submanifolds of $(\xi_1,\omega)$.
\item $(f,B,C)$ is an almost contact pencil for $(M,\xi_1,\omega)$.
\item $c_1((\xi_1)_{|F})=0$ for any regular fibre $F$ of $f$.
\end{enumerate}
\end{lemma}

\noindent Fix an associated contact form $\alpha$, i.e. $\xi=\ker\alpha$. The proof of the lemma is an exercise. Indeed, in a neighborhood of the link $B\cup C$ the difference between $\omega$ and $d\alpha$ is exact and its primitive (which can be chosen to vanish along the link) allows us to perturb the defining form until we achieve the contact condition $\omega=d\alpha_1$, $\xi_1=\ker\alpha_1$.\\

\noindent Both Corollary \ref{coro:starting_pencil} and Lemma \ref{coro:perturb_pencil} imply the following

\begin{proposition}\label{prop:good_ac}
Let $(M,\xi,\omega)$ be an exact quasi--contact closed $5$--dimensional manifold. Then there exists an almost contact perturbation $(\xi',\omega)$ of $(\xi,\omega)$ such that $(M,\xi',\omega)$ admits a good almost contact pencil $(f,B,C)$.
\end{proposition}

\section{Surgery and good ace fibrations}  \label{sec:blowup}

Let $(f,B,C)$ be a good almost contact pencil on $(M,\xi,\omega)$. The map $f$ does not define a smooth fibration on $M$ for two reasons: it is not defined on $B$ and there exist critical fibres. The former failure can be avoided if we change the domain manifold $M$, i.e. $f$ can be defined on a suitable closed manifold $\widetilde{M}$ obtained from $M$ by a specific surgery procedure. Let us introduce three pieces of terminology.

\begin{definition}
An almost contact Lefschetz fibration is an almost contact pencil $(f,B,C)$ with $B=\emptyset$. A contact Lefschetz fibration is a contact pencil $(f,B,C)$ with $B=\emptyset$.
\end{definition}

\begin{definition}
An almost contact exceptional fibration on $(M,\xi,\omega)$ is a triple $(f,C,E)$ where $(f,C)$ is an almost contact Lefschetz fibration and $E$ a non--empty collection of embedded 3--spheres with trivial normal bundle such that $f$ restricts to the Hopf fibration on any of them.
\end{definition}
\noindent An almost contact exceptional fibration will be shortened to an ace fibration.
\begin{definition}
An ace fibration is said to be good if the curves $C$ and the spheres in $E$ are contact submanifolds of $(M,\xi,\omega)$, the contact structure in any $3$--sphere of $E$ is the standard tight contact structure and any smooth fibre $F$ of $f$ satisfies $c_1(\xi_F)=0$.
\end{definition}

\noindent An almost contact Lefschetz fibration can be obtained out of an almost contact Lefschetz pencil by performing a surgery along the base locus. In particular, each connected component of the link $B$ is replaced by a standard $3$--sphere $(\sS^3,\xi_{std})$. The aim of this Section is to produce a good ace fibration from a good almost contact pencil on a $5$--dimensional manifold.

\begin{theorem}\label{thm:good_blowup}
Let $(M,\xi,\omega)$ be an almost contact 5--manifold and $(f,B,C)$ a good almost contact pencil. There exist a homotopic deformation $(\xi_1,\omega_1)$ of $(\xi,\omega)$, an almost contact manifold $(\widetilde{M},\widetilde{\xi},\widetilde{\omega})$ with a good ace fibration $(\widetilde{f},E,\widetilde{C})$, a closed neighborhood $\SN(B)$ of $B$ and a diffeomorphism $\Pi:\widetilde{M}\setminus E\longrightarrow M\setminus\SN(B)$ such that
\begin{itemize}
\item[-] The almost contact structure $(\xi_1,\omega_1)$ is contact on a neighborhood of $\SN(B)$.
\item[-] $(\Pi_*\widetilde\xi,\Pi_*\widetilde\omega)=(\xi_1,\omega_1)$ on $M\setminus\SN(B)$.
\end{itemize}
\end{theorem}

\noindent Note that in the context of this article, we are implicitly assuming that the map $f$ has been constructed using asymptotically holomorphic techniques and thus the map $f$ is defined using a section of the bundle $\C^2\otimes\det(\xi)$ (we refer the reader to the paragraph following Theorem \ref{thm:exist_pencil}). The description of the almost contact manifold $(\widetilde{M},\widetilde{\xi},\widetilde{\omega})$ is explicit from the data $(M,\xi,\omega)$. The good ace fibration $(\widetilde{f},E,\widetilde{C})$ is also constructed directly from $(f,B,C)$. This procedure we use is a particular case of a blow--up operation. The analogy with the blow--up of a base point for a symplectic Lefschetz pencil on a 4--manifold can be useful for the reader. See \cite{CPP}.\\

\noindent The description of $(\widetilde M,\widetilde\xi,\widetilde\omega)$ is given in Section \ref{ssec:surgery}. The compatibility of $(\widetilde{M},\widetilde{\xi},\widetilde{\omega})$ with the fibration $(\widetilde{f},C)$ is detailed in Subsection \ref{ssec:comp}. In Subsection \ref{ssec:gace}, we describe a method that ensures that the regular fibres of the new fibration $\tilde{f}$ have vanishing Chern class.

\subsection{Surgery.} \label{ssec:surgery} The almost contact manifold $(\widetilde{M},\widetilde{\xi},\widetilde{\omega})$ is obtained from $(M,\xi,\omega)$ via a surgery procedure. The only topological requirement to perform surgery along a sphere is the triviality of its normal bundle. In contact topology, a standard contact neighborhood also appears in the description. In particular there exists a restriction on the radius in the local model. See \cite{NP}. This is not an issue in the almost contact case: the size of a neighborhood of a contact submanifold of an almost contact manifold can be enlarged by a homotopy of the distribution. In precise terms:

\begin{lemma} \label{lemma:expand}
Let $(M, \xi, \omega)$ be an almost contact manifold and $(S,\xi=\ker\alpha)$ be a contact submanifold with trivial normal bundle $\nu_S\cong S\times\R^{2q}$. Fix a radius $R\in\R$. Then there exists an almost contact homotopy $(M,\xi_t,\omega_t)$ such that $(M,\xi_0,\omega_0)=(M,\xi,\omega)$ and it conforms the following conditions:
\begin{itemize}
\item[-] The homotopy is supported in an annulus around $S$, i.e. given a smooth fiberwise metric on $\nu_S$ there exist $\rho_1,\rho_2\in\R^+$ with $\rho_1<\rho_2$ such that
$$\xi_t|_{\mathbb{D}(\nu_S,\rho_1)}=\xi|_{\mathbb{D}(\nu_S,\rho_1)},\quad 
\xi_t|_{M\setminus \mathbb{D}(\nu_S,\rho_2)}=\xi|_{M\setminus\mathbb{D}(\nu_S,\rho_2)},
$$
where $\mathbb{D}(\nu_S,r)$ is the disk bundle of radius $r$. The almost contact homotopy can be chosen such that $\rho_1,\rho_2$ are arbitrarily small.\\

\item[-] There exist a neighborhood $U$ of $S$ and a diffeomorphism $\varphi$ such that
$$\varphi:S \times B^{2q}(R) \longrightarrow U,\quad\varphi^*\xi_1=\ker(\alpha-r^2\alpha_{std}),\quad\varphi^* \omega_1 =d\alpha-2rdr\wedge d\alpha_{std},$$
where the $1$--form $\alpha_{std}$ is the standard contact form on $\partial B^{2q}(R)$.
\end{itemize}
\end{lemma}
\begin{proof}
This is a statement about a neighborhood $S\times B^{2q}(\varepsilon)$. Suppose that $R>\varepsilon$. In $S\times B^{2n}(\varepsilon)$ the almost contact distribution $(\xi,\omega)$ is a contact structure described as the kernel of the $1$--form $\eta_0= \alpha-r^2 \alpha_{std}$. Consider a function $H\in C^\infty([0,\varepsilon],\R^+)$ such that:
\begin{itemize}
\item[a.] $H(r)=r^2$ for $r\in [0,\varepsilon/4] \cup [3\varepsilon/4,\varepsilon]$,
\item[b.] $H'(r)>0$ for $r \in(0,\varepsilon/2)$,
\item[c.] $H(\varepsilon/2)= R^2$.
\end{itemize}
Consider the two values $\rho_1= \varepsilon/4$ and $\rho_2= \varepsilon$. There exists a homotopy $\{H_t\}$ of functions in $C^\infty([0,\varepsilon],\R^+)$ with $H_0(r)=r^2$, $H_1(r)=H(r)$ and any $H_t$ satisfying properties a and b above. The homotopy of $1$--forms $\eta_t= \alpha-H_t(r)\alpha_{std}$ defines a homotopy of almost contact distributions. The distributions are $\xi_t=\ker\eta_t$. The symplectic structures are of the form $\omega_t=d\alpha-H_t d\alpha_{std}- \SH_t(r)dr\wedge\alpha_{std}$ where $\SH_t(r)$ is a positive smooth function coinciding with $\partial_rH_t$ in $r\in[0,\varepsilon/2)\cup(3\varepsilon/4,\varepsilon]$.
The diffeomorphism
\begin{eqnarray*}
\Psi: S \times B^{2q}(R) &\longrightarrow & S \times B^{2q}(\varepsilon/2) \\
(s,r, \theta) & \longmapsto & (s, \sqrt{H(r)}, \theta)
\end{eqnarray*}
satisfies $\Psi^* \eta_0 = \eta_1$ and the statement of the Lemma follows.
\end{proof}

\noindent The Lemma does not hold for a contact structure since the contact condition is violated at the region $(\varepsilon/2,3\varepsilon/4)$ in the course of the homotopy.\\

\noindent Theorem \ref{thm:good_blowup} concerns both the construction of an almost contact manifold and a good ace fibration. The description of the former naturally leads to that of the latter. Let us then begin with the almost contact manifold. Both the statement and the proof of the following result are relevant. Subsections \ref{ssec:comp} and \ref{ssec:gace} refer to the proof and notation therein.

\begin{theorem} \label{thm:blow-up}
Let $(M^{2n+1}, \xi, \omega)$ be an almost contact manifold and $S\subset M$ a smooth transverse loop. Suppose that $(\xi,\omega)$ is a contact structure on a neighborhood of $S$. There exist a homotopic deformation $(\xi_1,\omega_1)$ of $(\xi,\omega)$, a manifold $\widetilde{M}$, a codimension--$2$ submanifold $E\subset\widetilde M$, a neighborhood $\SN(S)$ of $S$ and a diffeomorphism $\Pi:\widetilde{M}\setminus E\longrightarrow M\setminus\SN(S)$ conforming the following conditions:
\begin{itemize}
\item[-] There exists an almost contact structure $(\widetilde{\xi},\widetilde{\omega})$ on $\widetilde{M}$.
\item[-] The codimension--$2$ submanifold $E$ is a contact submanifold of $(\widetilde{M},\widetilde{\xi},\widetilde{\omega})$ contactomorphic to the standard contact sphere $(\sS^{2n-1},\xi_{st})$.
\item[-] $(\Pi_*\widetilde\xi,\Pi_*\widetilde\omega)=(\xi_1,\omega_1)$ on $M\setminus\SN(S)$.
\end{itemize}
The submanifold $E$ is called the exceptional divisor.
\end{theorem}

\begin{proof} This proof depends on a fixed integer $k\in\Z$. This parameter becomes relevant in the description of the good ace fibration $(\widetilde{f},E,\widetilde{C})$. It can be chosen quite arbitrarily in this argument, but there shall be a specific choice in the proof of Theorem \ref{thm:good_blowup}.\\

\noindent Consider the standard contact form  $\alpha_{std}$ on $\sS^{2n-1}$, induced by the restriction of the standard Liouville form on $\R^{2n}$, and the contact structure $\xi_{std}= \ker \{ d\theta - \rho^2 \alpha_{std} \}$ on $\sS^1\times B^{2n}$ endowed with polar coordinates $(\theta;\rho,\sigma)$. The contact neighborhood theorem for the transverse loop $S$ provides an open neighborhood $U$ of $S$, a constant $\rho_0\in\R^+$ and a diffeomorphism
\begin{eqnarray*}
\phi: S \times B^{2n}(\rho_0) & \longrightarrow & U \\
(\theta, \rho, \sigma) & \longmapsto & \phi(\theta, \rho, \sigma)
\end{eqnarray*}
such that $\phi^*(\xi_{|_U}) =  \xi_{std}$. If $k$ is a positive integer, suppose that the radius $\rho_0$ is small enough so that $k\rho_0^2<1$. This condition is necessarily satisfied for $k<0$. Consider the positive number $\rho_k\in \R^+$ satisfying $\rho_0= \frac{\rho_k}{\sqrt{1+ k \rho_k^2}}$ and the diffeomorphism
\begin{eqnarray*}
\psi_k: \sS^1 \times B^{2n}(\rho_k) & \longrightarrow & \sS^1 \times B^{2n}\left(\rho_0\right) \\
(\theta, \rho, w_1, \ldots, w_n) & \longmapsto & \left(\theta,  \frac{\rho}{\sqrt{1+k\rho^2}}, e^{ik\theta}w_1, \ldots, e^{ik\theta}w_n\right).
\end{eqnarray*}
The map $\psi_k$ preserves the distribution $\xi_{std}$. In case it is needed, apply the Lemma \ref{lemma:expand} to enlarge the neighborhood $\sS^1 \times B^{2n}(\rho_k)$ of $S$ to radius $R=2$. This yields a deformation $\xi_1$ of the contact structure $\xi_{std}$ supported in an annulus of radii $0 < \rho_a < \rho_b <\rho_k$ and a compatible embedding $\varphi: \sS^1 \times B^{2n}(2) \longrightarrow \sS^1 \times B^{2n}(\rho_b)$. The deformation is relative to the boundary and thus the distribution $(\phi\circ\psi_k\circ\varphi)_*(\xi_1)$ defined over $U$ admits an extension $\xi_1$ over $M$ using the original distribution $\xi$. There is also a corresponding extension for the symplectic structure $\omega_1$. To ease notation, we still refer to $(\xi_1,\omega_1)$ as $(\xi,\omega)$. In these terms, Lemma \ref{lemma:expand} provides a neighborhood $U'$ of $S$ in $M$ and a diffeomorphism
$$\Phi: \sS^1 \times B^{2n}(2) \longrightarrow U',\quad (\theta,r,\sigma)\longmapsto\Phi(\theta,r,\sigma)= \phi \circ \psi_k \circ \varphi,\quad \Phi^*(\xi|_S)=\ker(d\theta -r^2 \alpha_{std}).$$

\noindent Consider the diffeomorphism
\begin{eqnarray*}
\phi_1: \sS^1\times (3/2,2) \times \sS^{2n-1} & \longrightarrow & \sS^1\times (3/2,2) \times \sS^{2n-1} \\
(\theta, r, w_1, \ldots, w_n) & \longrightarrow & (\theta, r, e^{i\theta}w_1, \ldots, e^{i\theta}w_n).
\end{eqnarray*}

\noindent If $V= \Phi(\sS^1 \times B^{2n}(3/2))$, then $g= \Phi \circ \phi_1: \sS^1\times (3/2,2) \times \sS^{2n-1} \longrightarrow U\setminus V \subset M$ satisfies
$$g^* \xi = \ker \left\{-\left(\alpha_{std} + \frac{r^2-1}{r^2}d\theta \right)\right\}.$$

\noindent Note that the function
\begin{eqnarray*}
h: (3/2,2) & \longrightarrow & \R \\
r & \longmapsto & h(r)= \frac{r^2-1}{r^2}
\end{eqnarray*}
satisfies $h(r)>5/9$. Therefore it is possible to extend it to a smooth function $\widetilde{h}:[0,2) \longrightarrow \R$ satisfying the following conditions (See Figure \ref{fig:h}):
\begin{itemize}
\item[-] $\widetilde{h}(r)=r^2$, for $r\in[0,1/2]$,
\item[-] $\widetilde{h}(r)=h(r)$, for $r>3/2$,
\item[-] $\widetilde{h}(r)'>0$ for $r\in[1/2,3/2]$.
\end{itemize}
\begin{figure}[ht]
\includegraphics[scale=0.4]{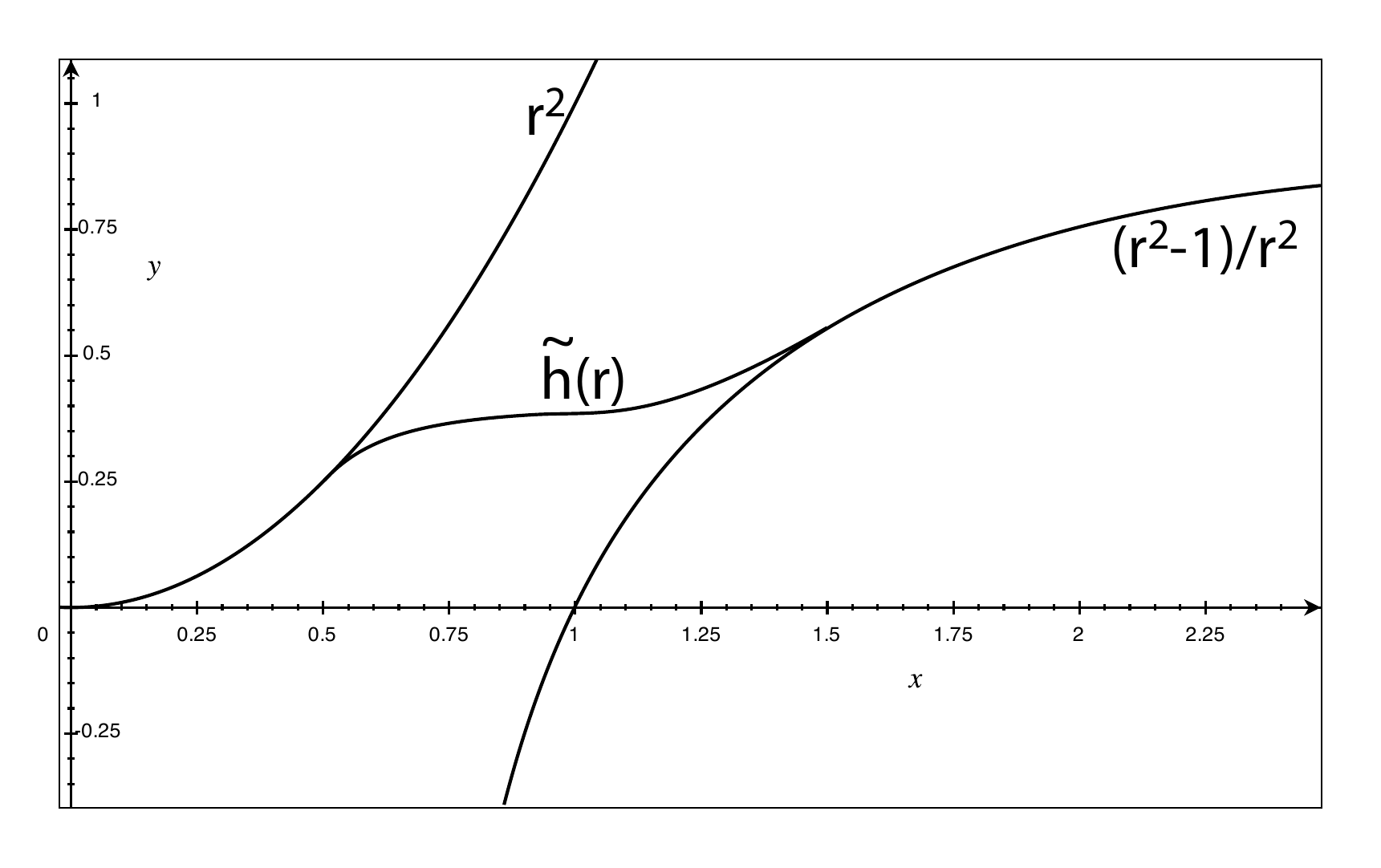}
\caption{The function $\widetilde{h}$.}\label{fig:h}
\end{figure}
Therefore $\widetilde{\eta}=-\alpha_{std}-\widetilde{h}(r) d\theta$ defines a distribution $\widetilde\xi$ over $\sS^1\times [0,2) \times \sS^{2n-1} \cong B^2(2) \times \sS^{2n-1}$. Note that $\widetilde\eta$ is a contact form near the core $\{0\}\times\sS^{2n-1}$. We can glue the manifold $(M\setminus V, \xi)$ and $(B^2(2) \times \sS^{2n-1}, \ker\widetilde{\eta})$ with the gluing map $g$ to define an almost contact manifold $(\widetilde{M}, \widetilde{\xi})$. This manifold satisfies the statement of the theorem with $\SN(S)=\Phi(\sS^1 \times B^{2n}(1))$.
\end{proof}
\subsection{Compatibility with an almost contact pencil.}\label{ssec:comp}
Let $(f,B,C)$ be a good almost contact pencil on a $5$--dimensional almost contact manifold $(M,\xi,\omega)$. The almost contact structure $(\xi_1,\omega_1)$ obtained in Lemma \ref{lemma:expand} can be chosen to remain adapted to the almost contact pencil $(f,B,C)$ (this can be done by proving a standard neighborhood theorem using the local models provided by the definition of a good almost contact pencil). Let us understand the choices involved in the Theorem \ref{thm:blow-up}. The map $f$ pulls--back to
$$f\circ\Pi:\widetilde M\setminus E\longrightarrow\CP^1.$$
Due to the surgery procedure it can be extended to a map $\widetilde f:\widetilde M\longrightarrow\CP^1$. Let us explain this.\\

\noindent The first choice in the previous construction is the chart map $\phi: \sS^1 \times B^{2n}(\rho_0) \longrightarrow U$ for a neighborhood $U$ of a connected component $\gamma\cong\sS^1$ in the base locus $B$. This amounts to a choice of framing of the trivial normal bundle along this $\sS^1$. Since $\sS^1\subset B$ we can use the adapted charts in Definition \ref{def:almost_pencil} and require that $\phi$ satisfies that the map
$$f \circ \phi: \sS^1 \times(B^{4}(\rho_0)\backslash\{ 0 \}) \longrightarrow \CP^1$$
is precisely $(f \circ \phi)(\theta, w_1, w_2)= [w_1 : w_2]$. Therefore, the compactified fibres are of the form $\sS^1 \times L$, for any complex line $L\subset \C^2$. It is also satisfied that $(f \circ \phi \circ \psi_k)(\theta, w_1, w_2)= [w_1 : w_2]$ and again the same compactification for the fibres still holds. Moreover the fibres are almost contact. It is left to study the effect of $\varphi$ and $\phi_1$.\\

\noindent The deformation performed in the enlargement of the neighborhood from $(\xi_0, \omega_0)$ to $(\xi_1, \omega_1)$ preserves the fibres as almost contact submanifolds. The reason being that in Lemma \ref{lemma:expand} the fibres in the coordinates $(\theta, \rho, \sigma)=(\theta, \rho, w_1, w_2)$ are given by the equation
$$ F_{z}= \{ (\theta, \rho, w_1, w_2): [w_1:w_2]=z\}\quad\mbox{for }z\in \CP^1,$$
and the restriction of $(\xi_1, \omega_1)$ is given by
$$(\ker\{ d\theta + H(\rho) (\alpha_{std})_{|_{\sS^3 \cap L_z}}\},\SH(\rho)d\rho\wedge(\alpha_{std})_{|_{\sS^3 \cap L_z}}),$$
where $L_z$ is the line represented by $z\in\CP^1$ and $\SH$ is a smooth function which equals $\partial_\rho H$ in the region of radius $\rho\in[0,\rho_a)\cup(\rho_b,\rho_k]$ and it is strictly positive for $\rho\in[\rho_a,\rho_b]$. In particular, $\SH$ is positive and the restriction of $\omega_1$ is indeed a symplectic structure. \\

\noindent Let us focus on the compactification of fibres in $\widetilde M$, i.e. the extension of $\widetilde{f}$ from $\Pi^{-1}(M\setminus\SN(B))$ to $\widetilde{M}$. We first restrict ourselves to the transition region $\sS^1\times(3/2,2)\times\sS^3\subset\sS^1\times\C^2$. The gluing map is $\phi \circ\psi_k\circ\varphi\circ\phi_1$. In order to understand the fibres we just need to describe the map $\widetilde{f}=f\circ g=f\circ\phi\circ\psi_k\circ\varphi\circ\phi_1$. We can easily verify that
$$\widetilde{f}(\theta,r,w_1,w_2)=(f\circ g)(\theta,rw_1,rw_2)=[w_1:w_2]$$
since $\phi \circ\psi_k\circ\varphi$ and $\phi_1$ act as complex scalar multiplication in the transition area.\\

\noindent Notice that the domain of definition of $\widetilde{f}$ is $\sS^1\times(3/2,2)\times\sS^3$, and it is invariant with respect to the coordinates $(\theta,r)\in\sS^1\times(3/2,2)$. Hence, the map $\widetilde{f}$ extends trivially to the model $(B^2(2)\times\sS^3,\ker\widetilde{\eta})$. In particular, the extension of $\widetilde{f}$ restricted to the exceptional divisor $\{0\}\times\sS^3$ is the Hopf fibration.\\

\noindent The fibres of the fibration $\widetilde f$ are thus almost contact submanifolds. The critical locus $\widetilde C$ is in bijection with $C$ and it is a contact submanifold since the almost contact structure remains unchanged near them. The exceptional divisors $E$ are also contact submanifolds and the fibres of $\widetilde{f}$ restricted to $(B^2(2)\times\sS^3,\ker\widetilde{\eta})$ are diffeomorphic to $B^2(2)\times\sS^1$, the $\sS^1$--factor being a transverse Hopf fibre. These fibres are also contact submanifolds.

\subsection{The good ace fibration}\label{ssec:gace} The fibres $\widetilde{F}$ of the Lefschetz fibration $(\widetilde{f},\widetilde C)$ differ from the fibres $F$ of $(f,B,C)$. Let us provide a precise description of $\widetilde{F}$ and show that the procedure described in the previous two subsections can be performed to obtain $c_1(\widetilde{\xi}_{\widetilde{F}})=0$. This concludes Theorem \ref{thm:good_blowup}.\\

\noindent The trivialization of a neighborhood of a connected component $\gamma\cong\sS^1\subset B$ of the base locus provided in Definition \ref{def:almost_pencil} induces a natural framing $\nu_S\cong\sS^1\times\C^2$, i.e. $\langle(1,0),(i,0),(0,1),(0,i)\rangle$. It restricts to a framing inside the two fibres corresponding to the two complex axes of $\C^2$. Hence it induces framings in any complex line $\sS^1\times\C\subset\sS^1\times \C^2$: for the complex line $\{(z,w)\in\C^2:z-\alpha w=0\}$, we use $\langle(\alpha,1),i(\alpha,1)\rangle$. Denote by $\mathbb{F}_p(0)$ such framing of $B\subset\overline{f^{-1}(p)}$. Let $\mathbb{F}_p(n)$ be the $n$--twist of $\mathbb{F}_p(0)$ and $k_\gamma$ be the parameter used in the construction of Theorem \ref{thm:blow-up} when performing the surgery along $\gamma$.

\begin{lemma}
Let $(M,\xi,\omega)$ be an almost contact 5--manifold, $(f, B, C)$ a good almost contact pencil adapted to it and $(\widetilde{M} ,\widetilde{\xi},\widetilde\omega)$ a manifold as described in Theorem \ref{thm:blow-up}. Then $(\widetilde{M} ,\widetilde{\xi},\widetilde\omega)$ has an almost contact fibration
$(\tilde{f},\widetilde C)$ that coincides with $(f,B,C)$ away from $B=\gamma_1\cup\ldots\cup\gamma_s$. Near $\gamma\in B$ the fibre over $p\in\CP^1$ is contactomorphic to a transverse contact $(0,1)$--surgery performed on $\overline{f^{-1}(p)}$ along $\gamma_i$ with framing $\mathbb{F}_p(-k_i-1)$, for some $k_i\in\Z$. The restriction of the map $f$ to each of the exceptional
divisors is given by the Hopf fibration.
\end{lemma}

\begin{proof}
The map $\psi_k$ in Theorem \ref{thm:blow-up} modifies the initial framing from $\mathbb{F}_p$ to $\mathbb{F}_p(-k_i)$, $k_i=k_{\gamma_i}$ being the corresponding parameter $k$ in the surgery along $\gamma_i$. Using the map $\phi_1$ substracts another twist and sends the meridian to the longitude of the added solid torus. It is thus a $(p,q)=(0,1)$--Dehn surgery with respect to $\mathbb{F}_p(-k_i-1)$.
\end{proof}

\noindent Note that the coefficients $k_i$ can be arbitrarily chosen. The constructive argument will use the fact that $c_1(\widetilde{\xi}_{\widetilde{F}})=0$ for any fibre $\widetilde{F}$ of $\widetilde f$. This has been achieved for the initial fibres of the pencil. The procedure changes the almost contact manifold $(F,\xi)$ to $(\widetilde{F},\widetilde{\xi})$ and we cannot directly assume that $c_1(\widetilde{\xi}_{\widetilde{F}})=0$. This will be fixed in the following discussion.
\begin{proposition}\label{prop:choice}
Let $(M,\xi,\omega)$ be an almost contact 5--manifold, $(f, B, C)$ a good almost contact pencil adapted to it and $(\widetilde{M} ,\widetilde{\xi},\widetilde\omega)$ a manifold obtained as in Theorem \ref{thm:blow-up}. Suppose that $(f,B,C)$ is obtained via asymptotically holomorphic sections as in Corollary \ref{coro:starting_pencil}. There is a choice of $(k_1,\ldots,k_s)\in \Z^s$ such that the first Chern class of the almost contact structure $(\widetilde{M},\widetilde{\xi},\widetilde{\omega})$ on any regular fibre $\widetilde{F}$ is zero.\label{prop:good_ac}\end{proposition}

\noindent In the proof there is no need for the sections to be asymptotically holomorphic. The only requirement is that the pencil is obtained as the linear system associated to two sections.

\begin{proof}
Consider a connected component $\gamma\subset B$. The good almost contact pencil is obtained from a section
$$s=(s_0, s_1): M \longrightarrow \C^2 \otimes\det(\xi).$$
and it is the input of Corollary \ref{coro:starting_pencil}.\\

\noindent Suppose that the section $(s_0,s_1)$ can be lifted to a non--vanishing section $(\widetilde{s}_0, \widetilde{s}_1)$ from the manifold $\widetilde M$ to the bundle  $\C^2 \otimes \det \widetilde{\xi}$. That is, the map $\tilde{f}$ comes as a quotient of two sections $(\widetilde{s}_0, \widetilde{s}_1)$ of the bundle $\det \widetilde{\xi}$. Then Lemma \ref{lem:zero} implies that its regular fibres satisfy the required property. Hence, we just need to find a non--vanishing lift of the two sections $(s_0,s_1)$. Let us show that this lift exists for a particular choice of integers $(k_1, \ldots, k_s)$.\\

\noindent The study of sections of a complex bundle $\det\xi$ with $\xi\subset TM$ does not depend on the homotopy class of $\xi$ as a complex subbundle of $TM$. In particular, we can deform $\xi$ to a complex subbundle $\xi_h$ and study the extension properties of two sections of $\det(\xi_h)$ corresponding to a deformation of $(s_0,s_1)$. The bundle $\xi_h$ yields simpler computations. A word of caution, the notation $\xi_h$ will now be used to refer to a distribution in a local chart and not in the manifold $M$ itself.\\

\noindent Consider polar coordinates $(\theta;r,\sigma)\in\sS^1\times B^4(2)$. The pull--back of the distribution $\xi$ by the map $\Phi=\phi\circ\psi_k\circ\varphi$ is
$$\Phi^*(\xi)=\ker\eta,\quad \eta=d\theta+r^2\alpha_{std}.$$
Let $\chi:[0,2]\longrightarrow[0,1]$ be a smooth increasing function such that
$$\chi|_{[0,1.7]}=0\mbox{ and }\chi|_{[1.9,2]}=1.$$
Define the form $\eta_h=d\theta+\chi(r)r^2\alpha_{std}$ and the distribution $\xi_h=\ker\eta_h$. The distribution $\Phi_*\xi_h$ can be extended to the manifold $M$ using $\xi$. A linear interpolation between $\eta$ and $\eta_h$ induces a homotopy between the two complex bundles $\Phi^*\xi$ and $\xi_h$. The map $\phi_1$ is a diffeomorphism in $\sS^1\times(1.5,2)\times\sS^3$. The pull--backs of the kernels of these two forms via the map $\phi_1|_{\sS^1\times(1.5,1.7)\times\sS^3}$ are two distributions $\phi_1^*(\ker\eta)$ and $\phi_1^*(\ker\eta_h)$. \\

\noindent Consider the function $\widetilde{h}$ defined in the proof of Theorem \ref{thm:blow-up} and a smooth increasing function $\sigma: [0, \infty) \longrightarrow [0, \pi/2]$ constant equal to $0$ in $[0,1/2]$ and constant equal to $\pi/2$ in $[1.5,\infty)$. Define also the form
$$\widetilde{\eta}_h=\sin(\sigma(r))d\theta+\cos(\sigma(r))\alpha_{std}.$$

\noindent First, the kernel of the contact form $\widetilde{\eta}=\alpha_{std}+\widetilde{h}d\theta$ extends the distribution $\phi_1^*(\ker\eta)$ to $B^2(1.7)\times\sS^3$, with polar coordinates $(r,\theta)\in B^2(1.7)$. Let $\widetilde\xi$ be the push--foward to the manifold of $\ker\widetilde\eta$ extended by $\phi_1^*(\ker\eta)$. Second, the distribution $\phi_1^*(\ker\eta_h)$ coincides with $\ker d\theta$ in $\sS^1\times(1.5,1.7)\times\sS^3$ and $\ker\widetilde{\eta}_h$ extends $\phi_1^*(\ker\eta_h)$ to $B^2(1.7)\times\sS^3$. Let $\widetilde\xi_h$ be the push--foward to the manifold of $\ker\widetilde\eta_h$ extended by $\phi_1^*(\ker\eta_h)$. The distributions $\ker\widetilde\eta$ and $\ker\widetilde{\eta}_h$ are homotopic via linear interpolation. The homotopy coincides with the homotopy between $\Phi^*\xi$ and $\xi_h$ in the region $\sS^1\times(1.5,1.7)\times\sS^3$. Hence, the homotopy extends to a homotopy between $\widetilde\xi$ and $\widetilde\xi_h$ inside the manifold $M$.\\

 \noindent Let $X_r=\partial_r,X_i=iX_r,X_j=jX_r,X_k=kX_r$ be a basis generating $T\C^2=\C^2\cong\mathbb{H}^1$. 
Consider the chart defined by $\phi$ with polar coordinates
$$(\theta; r, w_0,w_1)\in\sS^1\times\C^2\cong\sS^1\times\R^{\geq0}\times\sS^3.$$
The distribution $\xi_h=\ker d\theta$ will be identified with $\C^2$. The original sections $(s_0,s_1)$ will be identified as sections of $\Phi_*\det\xi_h$. Suppose the sections $(s_0,s_1)$ restrict to an $m$--twisted frame, i.e. in the chart above the pair of sections is written up to homotopy as 
$$\phi^*(s_0, s_1) \simeq e^{m\cdot i\theta}(w_0, w_1) (1,0)\wedge(0,1).$$
The change of coordinates is defined, up to homotopy, by
$$(\psi_k \circ\varphi\circ\phi_1)(\theta, r, w_0, w_1)= (\theta, r, e^{i(1+k)\theta}w_0, e^{i(1+k)\theta}w_1).$$
It pulls--back the basis framing to
$$(\psi_k \circ\varphi\circ\phi_1)^* (1,0) \wedge (0,1)= e^{-2i(1+k)\theta}(1,0) \wedge (0,1).$$
Therefore the pull--back of the 2 sections is
\begin{eqnarray*}
(\Phi\circ\phi_1)^* (s_0, s_1) & = & (\phi\circ\psi_k \circ\varphi\circ\phi_1)^*  \simeq  e^{(m-k-1)\cdot i\theta}(w_0, w_1) (1,0)\wedge (0,1)= \\ & = &
e^{(m-k-1)\cdot i\theta}(w_0, w_1) X_r\wedge X_j
 =  -i e^{(m-k-1)\cdot i\theta}(w_0, w_1) X_i\wedge X_j.
\end{eqnarray*}
Observe that $k$ controls the twisting of the section around the component $\gamma$. The distribution $\xi_h$ is extended to $B^2(1.7)\times\sS^3$ with the distribution $\Phi^*\widetilde\xi_h$. The four vector fields $X_r,X_i,X_j,X_k$ define a framing of $\xi_h$ in $\sS^1\times(1.5,1.7)\times\sS^3$. This framing needs to be extended to the interior $B^2(1.7)\times\sS^3$ to a framing of the distribution
$$\ker\widetilde\eta_h=\ker\{\sin(\sigma(r))d\theta+\cos(\sigma(r))\alpha_{std}\}.$$
A possible extension is given by $\langle X_r,\sin(\sigma(r))X_i-\cos(\sigma(r))\partial_\theta,X_j,X_k\rangle$.\\

\noindent Consider $p=m-k-1$ and let us identify $\Phi_*\xi_h$ and $\widetilde\xi_h$ in their common region. The section $(\Phi\circ\phi_1)^*(s_0,s_1)$ seen as a section of $\C^2\otimes\det\widetilde\xi_h$ can be extended to
$$(\widetilde{s}_0, \widetilde{s}_1)\simeq -ie^{p\cdot i\theta}(w_0, w_1) (\sin(\sigma(r))X_i-\cos(\sigma(r))\cdot\partial_{\theta})\wedge X_j.$$
Thus it is an extension of the section to $\widetilde{M}$. For radius $r=0$, in the new compactification $B^2(r,\theta)\times\sS^3(w_0, w_1)$, the section reads
$$(\widetilde{s}_0, \widetilde{s}_1)= ie^{p\cdot i\theta}(w_0, w_1) \partial_{\theta}\wedge X_j,$$
which extends without zeroes if and only if $p=-1$. The choice $k=m$ allows us to extend the section $(\widetilde{s}_0, \widetilde{s}_1)$  to the interior of the exceptional sphere without zeroes.\\

\noindent In short, the required section $\widetilde{s}=(\widetilde{s_0},\widetilde{s_1})$ extends to the previous section $s=(s_0,s_1)$ away from the surgery area. Since the sections can be extended to the manifold $\widetilde M$ in a non--vanishing manner we conclude $c_1(\widetilde{\xi}|_{\widetilde{F}})=0$ and the base locus is empty, that is $\widetilde B=\emptyset$.
\end{proof}

\noindent This concludes the proof of Theorem \ref{thm:good_blowup}. The argument developed in this article to prove Theorem \ref{main} requires a smooth fibration, hence the reason for Theorem \ref{thm:good_blowup}. There is an alternative approach not involving the manifold $\widetilde M$ that leads to a quite complicated version of the local models used in Sections \ref{sec:vertical}, \ref{sec:skeleton} and \ref{sec:bands}. These models are essential to describe the deformation of the almost contact structure. The simpler, the better. In particular, the description in Section \ref{sec:bands} would be rather technical if the modified model was used.\\

\section{Vertical Deformation.}  \label{sec:vertical}

In Section \ref{sec:pencils} we endowed our initial $5$--dimensional almost contact manifold $(M,\xi,\omega)$ with an almost contact pencil  $(f,B,C)$ such that $B\neq0$ and $c_1(\xi_F)=0$ for the fibres $F$ of $f$. In Proposition \ref{prop:good_ac} we have obtained a contact structure in a neighborhood of the base locus $B$ and the critical curves $C$. According to Theorem \ref{thm:good_blowup} there exists a good ace fibration $(\widetilde{f},E,\widetilde C)$ in an almost contact manifold $(\widetilde{M},\widetilde{\xi},\widetilde{\omega})$ isomorphic to $(M\setminus\SN(B),\xi,\omega)$ away from a codimension--2 contact submanifold $E$. In order to obtain a contact structure in the manifold $(M,\xi,\omega)$ we use the splitting induced by the existence of the Lefschetz fibration $(\widetilde{f},\widetilde C)$ on  $(\widetilde{M},\widetilde{\xi},\widetilde{\omega})$. Henceforth we shall consider an almost contact manifold with a good ace fibration. These will be respectively denoted $(M,\xi,\omega)$ and $(f,C,E)$ even though in our situation they refer to the manifold $(\widetilde M,\widetilde\xi,\widetilde\omega)$ and the good ace fibration $(\widetilde f,\widetilde C,E)$. This should not lead to confusion. The initial manifold is recovered in Section \ref{sec:end}.\\

\noindent Let $(M,\xi,\omega)$ be a $5$--dimensional closed orientable almost contact manifold.

\begin{definition}
An almost contact structure $(M,\xi,\omega)$ is called vertical contact with respect to an almost contact
fibration $(f,C)$ if the fibres of $f$ are contact submanifolds for $(\xi,\omega)$ away from the critical points.
\end{definition}

\noindent The main result of this section reads:

\begin{theorem} \label{thm:vert_defor}
Let $(M,\xi,\omega)$ be an almost contact manifold and $(f,C,E)$ an associated good ace fibration. Then there exists a homotopic deformation of the almost contact structure relative to $C$ and $E$ such that the almost
contact structure becomes vertical contact for $(f,C)$.
\end{theorem}

\noindent The proof of the theorem relies on the existence of an overtwisted disk in each fibre, such structure allows more flexibility in handling families of distributions. Hence, it will be essential for the argument to apply that the fibres of the good ace fibration $(f,C,E)$ are $3$--dimensional manifolds. In order to obtain a vertical contact fibration we need Eliashberg's classification result of overtwisted contact structures ~\cite{El}.\\

\noindent The almost contact structure obtained in Theorem \ref{thm:vert_defor} is constructed as a deformation of the vertical distributions $\{ \xi_z= \xi \cap Tf^{-1}(z) \}_{z\in \CP^1}$ relative to open neighborhoods of $C$ and $E$. A naive description of the argument consists of two parts. An overtwisted disk is first introduced in each fibre. This is the content of Subsection \ref{subsec:otdisks}. Then Eliashberg's result allows us to deform the family $\{ \xi_z \}_{z\in \CP^1}$ to a family of overtwisted contact structures. This corresponds to Subsection \ref{ssec:adapted}.\\

\noindent This argument cannot be readily applied because of two issues. On the one hand the almost contact fibration does not necessarily admit a section. In particular there is no naturally prescribed continuous family of overtwisted disks. This is solved using two local families to deal with each of the fibres. On the other hand the argument in \cite{El} deals with families of distributions over a fixed manifold. In our case the topology of the fibres changes if a curve in $f(C)$ is crossed. Therefore a refined version of Eliashberg's arguments is needed. It strongly uses the relative character of the result, both with respect to the parameter spaces and the open subsets of the manifold.\\

\noindent A technical step requires to define a suitable finite open cover of $\CP^1$ by 2--disks. In particular, the fibres over each 2--disk are diffeomorphic relative to a certain subset and there exists a continuous choice of overtwisted disks over each of these fibres. This cover is associated to $(f,C)$ and a cell decomposition of $\CP^1$. This will be explained.

\subsection{3--dimensional Overtwisted Structures.}
Our setup provides a fibration with a distribution on each fibre. Given such an almost contact fibration $f:M\longrightarrow\CP^1$, let $F_z$ denote the fibre over $z\in \CP^1$ and $(\xi_z, \omega_z)$ the
induced almost contact structure on $F_z$. Then the family $(F_z, \xi_z)$ can locally be viewed as a $2$--parametric family of $2$--distributions on a fixed fibre.\\

\noindent In the proof of Theorem~\ref{thm:vert_defor} we use a relative version of the following:

\begin{theorem} [Theorem 3.1.1 in \cite{El}]  \label{thm:eliash}
Let $M$ be a compact closed $3$--manifold and let $G$ be a closed subset
such that $M \setminus G$ is connected. Let $K$ be a compact space and $L$ a
closed subspace of $K$. Let  $\{\xi_t \}_{t \in K}$ be a family of cooriented $2$--plane
distributions  on $M$ which are contact everywhere for $t\in L$ and are contact near $G$ for
$t\in K$. Suppose there exists an embedded $2$--disk $\SD \subset M \setminus G$ such that
$\xi_t$ is contact near $\SD$ and $(\SD, \xi_t)$ is equivalent to the standard
overtwisted disk for all $t\in K$. Then there exists a family $\{ \xi_t' \}_{t\in
K}$ of contact structures of $M$ such that $\xi_t'$ coincides with $\xi_t$ near $G$
for $t\in K$ and coincides with $\xi_t$ everywhere for $t\in L$. Moreover $\xi_t'$
can be connected with $\xi_t$ by a homotopy through families of distributions that
is fixed in $(G \times K) \cup (M \times L)$.
\end{theorem}

\noindent In order to allow the case of a $3$--manifold with non--empty boundary we also need:

\begin{corollary} \label{coro:eliash}
Let $M$ be a compact $3$--manifold with boundary $\partial M$ and let $G$ be a closed 
subset of $M$  such that $M \setminus G $ is connected and $\partial M \subset G.$ 
Let $K$ be a compact space and $L$ a closed subspace of $K.$
Let  $\{ \xi_t \}_{t \in K}$ be a family of cooriented $2$--plane distributions on $M$ 
which are contact everywhere for $t\in L$ and are contact near $G$ for $t\in K$. Suppose
there exists an embedded $2$--disk $\SD \subset M\backslash G$ such that $\xi_t$ is contact
near $\SD$ and $(\SD, \xi_t)$ is equivalent to the standard overtwisted disk for all
$t\in K$. Then there exists a family $\{ \xi_t' \}_{t\in K}$ of contact structures
of $M$ such that $\xi_t'$ coincides with $\xi_t$ near $G$ for $t\in K$ and coincides
with $\xi_t$ everywhere for $t\in L$. Moreover $\xi_t'$ can be connected with
$\xi_t$ by a homotopy through families of distributions that is fixed in $(G\times K)
\cup (M \times L)$.
\end{corollary}
\begin{proof}[Outline]
The proof for the closed case uses a suitable triangulation $P$ of the $3$--manifold having a subtriangulation $Q$ containing $G$, for which the distributions are already contact structures. Then Eliashberg's argument is of a local nature, working with neighborhoods of the $0$, $1$, $2$ and $3$--skeleton of $P\backslash Q$ and assuring that no changes are made in a neighborhood of $Q$. Thus the method for a manifold $M$ with $\partial M\neq0$ is still valid since $P$ and $Q$ do exist in this case and only $Q$ contains the boundary.
\end{proof}

\noindent We locally treat an almost contact fibration as a $2$--parametric family of distributions over a fixed fibre, thus we may use a disk as a parameter space and the central fibre as the fixed manifold. It will be useful to be able to obtain a continuous family of distributions such that the distributions in a neighborhood of the central fibre become contact structures while the distributions near the boundary are fixed. Such a family is provided in the following

\begin{corollary} \label{coro:eliash2}
Consider the notation and hypotheses of Corollary \ref{coro:eliash} with $K$ diffeomorphic to a disk, $\SS=\partial K$ its boundary sphere and coordinates $(p,r)\in\SS\times[0,1]$. Let $\{\xi_t\}$ be a family of distributions parametrized by $\SS\times[0,1]$ which are contact near $G$ and $\SD$. Suppose that $\{\xi_t\}$ are contact distributions for $t\in\lambda\subset\SS\times[0,1]$. Given a homotopy $\xi^s_{(p,0)}$ of the distributions over $\SS\times\{0\}$, $s\in[0,1]$, there exists a homotopy $\{\xi^s_t\}$ relative to $G\times\SS\times[0,1]\cup M\times\lambda$ such that
$$\xi^0_t=\xi_t,\quad \xi^s_t=\xi^s_{(p,0)}\mbox{ for }t=(p,0)\mbox{ and }\xi^1_t=\xi_t\mbox{ for }t=(p,1).$$
\end{corollary}

\noindent The assumption that $K$ is a disk is not necessary. But we use Corollary~\ref{coro:eliash2} only in such a case. Its proof is left as an exercise for the reader.\\

\noindent We need at least one overtwisted disk over each fibre in order to apply Corollary \ref{coro:eliash}. The family should behave continuously. Let us provide such a family of disks.

\subsection{Families of overtwisted disks.}\label{subsec:otdisks}

There are two basics issues to be treated: the location of the disks and their overtwistedness. The second issue is simply guaranteed since once a disk with a contact neighborhood is placed in each fibre we can produce overtwisted disks using Lutz twists. In order to decide the location of the disks in each fibre we need to find a section of the good ace fibration.\\

\noindent Let $(f,C,E)$ be a good ace fibration. Denote by $U(C),U(E_i)$ open neighborhoods of the critical curves $C$ and the exceptional spheres $E_i\in E$. Consider $U(f)=U(C)\cup U(E_i)$ the union of these open neighborhoods, so in the complement of $U(f)$ the map $f$ becomes a submersion. Instead of finding a global section mapping away from $U(f)$, we shall construct two disjoint local sections that will provide at least one overtwisted disk in each fibre $F_z=f^{-1}(z)$. The distribution $\xi_z = \xi \cap TF_z$ is well--defined over $F_z \setminus U(f)$ and varies smoothly with the parameter $z\in \CP^1$. The global situation we achieve is described as follows:

\begin{proposition} \label{propdisc_family}
Let $(f,C,E)$ be a good ace fibration for $(M,\xi,\omega)$. Consider two open disks $\SB_0,\SB_{\infty} \subset
\CP^1$, containing $0$ and $\infty$ respectively such that the intersection $\SB_0\cap\SB_\infty$ is an open annulus, the complement of $\SB_0 \cap \SB_{\infty}$ consists of two disjoint disks and the curves $\partial \SB_0,\partial \SB_\infty$ are disjoint from the set of curves $f(C)$.\\

\noindent Then there exists a deformation $(F_z,\widetilde{\xi}_z)_{z\in \CP^1}$ of the family $(F_z, \xi_z)_{z\in
\CP^1}$ fixed at the intersection of the set $U(f)$ with each $F_z$ such that there are two disjoint families of embedded $2$--disks $\SD_z^i\subset F_z$,
with $z\in\SB_i$, for $i=0, 1$, not intersecting $U(f)$. The distribution $\widetilde{\xi}_z$ is a contact structure in a neighborhood of such families and $(\SD_z^i, \widetilde{\xi}_z)$ are equivalent to standard overtwisted disks.
\end{proposition}

\noindent The fact that $\widetilde{\xi}_z$ equals $\xi_z$ in the intersection of the set $U(f)$ with $F_z$ ensures that no deformation is performed near the critical curves nor the exceptional spheres. This is mainly a global statement, involving the whole of the fibres. In order to prove the result we study the local model of a tubular neighborhood of an exceptional divisor of the good ace fibration $(f,C,E)$.\\

A good ace fibration $(f,C,E)$ is obtained by surgery along the base locus $B$ of a certain good almost contact Lefschetz pencil. Let $K_i$ be a knot belonging to this base locus $B$. After the surgery procedure it is replaced by an exceptional contact divisor $E_i\in E$ contactomorphic to $(\sS^3,\xi_{st})$. As explained in Section \ref{sec:blowup} the restriction of the fibration $f$ to $E_i$ is the Hopf fibration. Since the distribution $\xi$ is locally a contact structure the tubular neighborhood theorem provides a chart
\begin{equation}
\Psi: U  \longrightarrow \sS^3 \times \D^2(\varepsilon),\quad\Psi^* \xi_{st} = \xi \label{eq:chart0}
\end{equation}
where $\xi_{st}= \ker \{ \alpha_{\sS^3} + r^2 d\theta \}$, $\varepsilon\in\R^+$ and $\Psi(E_i)=\sS^3\times\{0\}$. Suppose $\varepsilon=1$ in order to ease notation.\\

\noindent The induced map $f_U$ defined as
$$\xymatrix{
& \sS^3\times \D^2\ar@{->}[rd]^{f_U}\ar@{->}[r]^{\qquad\Psi^{-1}}
&U\ar@{->}[d]^{f}
\\
& & \CP^1
}$$
can be expressed as $f_U(x,r,\theta)=h(x)$ for $x\in\sS^3$. The fibres $F_z=f^{-1}(z)\cap U$ are contact submanifolds of $(\sS^3\times\D^2,\xi_{std})$. The induced contact structure $\xi_v(z)$ on $F_z$ depends on the point $z\in\CP^1$. These fibres are contactomorphic to $(\sS^1\times \D^2, \xi_v= \ker (d\beta+r^2d\theta))$ for each $z\in \CP^1$. Note that the variable $\beta\in \sS^1$ parametrizing each Hopf fibre is not global since the fibration is not trivial. The differential $d\beta$ is globally well--defined since it is dual to the vector field generating the associated $\sS^1$--action. The standard contact structure in $\sS^3\times\D^2$ can be expressed as the direct sum of distributions
\begin{equation}
\xi_{st}(x,r,\theta) = \xi_v(h(x))  \oplus H(x,r,\theta), \label{eq:conn_tubular}
\end{equation}
where $\xi_v$ is the standard contact structure in $\sS^1 \times \D^2$, the vertical direction, and $H$ is a horizontal complement associated to the fibration of $\sS^3 \times \D^2$ over $\CP^1$.\\

\noindent Topologically, the $4$--distribution $\xi_{st}$ is expressed as a direct sum of two distributions of $2$--planes. Since the $2$--form $\omega$ providing the almost contact structure is given and so is $\xi$, we may interpret $(\sS^3 \times\D^2, \xi_v(z))$ as a non--trivial family of contact
structures parametrized by the base ${z\in \CP^1 }$. We have detailed the topology and contact structure of the local model of the good ace fibration along an exceptional sphere $E_i$. A neighborhood of this exceptional sphere is a piece of the fibration and the knots are the intersection of the fibres of the almost contact pencil with it.\\

The local model described above allows us to prove the following

\begin{lemma} \label{lem:local_family}
Let $z\in\CP^1$ be a coordinate, $(\sS^3\times\D^2,\xi_v(z))$ a $\CP^1$--family of contact structures on $\sS^3\times\D^2$ and $f_U:\sS^3\times\D^2\longrightarrow\CP^1$ the map described above. Consider two open disks $\SB_0,\SB_{\infty} \subset
\CP^1$, containing $0$ and $\infty$ respectively such that the intersection $\SB_0\cap\SB_\infty$ is an open annulus and the complement of $\SB_0 \cap \SB_{\infty}$ consists of two disjoint disks.\\

\noindent There exists a homotopy $\xi^s_v(z)$ of $\CP^1$--families of plane fields, $s\in[0,1]$, such that\\
\begin{itemize}
\item[-] $\xi^0_v(z)=\xi_v(z)$, $\forall z\in\CP^1$.\\
\item[-] Near the boundary of $f_U^{-1}(z)\cong\sS^1\times\D^2$ and $\forall(z,s)\in\CP^1\times[0,1]$, $\xi^s_v(z)=\xi_v(z)$.\\
\item[-] For any $z\in\CP^1$, the distribution $\xi^1_v(z)$ is an overtwisted contact structure on $f_U^{-1}(z)$ containing two disjoint Lutz tubes $L^0_z$ and $L^\infty_z$ away from $\sS^3\times\{0\}$.\\
\item[-] There exist a smooth family of embedded overtwisted $2$--disks $\SD_z^0$ in $L_z^0$ for $z\in\SB_0$ and $\SD_z^\infty$ in $L_z^\infty$ for $z\in\SB_\infty$.\\
\end{itemize}
Both $\SB_0\backslash\partial\SB_0,\SB_\infty\backslash\partial\SB_\infty$ can be thought as neighborhoods of the upper and lower semi--spheres.
\end{lemma}

\begin{proof}
Let $h: \sS^3 \longrightarrow \CP^1$ be the Hopf fibration, extend the fibration to $h: \sS^3
\times \D^2 \longrightarrow \CP^1$ by projection onto the first factor. The idea is to use the exceptional divisor to create a couple of sections along $\SB_0$ and $\SB_\infty$. On the one hand, the exceptional divisor has a contact structure and we would rather not perturb around a small neighborhood of it. On the other hand the exceptional divisor is not $\CP^1$ but $\sS^3$. Hence a global section cannot exist. We use two copies of the exceptional divisor away from $\sS^3\times\{0\}\subset \sS^3\times\D^2$ and we cover the base $\CP^1$ with the two disks $\SB_0$, $\SB_\infty$.\\

\noindent Let $q_0=(1/2,0)$, $q_{\infty}=(0,1/2) \in\D^2$ be two fixed points and consider the two $3$--spheres
$$\sS^3_0= \sS^3 \times \{ q_0 \},\qquad \sS^3_{\infty}= \sS^3 \times \{ q_{\infty} \}.$$

\noindent The fibre of the restriction of the fibration $(\sS^3 \times\D^2, \xi_v(z))\longrightarrow\CP^1$ to the
submanifold $\sS^3_0$ (resp. $\sS^3_{\infty}$) is a transverse knot $K_0^z$ (resp. $K_{\infty}^z$). We will now insert two families of overtwisted disks.\\

\noindent Apply a full Lutz twist in a small neighborhood of each of those knots $K_0^z \in h^{-1}(z)$ parametrically on $z\in\CP^1$. This produces a $3$--dimensional full Lutz twist on each fibre. See \cite{Lu1},\cite{Ge}. This yields an $\sS_0^3$--family of overtwisted disks parametrized as $\{\SD^0_{t} \}_{t\in \sS_0^3}$, thus we obtain a $\sS^1$--family of overtwisted disks at each fibre. Note that the dependency of this parametric family of full Lutz twists on the point $z\in\CP^1$ is well--behaved. Indeed, let $i_z:K^0_z\longrightarrow\sS^3_0$ be the injection and consider coordinates $(\rho,\varphi)$ in the normal bundle of this embedding. In a small neighborhood of the zero section, the contact structure reads
$$\xi_v(z)=\ker\{i_z^*\alpha_{\sS^3}+\rho^2d\varphi\}.$$ The pair of functions $(h_1,h_2)$ used in Section 4.3 \cite{Ge} to perform the full Lutz twist can be made $\rho$--dependent. Thus the resulting contact structure has the form
$$\xi^1_v(z)=\ker\{h_1(\rho)\cdot i_z^*\alpha_{\sS^3}+h_2(\rho)\cdot\rho^2d\varphi\}.$$
\noindent This clarifies the dependency of the construction with respect to $z\in\CP^1$.\\

\noindent Perform the same twist procedure for the family of knots $K_{\infty}^z \in
h^{-1}(z)$ to obtain another family of overtwisted disks $\{ \SD_{t}^\infty \}_{t \in
\sS_{\infty}^3}$. The two families of disks can indeed be assumed disjoint by letting the radius in which we perform the full Lutz twists be small enough. The support of the pair of full Lutz twists can be chosen not to intersect the exceptional divisor and be contained in the interior of $\sS^3\times\D^2$. This construction provides the homotopy in the statement of the Lemma.  See Figure \ref{fig:exceptional}.\\

\begin{figure}[ht]
\includegraphics[scale=0.4]{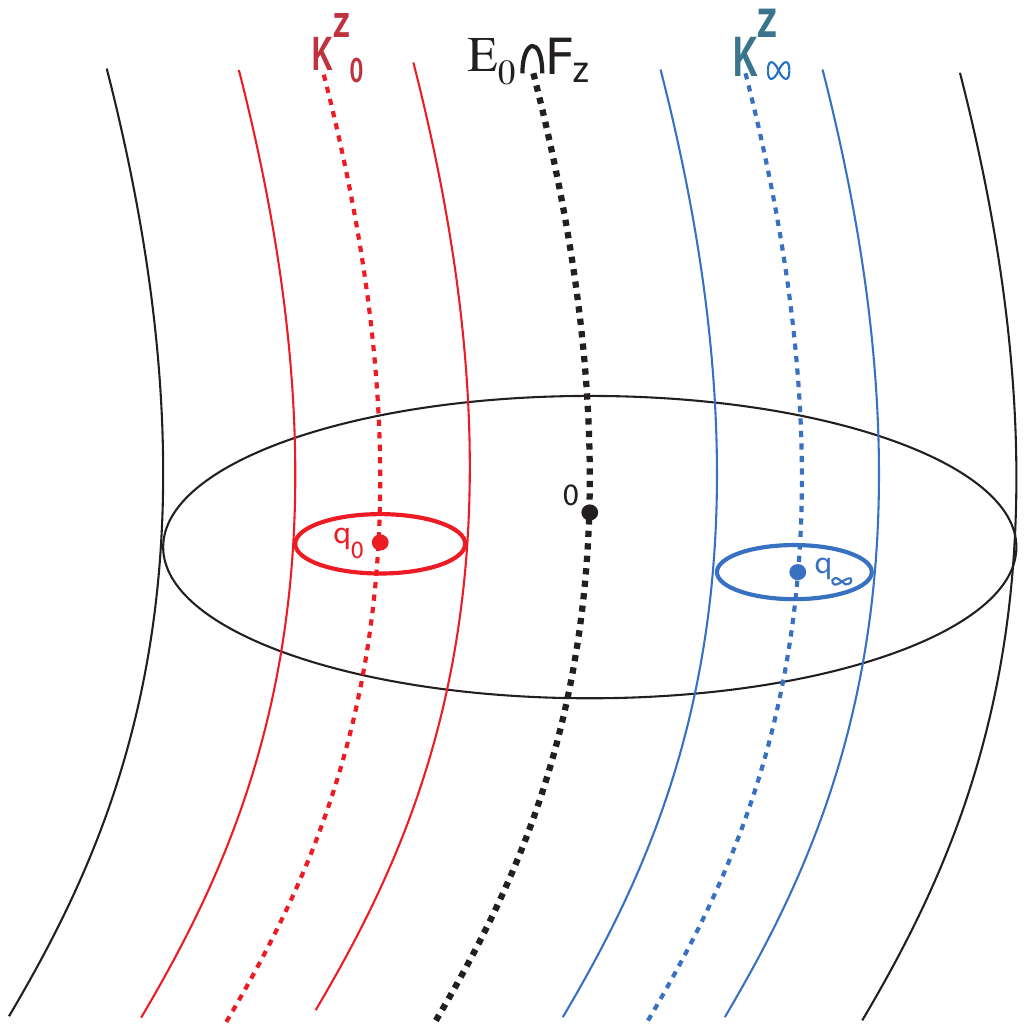}
\caption{The neighborhood of the exceptional divisor intersected with a fibre $F_z$. The cylinder on the left (with axis $K^z_0$) is the support of the full Lutz twist around the knot $K_0^z\cong\sS^1 \times \{ q_0 \}$ and the cylinder on the right (with axis $K_\infty^z$) corresponds to the support of the full Lutz twist around the knot $K_\infty^z\cong\sS^1 \times \{ q_{\infty} \}$. }\label{fig:exceptional}
\end{figure}

\noindent We need the base $\CP^1$ to be the parameter space instead of the $3$--spheres $\sS^3_0$ and $\sS^3_\infty$. Restricted to $\SB_0$ or $\SB_\infty$ the Hopf fibration becomes trivial and therefore there
exist two sections $s_0: \SB_0 \longrightarrow \sS^3 \cong \sS^3_0$ and $s_{\infty}: \SB_{\infty} \longrightarrow \sS^3 \cong \sS_{\infty}^3$. The required families are defined as
$$ \{ \SD_z^0 \} = \{ \SD^0_{s_0(z)} \},   z\in \SB_0, $$
$$ \{ \SD_z^{\infty} \} = \{ \SD^\infty_{s_{\infty}(z)} \},  z\in \SB_{\infty}.$$
Note that the two families of overtwisted disks are disjoint since the two families of Lutz twists
are. Further, there exists a small neighborhood of the exceptional divisor $\sS^3\times\{0\}$ where no deformation is performed. The statement of the Lemma follows.
\end{proof}

\noindent The global construction can be simply achieved:\\

\noindent{\em Proof of Proposition \ref{propdisc_family}.}
Apply Lemma \ref{lem:local_family} to a neighborhood of one exceptional sphere $E_0\in E=\{E_0,E_1,\ldots,E_s\}$. The families of overtwisted disks do not meet $C$ or any $E_j$. Indeed, the two families are arbitrarily close to $E_0$ and the exceptional divisors are pairwise disjoint and none of them intersect the critical curves $C$. Thus, maybe after shrinking the neighborhood $U(E_0)$ in the construction, the families are located away from $U(f)$.\hfill $\Box$\\

\noindent Thus we obtain the families of overtwisted disks required to apply Theorem \ref{thm:eliash}. The vertical deformation is described using a suitable cell decomposition of the base $\CP^1$. The vertical contact condition is ensured progressively above the 0--cells, the 1--cells and the 2--cells.
\subsection{Adapted families}\label{ssec:adapted}
Let $(f,C)$ be an almost contact fibration. A finite set of oriented immersed
connected curves $T$ in $\CP^1$ will be called an adapted family for $(f,C)$ if
it satisfies the following properties:
\begin{itemize}
\item[-] The image of the set of critical values $f(C)$ is part of $T$.
\item[-] Given any element $c\in T$, there exists another element of $c'\in T$ having a non--empty intersection\footnote{In case $c$ has a self--intersection, then $c'=c$ is allowed.} with $c$. Any two elements of $T$ intersect transversally.
\item[-] There exists no triple intersection point between the curves of $T$.
\item[-] The complement $\CP^1\setminus|T|$ is a union of open disks.
\end{itemize}
$|T|\subset\CP^1$ denotes the underlying set of points of the elements of $T$. The elements of an adapted family $T$ that are not in the image of a component of $C$ are referred to as fake components. Let $N\in\N$ be fixed. The insertion of fake curves proves the existence of an adapted family with $\mbox{diam}_{g_0}(\CP^1\setminus|T|)\leq 1/N$, $g_0$ the standard round metric.\\

\begin{figure}[ht]
\includegraphics[scale=0.6]{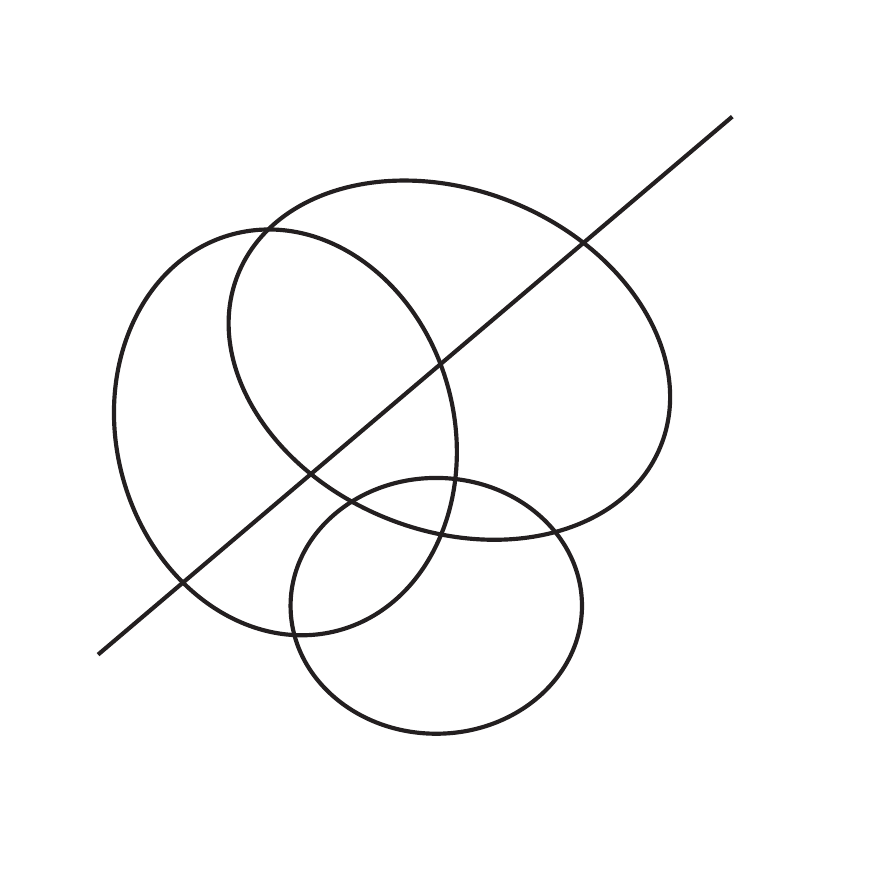}
\caption{Part of an adapted family $T$. The associated
subdivision consists of certain 2--cells with their boundaries being a union of parts of various
elements in the family $T$.}\label{fig:tri0}
\end{figure}

\noindent There is a cell decomposition of $\CP^1$ associated to an adapted family, the $1$--skeleton being $|T|$. See Figure~\ref{fig:tri0}. In order to conclude Theorem \ref{thm:vert_defor} we shall first deform in a neighborhood of each vertex relative to the boundary, proceed with a neighborhood of the $1$--cells and finally obtain the vertical contact condition in the $2$--cells. To be precise in the description of the procedure, we introduce some notation. This is not strictly necessary but it provides the adequate pieces in the framework to apply Eliashberg's result.\\

\noindent Let $L_j\in T$ be a curve, $U(L_j)$ be an open tubular neighborhood and denote
$$\partial \overline{U(L_j)}=L_j^0\cup L_j^1.$$
Suppose that $\bigcup_{j \in J}|L^i_j|$ is isotopic to $|T|$ for both $i=0,1$; this can be achieved by taking a small enough neighborhood of each $L_j$. See Figure \ref{fig:tri3}. We use $V(L_j)$ to denote a slightly larger tubular neighborhood satisfying this same condition. Fix an intersection point $p$ of two elements $L_j,L_k\in T$. Denote by $\SA_{p}$ the connected component of the intersection of $U(L_j)\cap U(L_k)$ containing $p$. Similarly, let $\SV \SA_{p}$ be the connected component of the intersection of $V(L_j)\cap V(L_k)$ that contains $p$, and denote $\SA\SA_p=\SV \SA_{p}\backslash\SA_p$.\\

Consider a small neighborhood $U(T)$ of $|T|$. The open connected components of
$$U(T)\backslash\{\cup\SA_p\}$$
are homeomorphic to rectangles $\SB_i$, $p$ being treated as an index over the intersection points. A suitable indexing for $i$ is also assumed. The third class of pieces constitute the interior of the complement in $\CP^1$ of the open set formed by the union of the sets $\SA_{p}$ and $\SB_i$. Its connected components are denoted $\SC_l$. Thus, neighborhoods of the $0$--cells, $1$--cells and $2$--cells are labeled $\SA_p$, $\SB_i$ and $\SC_l$ respectively. See Figure \ref{fig:tri3}.\\

\begin{figure}[ht]
\includegraphics[scale=0.5]{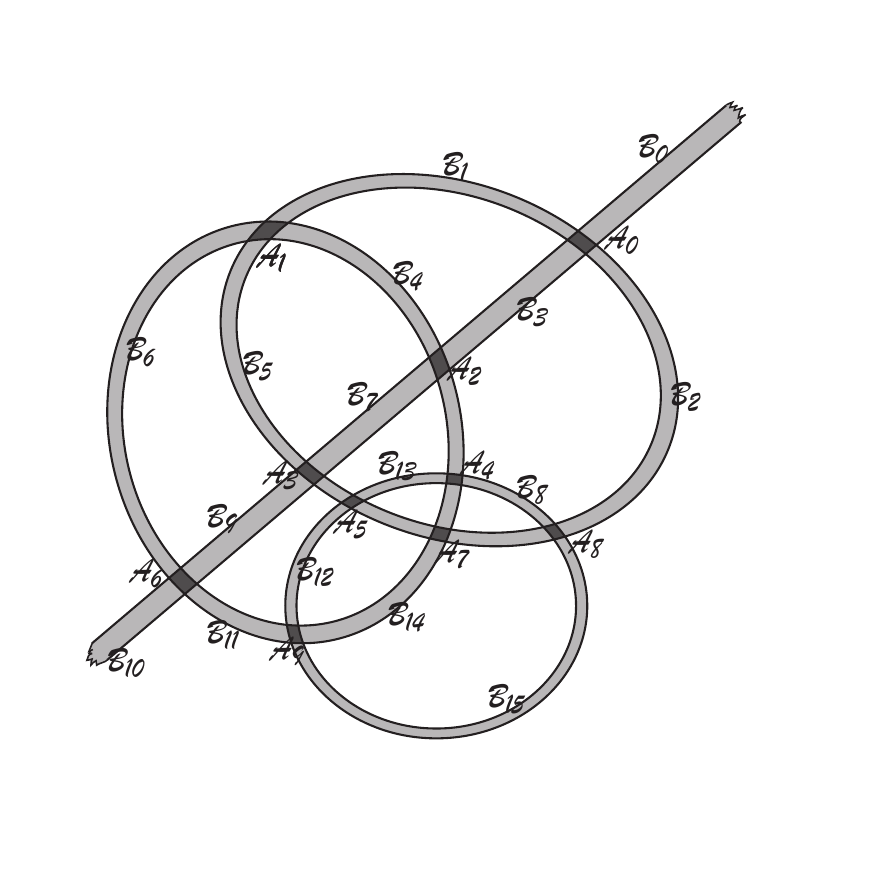}
\caption{The sets $A_p$ and $B_i$ associated to
the subdivision of the figure \ref{fig:tri0}. The sets $A_p$ are drawn in darker grey.}\label{fig:tri3}
\end{figure}

\noindent Finally, we define the sets $\SB\SB_i$. Let $\SB_i$ connect a couple of open sets\footnote{Both sets may be the same for the self--intersecting curves.} of the form $\SA_{p}$. There exists a curve $L_{\SB_i}$ contained in $\SB_i$ which is a part of a curve $L_i\in T$. $L_{\SB_i}$ is part of a $1$--cell in the decomposition associated to the adapted family $T$. Let $L_{\SB_i}^0$ and $L_{\SB_i}^1$ denote the two boundary components of $\overline{\SB}_i$ which are part of the curves $L_i^0$ and $L_i^1$ defined above. Then we declare $\SB \SB_i^0$ (resp. $\SB \SB_i^1$) to be the connected component of $V(L_i)\backslash\SB_j$ containing the boundary curve $L_i^0$ (resp. $L_i^1$). Their union $\SB \SB_i^0\cup\SB \SB_i^1$ will be denoted $\SB \SB_i$. See Figures
\ref{fig:tri4} and \ref{fig:tri7}.\\

\begin{figure}[ht]
\includegraphics[scale=0.4]{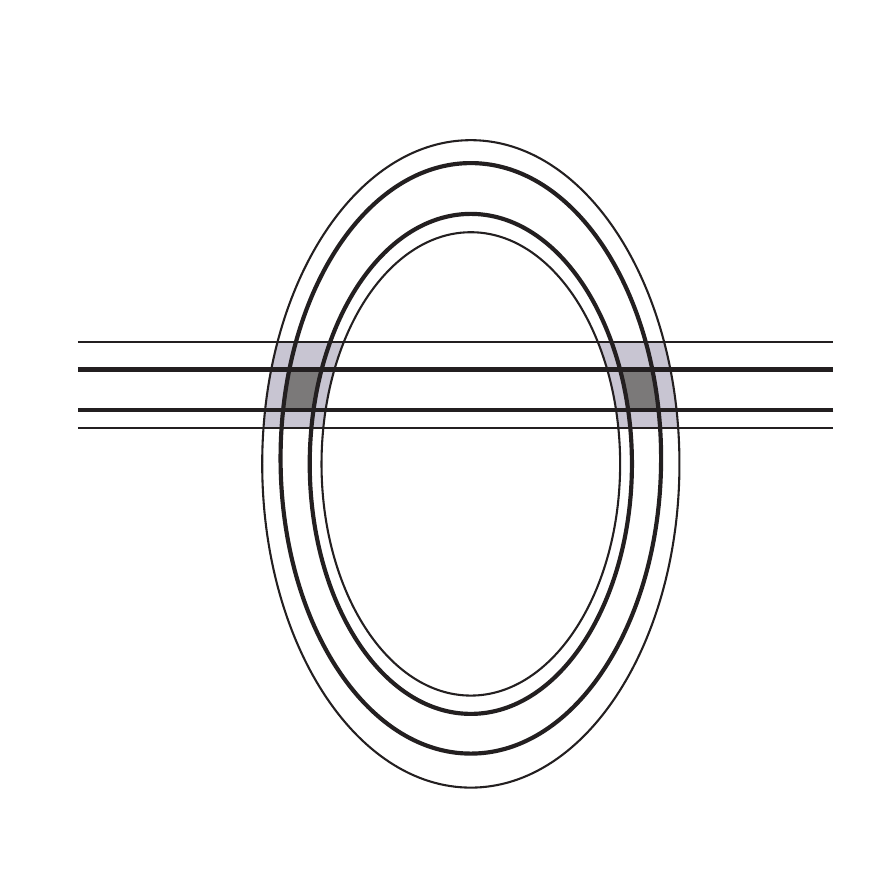}
\caption{Example of two components $\SV \SA_{p}$ and $\SV \SA_{q}$ in light gray, containing $\SA_{p}$ and $\SA_{q}$, in dark gray.}\label{fig:tri4}
\end{figure}

\begin{figure}[ht]
\includegraphics[scale=0.45]{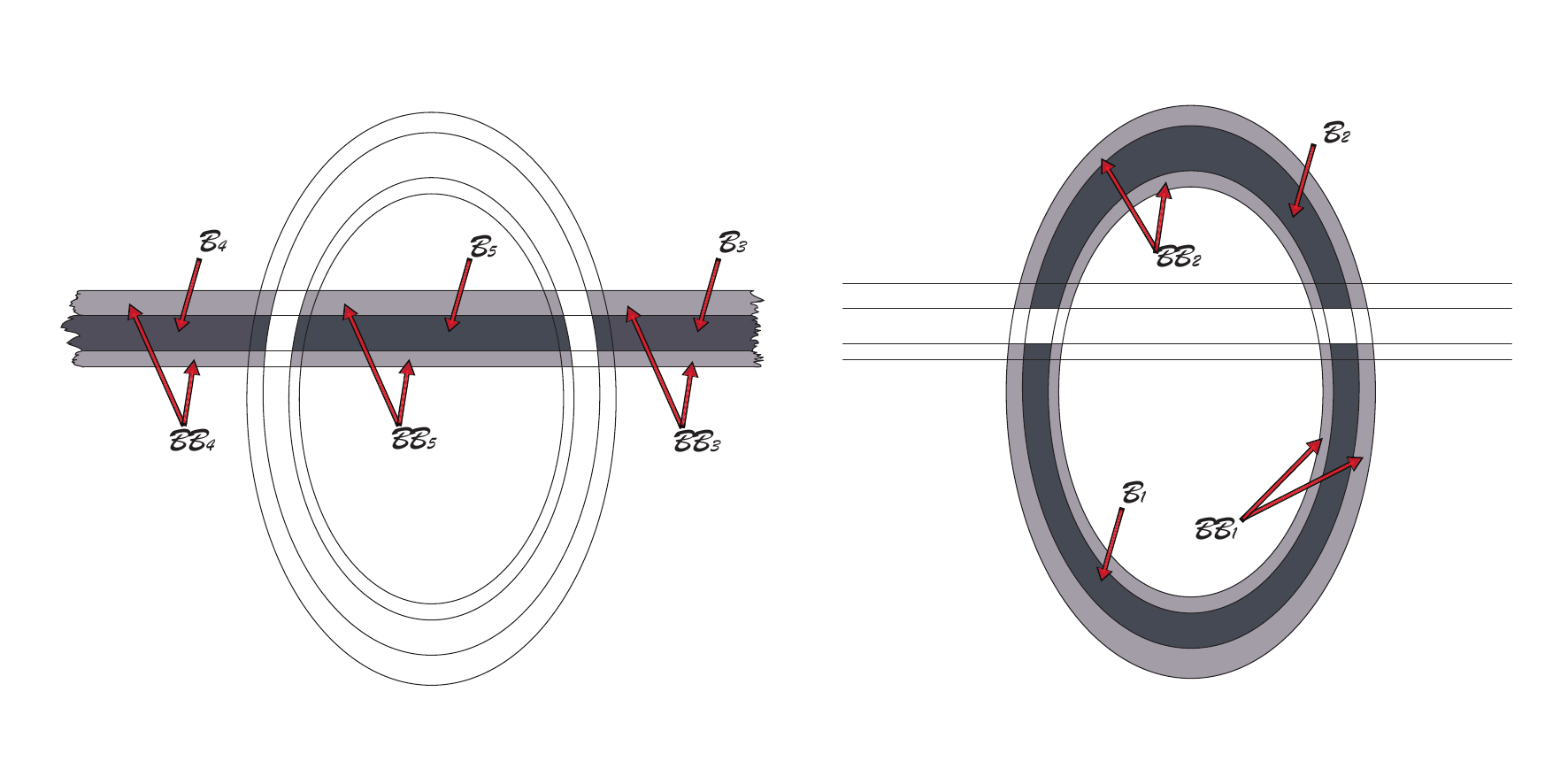}
\caption{Example of the sets $\SB_i$ and $\SB \SB_i$ for the subdivision of Figure
\ref{fig:tri4}.}\label{fig:tri7}
\end{figure}

\subsection{The vertical construction.} In this subsection we prove Theorem \ref{thm:vert_defor}. The following lemma is a simple exercise in differential topology and can be considered as a particular case of Ehresmann's fibration theorem. It will be used in the proof of Theorem \ref{thm:vert_defor}. We include it for completeness.
\begin{lemma} \label{lem:trivializing}
Let $f:E\longrightarrow \D^2$ be a locally trivial smooth fibration over the unit disk with compact
fibres $E_z$, $z\in\D^2$. Decompose $\partial E$ along its corners as $\partial E=f^{-1}(\partial\D^2)\cup \partial_hE$ and suppose that $\partial_hE$ is a smooth closed boundary. Suppose also that there is a collar neighborhood $N$ of $\partial_h E$ and a closed submanifold $S$ such that restricting $f$ to $S$ and $N$ induces locally trivial fibrations. Let $S_0,N_0$ be their fibres over $0\in\D^2$.\\

\noindent Then there exists a diffeomorphism $g: E \longrightarrow E_0 \times\D^2$ making the following diagram commute
$$\xymatrix{
E\ar@{->}[r]^{g}\ar@{->}[d]_\pi & E_0\times\D^2\ar@{->}[d]^{\pi_0}\\
\D^2\ar@{=}[r] & \D^2
}$$
such that $g(N)=N_0 \times\D^2$ and $g(S)=S_0 \times\D^2$.
\end{lemma}
\begin{proof}
Let $g$ be Riemannian metric in $E$ such that $(TE_z)^{\perp g} \subset TS$ and $(TE_z)^{\perp g} \subset T(\partial_hE)$, for the points $z$ where the condition can be satisfied. Let $X=\partial_r$ be the radial vector field in $\D^2\backslash\{ 0 \}$ and construct the connection $H_\pi$ associated to the Riemannian fibration:
$$H_{\pi}(e)=(T_e F_{\pi(e)})^{\perp g}.$$
The condition imposed on the Riemannian metric implies that $\partial_hE$ and $S$ are
tangent to the horizontal connection $H_{\pi}$. Let $\widetilde{X}$ be a lift of $X$ through $H_{\pi}$ and $\phi_t(e)$ the flow of this vector field. Define
\begin{eqnarray*}
E & \stackrel{g}{\longrightarrow} & E_0 \times \D^2 \\
e & \longmapsto & (\phi_{(-||\pi(e)||)}(e), \pi(e)).
\end{eqnarray*}
This map satisfies the required properties.
\end{proof}

\noindent {\bf Proof of Theorem \ref{thm:vert_defor}.}
Let $(f,C,E)$ be a good ace fibration and $T$ an adapted family to $(f,C,E)$. Note that a horizontal complement $H$ is defined away from $U(C)$ and provides the splitting specified in (\ref{eq:conn_tubular}). Proposition \ref{propdisc_family} and choose $\SB_0$ and $\SB_\infty$ in the statement such that $\partial\SB_0$ and $\partial\SB_\infty$ are both contained in two different $2$--cells $\SC_0$ and $\SC_\infty$. Lemma \ref{lem:split} implies that this procedure preserves the homotopy class of $(M,\xi,\omega)$.\\

\noindent In order to establish Theorem \ref{thm:vert_defor} we need to perform a deformation which is fixed in a neighborhood of $U(C)$ and leaves the distribution $H$ unchanged, i.e. it should be a strictly vertical deformation.\\

\noindent{\it Deformation at the $0$--cells}: Let $p$ be a vertex with neighborhood $\SA_p$ and
$$\SF=f^{-1}(\SV\SA_p)\setminus(f^{-1}(\SV\SA_p)\cap U(C)).$$
We can assume that $\SV\SA_p$ is small enough and choose a neighborhood $U(C)$ such that the map $f$ restricts to a trivial fibration on $\SF$ and induces a fibration on $\partial\SF$. Consider a trivialization of the former fibration over $\SV\SA_p$. The manifolds with boundary $\SF_z=f^{-1}(z)\backslash (f^{-1}(z)\cap U(C))$ are all diffeomorphic. Let $N_z$ be a collar neighborhood of $\partial\SF_z$ in which the distribution is contact. Given an exceptional divisor $E_i\in E$ denote by $U(E_i)_z$ the intersection of $U(E_i)$ with the fibre $\SF_z$. Applying the trivializing diffeomorphism provided in Lemma \ref{lem:trivializing}, we may assume $\SF_z \times \SV\SA_p \cong\SF$, $U(E_i)_z\times \SV\SA_p \cong U(E_i)$ and $N_z\times \SV\SA_p \cong N$.\\

\noindent Thus we have a manifold with boundary $\SF$ with a family of distributions $\xi_z$ parametrized by the topological disk $\SV\SA_p$ containing $K=\SA_p$. Also a good set $G$ of submanifolds that are already contact for any contact fibre over $\SV\SA_p$. The good set $G$ consists of the union of $N$, $U(E_j)$ and a neighborhood of one of the two overtwisted disks\footnote{These disks are trivialized along with $N$ using Lemma \ref{lem:trivializing}.}. Let us say $p\in\SB_0$ and we choose a neighborhood of $\SD^\infty$. A neighborhood of this set will not be perturbed. The remaining disk $\SD^0$ is contactomorphic to the standard overtwisted disk for each element of the family of distributions. This set--up satisfies the hypotheses of Corollary \ref{coro:eliash}. It should be applied to a smaller parameter space $K$ and then Corollary \ref{coro:eliash2} is used with $\lambda=\emptyset$ to obtain a deformation relative to the boundary. Since we are able to obtain a deformation relative to the boundary we may perform the deformation at each neighborhood of the $0$--cells and extend trivially to the complement of $\SV\SA_p$ in $\CP^1$.\\

\noindent{\it Deformation at the $1$--cells}: Almost the same strategy applied to the $0$--cells applies, although we should not undo the deformation in a neighborhood of the $0$--cells. Corollaries \ref{coro:eliash} and \ref{coro:eliash2} allow us to perform deformations relative to a subfamily, so in this case $\lambda$ will be non--empty. See Figure \ref{fig:Bi}.\\
\begin{figure}[ht]
\includegraphics[scale=0.6]{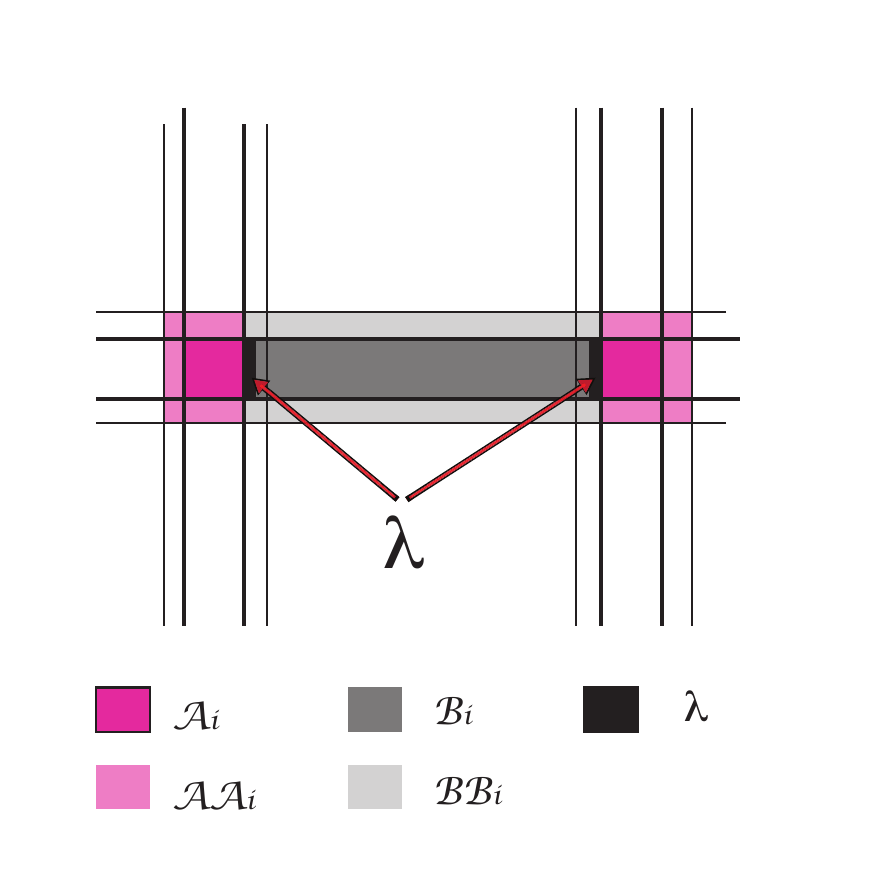}
\caption{The distributions set $\xi_z \subset\SB \SB_i$ with $z\in\lambda$ are already contact distributions.} \label{fig:Bi}
\end{figure}

\noindent{\it Deformation at the $2$--cells}: In this situation Theorem \ref{thm:eliash} also applies after a suitable trivialization of the smooth fibration provided by Lemma \ref{lem:trivializing}. Note that in this case the fibres do not have the boundary contribution of $U(C)$ since its image is not contained in the $2$--cells. The set $L$ is a small tubular neighborhood of the boundary of the $2$--cells. Except at $\SC_0$ and $\SC_\infty$, we may use any of the two families of overtwisted disks to apply the result. Let it be $\SD_z^0$. In the remaining family the distributions are contact and so we include the disks in the set $G$, that also contains $N$ and $U(E_i)$. At $\SC_0$ we use the family $\SD_z^0$, since it is the only one well--defined over the whole set. Proceed analogously at $\SC_\infty$. Note that this argument is possible because the deformation is relative to the boundary. Then Theorem \ref{thm:eliash} applies to the $2$--cells and we extend trivially the deformation. We obtain a vertical contact distribution $(F_z,\widetilde{\xi}_z)$ away from $U(C)$.\\

\noindent In order to conclude the statement of the Theorem, consider the direct sum $\widetilde{\xi}_z\oplus H$ to include the critical set, which has not been deformed. This is the required vertical contact structure. Notice that this construction preserves the almost contact class of the distribution since it is performed homotopically only in the vertical direction. Hence Lemma \ref{lem:split} provides a homotopy on the complement of $U(C)$ relative to the boundary. This yields a homotopy over the manifold $M$.\hfill $\Box$\\
\section{Horizontal Deformation I} \label{sec:skeleton}
Consider an almost contact distribution $(M,\xi,\omega)$ and a good ace fibration $(f,C,E)$ with associated adapted family $T$. Theorem \ref{thm:vert_defor} deforms $\xi$ to a vertical contact structure with respect to $(f,C,E)$. To obtain a honest contact structure the distribution has to be suitably changed in the horizontal direction. As in the previous section, this is achieved in three stages. The content of this Section consists of the first two of these: deformation in the pre--image of a neighborhood of the $0$-- and the $1$--cells of the adapted family $T$. The main result of this Section is the following theorem.

\begin{theorem}
\label{prop:pencil_skeleton}
Let $(M,\xi,\omega)$ be a vertical contact structure with respect to a good ace fibration $(f,C,E)$ and $T$ an adapted family. Then there exists a homotopic deformation $(\xi',\omega')$ of $(\xi,\omega)$ relative to $C$ and $E$ such that $(f,C,E)$ is a good ace fibration for $(\xi',\omega')$, $(\xi',\omega')$ is a vertical contact almost contact structure and $\xi'$ is a contact structure in the pre--image of a neighborhood of $|T|$.
\end{theorem}
\noindent The vertical distribution is fixed along the deformation. In this sense the deformation in the statement is horizontal. The fibration $(f,C,E)$ will not be deformed to prove this fact, just the almost contact structure.\\

\noindent Theorem \ref{prop:pencil_skeleton} follows Proposition \ref{cor:defo_curve} and Lemma \ref{lem:split}. To prove the statement we trivialize the vertical contact fibration over a neighborhood of the $0$--cells. Then the deformation is performed using an explicit local model. The deformation in a neighborhood of the $0$--cells is the content of Proposition \ref{cor:def0}. Then we proceed with the pre--image of a neighborhood of the $1$--cells. This is Proposition \ref{cor:defo_curve}. The same local model is used in both deformations.

\subsection{Local model} The following lemma is used to prove Proposition \ref{cor:def0} and Proposition \ref{cor:defo_curve}. It is a version of results in Section 2.3 of \cite{El} concerning deformations of a family of distributions near the $1$ and $2$--skeleta of a $3$--manifold. The connectedness condition is stated there as the vanishing of a relative fundamental group.

\begin{lemma} \label{lem:contactness_condition}
Let $\left(F, \xi_t\right)$ be a family of contact structures over a compact 3--manifold
$F$ parametrized by $(s,t)\in [-\varepsilon,
\varepsilon] \times [0,1]$ with $\xi_t$ is constant along the $s$--lines and $\alpha_t$ associated contact forms. Consider the projection
$$\xymatrix{F\times [-\varepsilon,
\varepsilon] \times [0,1]\ar@[->][r]^{\qquad\pi}& F\times[0,1],}$$
and the distribution $\xi$ on $F\times[-\varepsilon,
\varepsilon] \times [0,1]$ defined globally by the kernel of the form 
$$\alpha_H(p,s,t) = \alpha_t + H(p,s,t) dt,\qquad H\in C^\infty(F \times [-\varepsilon,
\varepsilon] \times [0,1]).$$
Suppose that $|H(p,s,t)|\leq c \cdot |s|$ and assume that the $1$--form $\alpha_H$ is a contact form in a compact set $G$ such that the intersection of $G$ with any segment $\{p\}\times[-\varepsilon,\varepsilon]\times\{t\}$ is either connected or empty.\\

\noindent Then, there is a small perturbation $\widetilde{H}$ of $H$ relative to $G$ such that $\alpha_{\widetilde{H}}$ defines a contact structure. In precise terms, $|\widetilde{H}-H|\leq 3c\varepsilon$ and $\widetilde{H}|_G=H|_G$.
\end{lemma}
\begin{proof}
Let us compute the contact condition on $\alpha=\alpha_H$. 
$$d\alpha=  d\alpha_t + dt\wedge\partial_t\alpha_t+ dH \wedge dt\quad\Longrightarrow\quad (d\alpha)^n = \left(d\alpha_t\right)^n + \left(d\alpha_t\right)^{n-1} \wedge dH \wedge dt.$$
Therefore, the contact condition is described as
$$(d\alpha)^n \wedge \alpha = \left(d\alpha_t\right)^{n-1} \wedge \alpha_t  \wedge(
\partial_s H\cdot ds \wedge dt).$$
Thus, the $1$--form $\alpha$ is a contact form if and only if $\partial_s H >0$.\\

\noindent Given $(p,t)\in F\times[0,1]$, $\pi^{-1}(p,t)$ is a $4$--parametric family of $1$--dimensional manifolds. The connectedness of $\pi^{-1}(p,t)\cap G$ and the compactness of $G$ assure that it is possible to perturb $H$ to an $\widetilde{H}$ relative to $G$ and satisfying the contact condition. Indeed, the connectedness condition allows us to perturb the function $H$ on at least one end of the curves in $F\times[-\varepsilon,\varepsilon]\times[0,1]$ and obtain a function $\widetilde{H}$ with $\partial_s\widetilde H>0$.
\end{proof}
\subsection{Contact connections}\label{ssec:contcon} The previous Lemma \ref{lem:contactness_condition} can be used if the contact form has the expression as in the hypotheses of the statement. This is achieved with the choice of an appropriate trivialization obtained by parallel transport. It is convenient to review the notions introduced in \cite{Le}.

\begin{definition} A contact fibration is a smooth fibration $\pi:M\longrightarrow B$ with a co--oriented codimension--1 distribution $\xi\subset TM$ such that the intersection of $\xi$ with any fibre induces a contact structure on that fibre.
\end{definition}

\noindent Consider a contact fibration $(\pi,\xi)$, a $1$--form $\alpha$ such that $\xi=\ker\alpha$ and the vertical bundle $\ker\pi$. A contact fibration has an associated contact connection $H_\xi$. It is defined as the orthogonal of the symplectic subbundle $(\ker\pi\cap\xi,d\alpha|_{\ker\pi\cap\xi})$ in $\xi$ with respect to $d\alpha|_\xi$. Note that the contact connection only depends on the contact structure and not on the choice of the contact form.

\begin{lemma}\label{lem:paralleltrans}
Let $(\pi,\xi)$ be a contact fibration. The parallel transport with respect to a contact connection is by contactomorphisms.
\end{lemma}

\noindent This is a simple computation. See \cite{Le}, \cite{Pr2}. A vertical contact almost contact structure $(M,\xi,\omega)$ with respect to a good ace fibration $(f,C,E)$ is in particular a contact fibration away from the critical locus $C$. Suppose that $\xi=\ker\alpha$ and let $\xi_v=\ker\alpha_v$ be the vertical distribution. The symplectic structure $\omega$ and $d\alpha|_\xi$ both provide a horizontal complement for the vertical distribution $\xi_v$ in $\xi$. These are defined as the annihilators  of the vertical bundles with respect to the $2$--forms $\omega$ and $d\alpha|_\xi$. Let us denote the first one by $H_\omega$ and note that the second one is the contact connection $H_\xi$ introduced above. The distribution $H_\xi$ is not necessarily symplectic for $\omega$. Consider a symplectic structure $\omega_\xi$ for $H_\xi$ coinciding with the symplectic structure $d\alpha|_{H_\xi}$ on a neighborhood of $C$ and $E$. Then $(M,\xi,d\alpha_v\oplus\omega_\xi)$ is a vertical contact almost contact structure for $(f,C,E)$. Lemma \ref{lem:split} implies the following

\begin{lemma}\label{lem:equivcon}
Let $(M,\xi,\omega)$ be a vertical contact almost contact structure with respect to a good ace fibration $(f,C,E)$, $\alpha_v$ such that $\xi_v=\ker\alpha_v$ and $\omega_\xi$ a symplectic structure for the contact connection associated to $(f,\xi)$. Then $(M,\xi,\omega)$ and $(M,\xi,d\alpha_v\oplus\omega_\xi)$ are homotopic almost contact structures.
\end{lemma}

\noindent In order to be able to apply Lemma ~\ref{lem:contactness_condition} we need a deformation of $(M,\xi,\omega)$ such that at least in one direction the parallel transport along the deformed almost contact connection is a contactomorphism. This allows us to trivialize with the almost contact connection and obtain a vertical contact distribution constant along that direction. Thus conforming the hypotheses of Lemma \ref{lem:contactness_condition}. Both Lemmas \ref{lem:paralleltrans} and \ref{lem:equivcon} provide such a construction. The following two subsections provide details.\\
\subsection{Deformation along intersection points.}\label{ssec:defpt}
In this subsection we obtain a contact structure in a neighborhood of the fibres over a neighborhood of the intersection points of an adapted family $T$. The precise statement reads as follows:

\begin{proposition}\label{cor:def0}
Let $(M,\xi,\omega)$ be a vertical contact structure with respect to a good ace fibration $(f,C,E)$ and $T$ an adapted family. Then there exists a deformation $(\xi',\omega')$ of $(\xi,\omega)$ relative to $C$ and $E$ such that $(f,C,E)$ is a good ace fibration for $(\xi',\omega')$ and $\xi'$ is a contact structure in the pre--image of a neighborhood of the 0--cells of $|T|$.
\end{proposition}
\begin{proof}
\noindent Let $z$ be a point of intersection of the adapted family $T$, $(\phi,U)$ a sufficiently small chart centered at $z$ with the diffeomorphism $\phi:U\longrightarrow[-1,1]\times[-1,1]$, Cartesian coordinates $(s,t)\in [-1,1]\times [-1,1]$ and $N=f^{-1}(U)\backslash U(f)$. The geometric argument to prove the statement is simple. Lemmas \ref{lem:equivcon} and \ref{lem:paralleltrans} are used to trivialize $f$ over a neighborhood of the $0$--cells such that the hypotheses of Lemma \ref{lem:contactness_condition} can be applied. Let us provide the details.\\

\noindent The map $f:N\longrightarrow U$ is a smooth trivial fibration with fibre $F$. Lemma 6.8 provides an adequate trivializing diffeomorphism $g:N\longrightarrow F \times [-1,1]\times[-1,1]$. Let $(\lambda,\Omega)=(g_*\xi,g_*\omega)$ be the almost contact structure in this local model and
$$f_\lambda=\phi\circ f \circ g^{-1}:F\times[-1,1]\times[-1,1]\longrightarrow [-1,1]\times[-1,1],\quad f_\lambda(p)=(\sigma(p),\tau(p)).$$
\noindent This is a contact fibration for the distribution $\lambda$ and the almost contact structure $(\lambda,\Omega)$ is a contact structure near $g(\partial N\setminus f^{-1}(\partial U))$. Consider the 1--forms $\alpha$ and $\alpha_v$ defining the distributions $\lambda$ and $\lambda_v$. Lemma \ref{lem:equivcon} allows us to deform the symplectic structure $\Omega$ to $d\alpha_v\oplus\Omega_\lambda$ for a suitable choice of symplectic structure $\Omega_\lambda$ in the $d\alpha$--orthogonal of $\lambda_v$ in $\lambda$. Lemma \ref{lem:paralleltrans} implies that the parallel transport along the lift of the vector field $\partial_s$ to the connection $H_\lambda$ consists of contactomorphisms. This provides a specific trivialization such that the contact form satisfies the hypotheses of Lemma \ref{lem:contactness_condition}.\\

\noindent Indeed, consider the connection $H_\lambda$ for the fibration $f_\lambda$ and the vector field $\partial_s$ in the base $[-1,1]\times[-1,1]$. Let $X_s$ be the lift of $\partial_s$ to $H_\lambda$ and $m^\tau_p$ the parallel transport along the segment
$$\gamma:[0,\tau]\longrightarrow[-1,1]\times[-1,1],\quad \gamma(r)=p+(r,0).$$
That is, $m^\tau_p$ is the time--$\tau$ flow of $X_s$. There exists a small $\varepsilon\in\R^+$ such that the flow $m^\tau_p$ is well--defined for all $|\tau|<\varepsilon$ and $p\in\{0\}\times[-1,1]$. This might require a perturbation of the trivializing diffeomorphism $g$ along a neighborhood of the boundary $f_\lambda^{-1}(\partial((-\varepsilon,\varepsilon)\times[-1,1]))$.\\

\noindent In order to obtain the required trivialization consider the diffeomorphism
$$\iota: F\times(-\varepsilon,\varepsilon)\times[-1,1]\longrightarrow F\times(-\varepsilon,\varepsilon)\times[-1,1],\quad p\longmapsto\iota(p)=(m^{-\sigma(p)}_{(0,\tau(p))}(p),f_\lambda(p)).$$
\noindent The lift of the direction $\partial_s$ is part of the trivialized distribution. In precise terms, the push--forward of $\xi$ in $g^{-1}(F\times(-\varepsilon,\varepsilon)\times[-1,1])$ along $\iota\circ g$ is a distribution $(\iota \circ g)_* \xi$ given by the kernel of a $1$--form
$$\alpha_{(s,t)} + H(p,s,t) dt,\mbox{ satisfying }\partial_s\alpha_{(s,t)}=0.$$
Lemma \ref{lem:contactness_condition} can then be applied. The good set $G$ is chosen to be a suitable neighborhood of the trivialization of the boundary $\partial F \times (-\varepsilon,  \varepsilon) \times [-1,1]$. The statement of the Lemma yields a smooth function
$$\widetilde H: F\times(-\varepsilon,\varepsilon)\times[-1,1]\longrightarrow\R$$
inducing a contact structure in this local model.\\

\noindent The previous procedure has to be considered inside the manifold. We should then perform the perturbation relative to the boundary of the base $(-\varepsilon, \varepsilon)\times [-1,1]$. To this aim, consider $\delta\in\R^+$ small enough and a smooth cut--off function $c_{\delta}: [-1,1]\longrightarrow[0,1]$ satisfying
$$c_\delta(x)=1\mbox{ for }|x|\leq \delta,\quad c_\delta(x)=0\mbox{ for }|x|\geq 1-\delta.$$
Then the interpolating function
$$h(p,s,t)= c_\delta(\varepsilon^{-1} s)c_\delta(t) \widetilde H(p,s,t) + (1- c_\delta(\varepsilon^{-1} s)c_\delta(t)) H(p,s,t)$$
induces the form $\alpha=\alpha_{(s,t)} + h(p,s,t) dt$ which coincides with $\alpha_{(s,t)} + H(p,s,t) dt$ near the boundary of $(-\varepsilon,\varepsilon)\times[-1,1]$. The perturbation can thus be made relative to the boundary and inserted in the manifold. The deformation from the initial distribution to that defined by the contact form $\alpha$ satisfies the statement of the Proposition.
\end{proof}

\subsection{Deformation along curves.}

Once we have achieved the contact condition in a neighborhood of the fibres over the $0$--skeleton, we proceed with a neighborhood of the fibres over the $1$--skeleton.\\

\begin{proposition}\label{cor:defo_curve}
Let $(M,\xi,\omega)$ be a vertical contact structure with respect to a good ace fibration $(f,C,E)$, $T$ an adapted family and $\mathbb{T}$ a neighborhood of $T$. Suppose that $(M,\xi)$ is a contact structure on a neighborhood $\mathbb{O}$ of the fibres over the $0$--cells of $T$. Then there exists a deformation $(\xi',\omega')$ of $(\xi,\omega)$ relative to $C$, $E$ and $\mathbb{O}$ such that $(f,C,E)$ is a good ace fibration for $(\xi',\omega')$ and $\xi'$ is a contact structure in the pre--image of $\mathbb{T}$.
\end{proposition}

\noindent Let $\sS$ be a small neighborhood of the set of fibres over $\mathbb{T}\backslash\mathbb{O}$. 
See Figure \ref{fig:def_domain}.
\begin{figure}[h]
\includegraphics[scale=0.6]{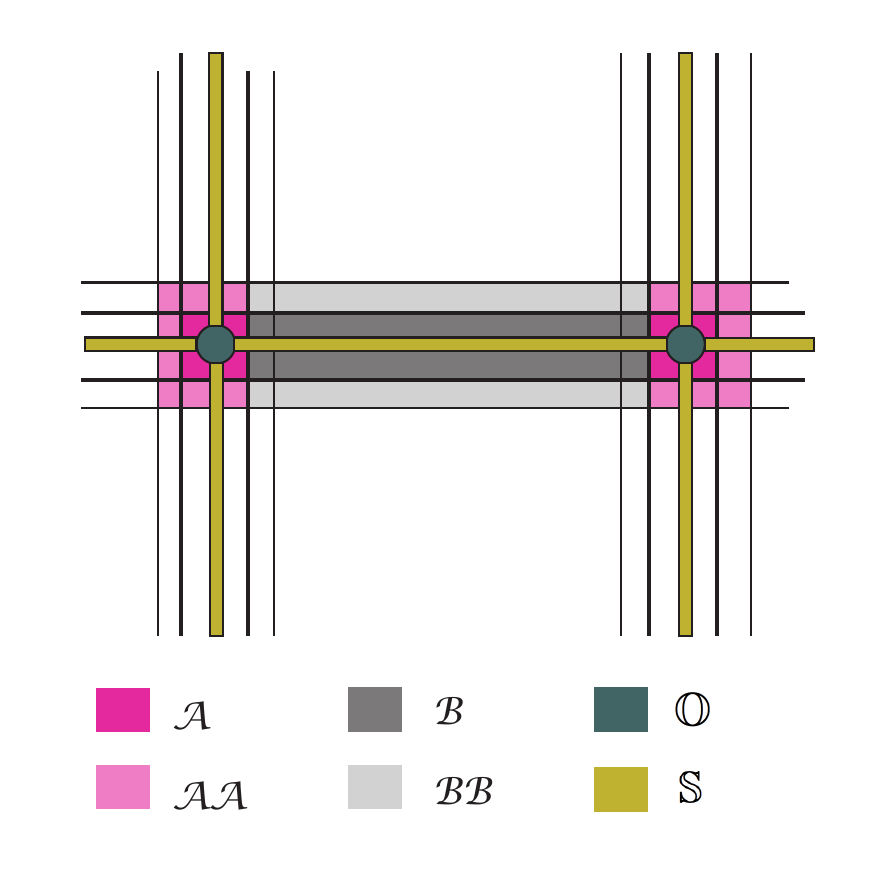}
\caption{The deformation domains.}\label{fig:def_domain}
\end{figure}
The argument applied over $\mathbb{O}$ in the previous subsection works analogously when applied to $\sS$. Thus, no detailed proof is given. The only subtlety lies in the appropriate choice of the compact set $G$ when Lemma~\ref{lem:contactness_condition} is applied.\\

\noindent Let $z,w\in\C\PP^1$ with corresponding neighborhood $\OO_z,\OO_w$; we focus on a line segment $S\subset |T|$ joining these two points. Let $\left(\phi,U\right)$ be a local chart around $S\backslash\left(\OO_z\cup\OO_w\right)$ with cartesian coordinates $(s,t)$ such that
$$\phi(U)=[-\varepsilon,\varepsilon]\times[0,1],\qquad \phi(S)=\{0\}\times[0,1].$$

\begin{lemma} \label{lem:first_def1}
There exist an arbitrarily small neighborhood $\mathbb{S}$ of $S$ and a horizontal deformation of the vertical contact almost
contact structure $(\xi,\omega)$ supported in the pre--image of $\mathbb{S}$, relative to the pre--images of $\mathbb{S}\cap\mathbb{O}_z$ and $\mathbb{S}\cap\mathbb{O}_w$, and conforming the following properties:
\begin{itemize}
\item[-] The deformation is relative to $U(f)$ where $\xi$ is already a contact structure.
\item[-] There exists a local chart $(\phi,U)$ such that the parallel transport of the associated almost contact connection along the vector field $\phi^*\partial_s$ consists of contactomorphisms.
\end{itemize}
\end{lemma}
\begin{figure}[h]
\includegraphics[scale=0.35]{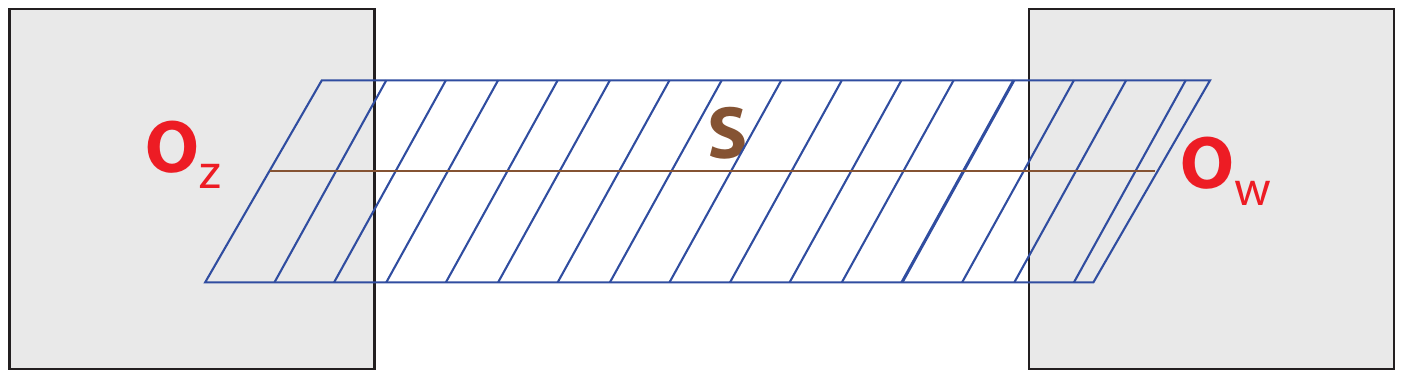}
\caption{The deformation curves $\phi^*\partial_s$.}\label{fig:convex}
\end{figure}
\noindent This follows from subsection \ref{ssec:contcon}.\\

\noindent{\bf Proof of Proposition \ref{cor:defo_curve}.} Use Lemma~\ref{lem:first_def1} to ensure that the parallel transport along the lift of $\partial_s$ is by contactomorphisms. Choose the $s$--coordinate in the neighborhood $\mathbb{S}$ in such a way that the curves which provide the lift of $\phi^*\partial_s$ either have at most one of the ends in the fibres over a small neighborhood of the $0$--skeleton or are contained therein. See Figure \ref{fig:convex}. This allows us to choose a compact set $G$ containing the fibres over the two endpoints plus a neighborhood of the boundary of all the fibres such that the intersection of $G$ with any such arc is connected. There might be the need to progressively shrink the neighborhoods of the fibres over the $0$--skeleton. Apply Lemma~\ref{lem:contactness_condition} to produce a contact structure in a
neighborhood of the fibres over the $1$--skeleton without perturbing the existing contact structure in
a small neighborhood of fibres over the endpoints.$\hfill\Box$

\section{Fibrations over the $2$--disk.}\label{sec:bands}

\noindent Let $(F,\xi_v)$ be a contact $3$--manifold, $\xi_v=\ker\alpha_v$ and $\D^2$ a 2--disk. In this Section we study contact structures on the product manifold $F\times\D^2$. Consider the coordinates $(p,r,\theta)\in F\times\D^2$. The previous sections essentially reduce Theorem \ref{main} to the existence of a contact structure on $F\times\D^2$ restricting to a prescribed contact structure on a neighborhood of the boundary $F\times\partial\D^2$. See Theorem \ref{thm:filling} in Section \ref{sec:end} for details on the end of the proof.\\

\noindent Fix an $\varepsilon\in(0,1)$ and consider $H\in C^\infty(F\times\D^2(1))$ to be a smooth function such that $\partial_r H>0$ for $r\in(1-\varepsilon,1]$. Then the $1$--form
$$\alpha=\alpha_v+H(p,r,\theta)d\theta$$
defines a distribution $\xi=\ker\alpha$. It can be endowed with the symplectic form
$$\omega=d\alpha_v+(1-\tau(r))\cdot rdr\wedge d\theta+\tau(r)dH\wedge d\theta,$$
where $\tau:[0,1]\longrightarrow[0,1]$ is an strictly increasing smooth function such that
$$\tau(x)=0\mbox{ for }x\in [0, 1-\varepsilon]\mbox{ and }\tau(x)=1\mbox{ for }x\in [1-\varepsilon/2,1].$$
Then $(\xi,\omega)$ is an almost contact structure on $F\times\D^2(1)$ which is a contact structure on the neighborhood $F\times(1-\varepsilon/2,1]\times\sS^1$ of the boundary $F\times\partial\D^2(1)$.\\

\noindent The main result in this Section is the following:
\begin{theorem}\label{thm:band}
Let $(F,\xi_v)$ be a contact $3$--manifold with $c_1(\xi_v)=0$, $\xi_v=\ker\alpha_v$ and $L$ a transverse link. Given $\varepsilon\in(0,1)$, consider a function $H\in C^\infty(F\times\D^2(1))$ such that $\partial_rH>0$ in $r\in(1-\varepsilon,1]$ and $H|_{L \times \D^2(1)} \geq 0$, and the almost contact structure
$$(\xi,\omega)=(\ker(\alpha_v+H(p,r,\theta)d\theta),d\alpha_v+(1-\tau(r))\cdot rdr\wedge d\theta+\tau(r)dH\wedge d\theta),$$
where $\tau$ is the function described above.\\

\noindent Then there exists a $1$--parametric family of almost contact structures $\{(\xi_t, \omega_t)\}$, constant along the boundary $F\times\partial\D^2(1)$ and with $(\xi_0, \omega_0)=(\xi,\omega)$ such that:
\begin{enumerate}
\item[a.] $(\xi_1, \omega_1)=(\ker\alpha, d\alpha)$ is a contact structure for some contact form $\alpha$ on $F\times\D^2(1)$.
\item[b.] The submanifold $L\times\D^2(1)$ is a contact submanifold of $(F\times\D^2(1),\xi_1)$ and the induced contact structure is a small neighborhood of a full Lutz twist along $L\times\{0\}$.
\end{enumerate}
\end{theorem}

\noindent In coordinates $(z,r, \theta) \in L\times\D^2(1)$, the contact structure obtained by a full Lutz twist in a neighborhood $\SN(L)\cong L\times\D^2$ of $L$ along $L\times\{0\}$ is described as
\begin{equation*}
\xi_{|L\times\D^2(1)} = \ker (\cos (2\pi r) dz + r \sin (2\pi r) d\theta).
\end{equation*}
\noindent Consider the domain $L \times \D^2(5/4)$ with the previous equation defining the contact structure. The term {\it small neighborhood} of a full Lutz twist refers to an open subset $U\cong L\times\D^2(1)$ such that it can be contact embedded as $L \times \D^2(1) \subset U \subset L \times \D^2(5/4)$. \\

\noindent This theorem is used to conclude Theorem \ref{main} in Section \ref{sec:end}. In brief, it is used to deform the almost contact structure over the $2$--cells of the decomposition associated to an adapted family $T$ of a vertical good ace fibration $(f,C,E)$. In this description of the fibration over the $2$--cells, the part corresponding to the exceptional divisors is the submanifold $L\times\D^2(1)$. Although the deformation in the statement is not relative to a neighborhood of them, the resulting contact structure is described in the part b. of Theorem \ref{thm:band}.\\

\noindent{\bf Example.} Suppose that the function $H\in C^\infty(F\times\D^2(1))$ also satisfies
$$H(p,1, \theta)>0,\mbox{ for all }(p,\theta)\in F\times\sS^1.$$
The contact condition for the initial form $\alpha_v+ H(p,r,\theta)d\theta$ is $\partial_rH>0$. Consider a smooth family $\{H_t\}_{t\in[0,1]}$ of functions in $F\times\D^2(1)$ such that
$$H_0=H,\quad H_1(p,0,\theta)=0,\quad \partial_r H_1>0\mbox{ for }r\in(0,1]\mbox{ and }H_t(p,1,\theta)=H_0(p,1,\theta).$$
Suppose that $H_1$ vanishes quadratically at the origin (this assumption will be implicitly made throughout the article). Then $\alpha_t=\alpha_v+H_t(p,r,\theta)d\theta$ is a family of almost contact distributions constant along the boundary $F\times\partial\D^2(1)$ such that $\ker\alpha_1$ is a contact structure. The corresponding symplectic structures on $\ker\alpha_t$ is readily constructed as in the previous discussion, and an interpolation to the symplectic form $\alpha_v+dH_1\wedge d\theta$ is required to obtain the almost contact structure $(\ker\alpha,d\alpha)$. This contact structure does conform property (a) in Theorem \ref{thm:band}.\\

\noindent The importance of Theorem \ref{thm:band} is that it also covers the case of almost contact distributions where $H$ is negative along a part of $F\times\partial\D^2(1)$. This case is handled at the cost of changing the contact structure on $L\times\D^2(1)$. This region is part of the exceptional locus $E$ and should a priori not be modified, however we will see in Section \ref{sec:end} that the control on this region ensured by Theorem \ref{thm:band} will be enough to correct that change.
\subsection{The model.} 
In this subsection we describe the model used to obtain the contact structure in the statement of Theorem \ref{thm:band}.\\

\noindent Consider the smooth $5$-dimensional manifold $F\times\sS^2$. The submanifolds
$$i_0:F_0=F\times\{(1,0,0)\}\longrightarrow F\times\sS^2\mbox{ and }i_\infty:F_\infty=F\times\{(-1,0,0)\}\longrightarrow F\times\sS^2$$
are referred to as the fibres at zero and infinity. A construction made relative to $F_\infty$ should be thought as construction on $F\times\D^2(1)$ relative to the boundary.\\

\noindent The compact smooth $3$--manifold $F$ is parallelizable. Hence the cotangent bundle $T^*F\longrightarrow F$ is isomorphic to the fibre bundle $F\times\R^3\longrightarrow F$ given by the projection onto the first factor. The canonical symplectic structure in the manifold $T^*F$ induces a contact structure in the manifold $F\times\sS^2$. For instance, given a Riemannian metric the manifold $F\times\sS^2$ can be identified with the unit cotangent bundle $\sS(T^*F)$ with respect to that metric. This is a convex hypersurface in $T^*F$ and the canonical Liouville vector field defines a contact structure $\xi_{can}$ on $\sS(T^*F)\cong F\times\sS^2$. The study of the distribution $\xi_{can}$ has been at the core of contact geometry since its foundations. See \cite{Lu2} and Appendix 4 in \cite{Ar}.\\

\noindent Consider a contact structure $(F,\Xi)$. The choice of a contact form $\alpha$ for $\Xi$ defines an embedding $F\longrightarrow T^*F$. The image of this embedding can be assumed to lie in $\sS(T^*F)$. Then $(F,\Xi)$ is seen as a contact submanifold of $(\sS(T^*F),\xi_{can})$. The symplectic normal bundle of this contact embedding is isomorphic to $\Xi$. In particular the embedding has trivial normal bundle if and only if $c_1(\Xi)=0$. See \cite{Ge3} for an application.\\

\noindent The construction of the contact structure in the following Proposition begins with the natural contact structure in $\sS(T^*F)$ thought of as a contact structure in the total space of $F\times\sS^2\longrightarrow\sS^2$.\\

\noindent In the manifold $\sS^1\times\sS^2$ there exists a unique tight contact structure. It is the contact boundary of the symplectic manifold $\sS^1\times\D^3$. The first Chern class of this tight contact structure is $0\in H^2(\{0\}\times\sS^2,\Z)\cong H^2(\sS^1\times\sS^2,\Z)$. Consider the overtwisted contact structure $\xi_{ot}$ in the homotopy class of plane fields $\{\theta\}\times T\sS^2$. It is obtained by performing half Lutz twist in the tight contact structure along the transverse knot $\sS^1\times\{0\}$. This is said to be the standard $2$--overtwisted structure on $\sS^1\times\sS^2$. Certainly its first Chern class $c_1(\xi_{ot})=2$ coincides with $c_1(T\sS^2)=2$. This homotopy class of plane fields is relevant since $T\sS^2$ is a horizontal bundle for the projection $\sS^1\times\sS^2\longrightarrow\sS^2$.\\

\noindent The basic geometric construction used to prove Theorem \ref{thm:band} is the content of the following result. A minor enhancement of the Proposition is also required, it is explained in Corollary \ref{coro:model}.
\begin{proposition} \label{propmodel}
Let $(F,\xi_v)$ be a contact $3$--manifold with $c_1(\xi_v)=0$, $\xi_v=\ker\alpha_v$ and $L$ a transverse link. Consider the manifold $(F\times\sS^2$, $\omega_{\sS^2})$ the standard area form on $\sS^2$ and the almost contact structure
$$(\xi,\omega)=(\ker\alpha_v, d\alpha_v +\omega_{\sS^2}).$$

\noindent Then there exists a contact structure $\xi_f= \ker \alpha_f$ on $F\times\sS^2$ conforming the properties:
\begin{enumerate}
\item[a.] The contact form $\alpha_f$ restricts to the initial contact form at the fibres $F_0$ and $F_\infty$:
$$i_0^* \alpha_f = \alpha_v\mbox{ and }i_{\infty}^* \alpha_f = \alpha_v.$$
\item[b.] Consider the inclusion $i_L: L \times\sS^2=\bigsqcup (\sS^1 \times \sS^2)\longrightarrow F \times \sS^2$. Then the contact form $i_L^* \alpha_f$ defines the contact structure $\xi_{ot}$ on each $\sS^1 \times \sS^2$.\\
\item[c.] The almost contact structures $(\xi,\omega)$ and $(\ker \alpha_f, d\alpha_f)$ are homotopic relative to $F_\infty$.
\end{enumerate}
\end{proposition}
\begin{proof} This is a rather long proof. It is divided according to the construction and the verification of each of the three properties.\\

\noindent {\bf Construction.} Since $c_1(\xi_v)=0$, there exist a global framing $\{ X_1, X_2 \in \Gamma(\xi_v)\}$ of the contact distribution $\xi_v$. Denote by $X_0$ the Reeb vector field associated to the contact form $\alpha_0=\alpha_v$. Therefore $\{ X_0, X_1, X_2 \}$ is a global framing of $TF$. Let $\{ \alpha_0, \alpha_1, \alpha_2 \}$ be the dual framing. It can be assumed that the transverse link $L$ is an orbit of the Reeb vector field $X_0$. In particular $\alpha_1$ and $\alpha_2$ vanish along $L$. Denote the standard embedding of the $2$--sphere as $e=(e_0, e_1, e_2): \sS^2 \longrightarrow \R^3$. The previous discussion endows the smooth manifold $F\times\sS^2$ with a natural contact structure. We use an explicit model for the argument. It is a computation to verify that
$$\lambda = e_0 \cdot \alpha_0 + e_1 \cdot \alpha_1 + e_2 \cdot \alpha_2 $$
is a contact form on $F\times \sS^2$. The important properties are that $\{ \alpha_0, \alpha_1,\alpha_2 \}$ is a framing and the map $e$ is a star--shaped embedding. The contact structure $\ker\lambda$ is contactomorphic to $\xi_{can}$. From the classical viewpoint it is clear that $\ker\lambda$ is a contact structure. See \cite{Lu2}.\\

\noindent In spherical coordinates $(t,\theta)\in[0,1]\times[0,1]$ the embedding can be described as
\begin{eqnarray}
e_0(t,\theta) & = & \cos (\pi t), \nonumber \\
e_1(t,\theta) & = & \sin (\pi t) \cos (2\pi \theta), \nonumber \\
e_2(t,\theta) & = & \sin (\pi t) \sin(2\pi \theta). \nonumber
\end{eqnarray}

\noindent Note that $F_{\infty}=F \times (-1,0,0)$ and $F_0=F \times (1,0,0)$ are contactomorphic contact submanifolds of $(F\times\sS^2,\ker\lambda)$ with trivial normal bundle. Consider two copies of $F\times\sS^2$, we can perform a contact fibered sum along their $F_{\infty}$ fibres, see \cite{Ge}. This operation is done in order to obtain two fibres with the contact form $\alpha_0$. Those coming from the two zero fibres $F_0$ in the two copies of $F\times\sS^2$. Let us provide an explicit equation for the contact form in this fibered sum.  \\

\noindent A tentative modification of $\lambda$ is obtained by considering the following map
\begin{eqnarray*}
\kappa_0(t,\theta) & = & \cos (2\pi t), \\
\kappa_1(t,\theta) & = & \sin (2\pi t) \cos (2\pi \theta), \\
\kappa_2(t,\theta) & = & |\sin (2\pi t)| \sin(2\pi \theta),
\end{eqnarray*}
and the $1$--form $\kappa_0\cdot\alpha_0 +\kappa_1\cdot\alpha_1+\kappa_2\cdot\alpha_2$. Due to the appearance of the absolute value this form is just continuous. Observe though that in the smooth area it is a contact form. Let us perturb it to a smooth $1$--form.\\

\noindent Define a smooth map $t:[0,1]\longrightarrow[0,1]$ such that:
$$t(0)=0,\mbox{ }t(1/2)=1/2,\mbox{ }t(1)=1, t'(v)>0\mbox{ for }v\in[0,1/2) \cup (1/2,1]\mbox{ and }t^{(k)}(1/2)=0\mbox{ }\forall k\in\N.$$ \\
\noindent This allows us to reparametrize the sphere with coordinates $(v,\theta)\in[0,1]\times[0,1]$. The following map is denoted by $(e_0,e_1,e_2)$ in order to ease notation. This should not lead to confusion since the map formerly referred to as $(e_0,e_1,e_2)$ is not to be considered again. Consider the smooth map
\begin{eqnarray*}
e_0(v,\theta) & = & \cos (2\pi t(v)), \\
e_1(v,\theta) & = & \sin (2\pi t(v)) \cos (2\pi \theta), \\
e_2(v,\theta) & = & |\sin (2\pi t(v))| \sin(2\pi \theta).
\end{eqnarray*}
It is indeed smooth because $t^{(k)}(1/2)=0$. This almost provides the desired $1$--form for the fibre connected sum. Define the smooth function $h(v)= v(1-v) \sin (2\pi v)$ and the $1$--form $\eta=c \cdot h(v)d\theta$, where $c$ is a small positive constant.\\

\noindent{\bf Assertion.} There exists a choice of $c\in\R^+$ such that the $1$--form defined as
\begin{equation}
\alpha_f = e_0 \alpha_0 + e_1 \alpha_1 + e_2  \alpha_2- \eta \label{eq:aT}
\end{equation}
is a contact form over the fibre connected sum of two copies of $F\times\sS^2$ along the fibres $F_\infty$.\\

\noindent This concludes the construction of the contact form in the manifold $F\times\sS^2$ obtained in the Theorem. The contact form $\alpha_f$ also conforms property a. in the statement of the Theorem.\\

\noindent{\bf Proof of Assertion.} Consider the following volume form $\nu = \sin(\pi v)dv \wedge d\theta \wedge \alpha_0 \wedge \alpha_1 \wedge \alpha_2$ on $F \times \sS^2$ and compute the exterior differential
\begin{equation*}
d\alpha_f = de_0 \wedge \alpha_0 + de_1\wedge \alpha_1 + de_2  \wedge \alpha_2  + e_0 d\alpha_0 + e_1 d\alpha_1 + e_2  d\alpha_2 - d\eta.
\end{equation*}
\noindent The contact condition states that $\alpha_f\wedge(d\alpha_f)^2$ is a positive multiple of $\nu$. Let us express it as
\begin{eqnarray*}
\alpha_f \wedge (d\alpha_f)^2 & = &  \eta_1 + c\eta_2 + c\eta_3,
\end{eqnarray*}
where $\eta_1,\eta_2,\eta_3$ are the following $5$--forms:
\begin{eqnarray*}
\eta_1 & = & \left|
\begin{array}{ccc} e_0 & e_1 & e_2 \\ 
\partial_t e_0 & \partial_t e_1 & \partial_t e_2 \\ 
\partial_\theta e_0 & \partial_\theta e_1 & \partial_\theta e_2 \\\end{array} \right| t'(v)^2 dv \wedge d\theta \wedge \alpha_0 \wedge \alpha_1 \wedge \alpha_2 = \\[10pt] & = & 4\pi^2|\sin(2\pi t(v))|(t'(v))^2dv \wedge d\theta \wedge \alpha_0 \wedge \alpha_1 \wedge \alpha_2, \\[10pt]
\eta_2 & = & - e_0^2\cdot h'(v)\cdot\alpha_0 \wedge d\alpha_0 \wedge dv \wedge d\theta, \\[10pt]
\eta_3 & = & -\sum_{i+j\geq1} (e_i\cdot e_j\cdot h'(v))\cdot\alpha_i \wedge d \alpha_j \wedge dv \wedge d\theta + \sum_{i,j} (e_i\cdot h(v))\cdot de_j\wedge d\alpha_i \wedge \alpha_j \wedge d\theta.
\end{eqnarray*}
The indices belong to $i,j\in\{0,1,2\}$. Evaluating at $v=1/2$ we obtain:
\begin{eqnarray*}
\eta_2 (p, 1/2, \theta) & = & \frac{\pi}{2} \alpha_0 \wedge d\alpha_0 \wedge dv \wedge d\theta =  \frac{\pi}{2} dv \wedge d\theta \wedge \alpha_0 \wedge \alpha_1 \wedge \alpha_2,\\
\eta_1 (p, 1/2, \theta)& = & 0,\\
\eta_3 (p, 1/2, \theta)& = & 0.
\end{eqnarray*}

\noindent Therefore, there is a small constant $\delta>0$ such that the $5$--form $\eta_2+\eta_3$ is a positive volume form in the region $F \times [1/2-\delta, 1/2 + \delta] \times [0,1]$. The function $t(v)$ is strictly increasing except at $v=1/2$. Hence, there exists a constant $B>0$ such that $t'(v) >B$ for any $v\in [0, 1/2-\delta] \cup [1/2+\delta, 1]$.\\

\noindent Let us write $\eta_1(p,v,\theta) = g_1(p, v, \theta) \nu$ and $\eta_2 +\eta_3 = g_2(p,v,\theta) \nu$. There exist constants $C,M\in\R^+$ such that $g_1> C>0$ for $v\in [0, 1/2-\delta] \cup [1/2+\delta, 1]$, and $|g_2| \leq M$.\\

\noindent Choose the initial constant $c\in\R^+$ to satisfy $cM \leq C$. Then we obtain the following bound for $v\in [0, 1/2-\delta] \cup [1/2+\delta, 1]$:
$$\alpha_f \wedge (d\alpha_f)^2  =  \eta_1 + c\eta_2 + c\eta_3 = (g_1 +c g_2) \nu> C-cM\geq0.$$
Hence the form $\alpha_f$ is a contact form in this region. The following bound holds in the remaining region $v\in [1/2-\delta,1/2+\delta]$:
$$\alpha_f \wedge (d\alpha_f)^2  =  \eta_1 + c\eta_2 + c\eta_3 = (g_1 +c g_2) \nu> cg_2\geq0.$$
\noindent Thus $\alpha_f$ is a contact form in the fibre connected sum $F\times\sS^2$.\hfill$\Box$\\

\noindent {\bf Property b.} The contact form $\alpha_v$ associated to $\xi_v$ has been chosen such that its Reeb vector field $X_0$ is tangent to the link $L$. Thus $\alpha_1$,$\alpha_2$ vanish on $L$. Restricting the contact form $\alpha_f$ in the equation (\ref{eq:aT}) to the submanifold we obtain
\begin{equation}
i_L^* (\alpha_T)  = \cos(2\pi t(v)) dz - c v(1-v) \sin (2\pi v) d\theta , \label{eq:OT}
\end{equation}
where $(z,v,\theta)\in \sS^1\times\sS^2$. This is an equation of the contact structure $\xi_{ot}$ on each $\sS^1 \times \sS^2$. Indeed, consider $a(v)=\cos(2\pi t(v))$ and $b(v)=v(1-v)\sin(2\pi v)$. Then the curve parametrized by $(a(v),b(v))$ rotates once around the origin and the tangent vector field  $(a'(t), b'(t))$ is transverse to the radial direction, i.e. $\partial_r$, on $(0,1)$.\\

\noindent {\bf Property c.} Let $f_F:F \longrightarrow [0,1]$ be a Morse function on the $3$--manifold $F$ with a single minimum $q\in F$. Then
$$f(p,v, \theta)= f_F(p) - (1+f_F(p))v^2: F\times\sS^2 \longrightarrow [-1,1]$$
is a smooth Morse--Bott function on $F \times \sS^2$ whose non--degenerate critical points belong to the central fibre $F_0$ and has $F_\infty$ as a critical manifold. Let us use the associated cell decomposition relative to the level $f^{-1}((-\infty, -1])= F_{\infty}$. It is generated by the descending manifolds associated to each critical point. It has a unique $2$--cell $\sigma_q^2= \{q \} \times (\sS^2 \backslash \{\infty \})$, corresponding to the critical point $(q,0,0)$.\\

\noindent Note that the resulting almost contact structure and the initial one coincide near $F_\infty$ and we only need to compare them as almost contact structures on the disk relative to the boundary. Due to Lemma \ref{lem:2sk}, a pair of almost contact distributions homotopic over the disk $\sigma_q^2$ relative to its boundary are homotopic on the $5$--manifold $F\times\sS^2$. To conclude Property c. we verify that such relative homotopy exists along $\sigma_q^2$. The almost contact distribution $\xi$ in the statement of the Proposition can be written as $\xi=\ker \alpha_v \oplus T\sS^2$. Its symplectic structure is induced by the symplectic structure on each of the factors. Note that both $\ker\alpha_v$ and $T\sS^2$ are $\rk_\R=2$ symplectic bundles. This is tantamount to $\rk_\R=2$ oriented bundles.\\

\noindent Consider a trajectory $\gamma$ of the Reeb flow through $q$
$$\gamma:(-\varepsilon,\varepsilon)\longrightarrow F,\quad \gamma(0)=q.$$
The submanifold $(V,\xi_{ot})=(\gamma \times \sS^2,\xi_f|_{\gamma \times \sS^2})$ is a contact submanifold of the contact manifold $(F\times\sS^2,\ker\alpha_f)$. A contact from is given by the equation (\ref{eq:OT}). As suggested by the notation, the contact form $\alpha_{ot}=\alpha_f|_V$ defines the overtwisted structure $\xi_{ot}$ on $(-\varepsilon, \varepsilon)\times(\sS^2\setminus\{\infty\})$.\\

\noindent Hence the two subbundles of $TV$
$$\xi_{ot}\longrightarrow\sigma_q^2,\quad T\sS^2\longrightarrow\sigma_q^2$$
are homotopic as oriented subbundles relative to the boundary of the disk. Thus relative homotopic as symplectic bundles. This provides a homotopy in the $2$--dimensional horizontal part. Let us deal with the vertical bundle.\\

\noindent  The initial vertical subbundle is $\xi_v=\ker\alpha_v$, it does satisfy the splitting
$${\xi_v}_{|\sigma_q^2} \oplus TV_{|\sigma_q^2}=T(F\times\sS^2)_{|\sigma_q^2}.$$
The resulting vertical subbundle in the distribution $\xi_f$ can be constructed as the symplectic orthogonal subbundle $\nu_{ot}$ of $\xi_{ot}$. This yields the decomposition
$${\nu_{ot}}_{|\sigma_q^2} \oplus TV_{|\sigma_q^2}=T(F\times\sS^2)_{|\sigma_q^2}.$$
The space of rank--2 oriented vector bundles transverse to the rank--3 vector bundle $TV$ is contractible. Hence ${\nu_{ot}}_{|\sigma_q^2}$ is homotopic to ${\xi_v}_{|\sigma_q^2}$ as rank--2 symplectic distributions.\\

\noindent On the unique 2--cell $\sigma_q^2$ both splittings $\xi=\xi_v \oplus T\sS^2$ and $\xi_f= \nu_{ot} \oplus \xi_{ot}$ hold. Note that the bundle $T\sS^2$ is homotopic to $\xi_{ot}$ inside $TV$ and $\xi_v$ is homotopic to $\nu_{ot}$ through planes transverse to $TV$. Since the subbundles are pairwise homotopic as symplectic distributions and these homotopies do not interact, $\xi$ and $\xi_f$ are also homotopic as symplectic distributions.
\end{proof}

\noindent In the proof of Property c. of Proposition \ref{propmodel} we have only used the 2--skeleton to verify the statement. Lemma \ref{lem:2sk} ensures that this is enough. There is an alternative geometric approach to produce the homotopy. Indeed, the Reeb trajectories of $\alpha_v$ produce a foliation $\SL$ on $F$. This induces a foliation $\SL\times\D^2$ with $3$--dimensional contact leaves. The argument in the proof of Property c. can be made parametric to construct an explicit almost contact homotopy.\\

\noindent The norm of the function $H$ in the statement of Theorem \ref{thm:band} does translate into a geometric feature. This is the size of a certain neighborhood. This is explained in the subsequent subsection. Let us enhance the conclusion of Proposition \ref{propmodel} in order to obtain an arbitrarily large contact neighborhood of a fibre.\\

\noindent{\bf Property d.} Let $R\in\R^+$ be given. There exists a neighborhood $U_\infty$ of the fibre $F_\infty$ and a trivializing diffeomorphism $\psi: F \times \D^2(R)\longrightarrow U_{\infty}$ such that
\begin{itemize}
\item[-] $\psi(F \times \{ 0 \})=F_{\infty}$,
\item[-] $\psi^* \alpha_f = \alpha_v + r^2d \theta$.
\end{itemize}

\noindent This property could have been included in the statement of Proposition \ref{propmodel}. It is stated apart to ease the comprehension.

\begin{corollary} \label{coro:model}
There exists a contact manifold $(F\times\sS^2,\xi_f=\ker\alpha_f)$ conforming a. to d.
\end{corollary}

\begin{proof}
The contact structure $(F\times\sS^2,\xi_f=\ker\alpha_f)$ obtained in Proposition \ref{propmodel} does satisfy properties a.-- c. Let us modify it in order to satisfy Property d. The contact neighborhood theorem provides a neighborhood $U_{\infty}$ of the fibre $F_{\infty}$ and a contactomorphism $\psi_\varepsilon: F \times\D^2(\varepsilon) \to U_{\infty}$, for some $\varepsilon\in\R^+$. In case $R\leq \varepsilon$ the statement follows.  \\

\noindent Suppose that $R\geq\varepsilon$, then we use the following covering trick (introduced in \cite{NP}). Let $k\in\N$ be an integer and consider the ramified covering
\begin{eqnarray*}
\phi_k: F \times \sS^2  = F \times \CP^1 & \longrightarrow & F \times  \CP^1 \\
(p,z) & \longmapsto & (p,z^k). 
\end{eqnarray*}
The branch locus consists of the fibres $F_0$ and $F_{\infty}$. Both fibres are contact submanifolds in $(F\times\sS^2,\ker\alpha_f)$ and we can lift the contact form to a contact form $\alpha_f^k=  \phi_k^* \alpha_f$ in the domain of the covering map. Lifting the formula (\ref{eq:aT}), we obtain
\begin{equation}
\alpha_f^k = \cos(2\pi t(v)) \alpha_0 + \sin (2\pi t(v)) \cos (2\pi k\theta) \alpha_1 + |\sin (2\pi t(v))| \sin(2 \pi k\theta) \alpha_2+ k\eta \label{eq:aTk}
\end{equation}
\noindent The reader can verify that properties a.-- c. are still satisfied by the contact structure $\ker\alpha_f^k$. Regarding Property d, observe that $\psi^* \alpha_f^k= \alpha_v +kr^2 d\theta$. Consider the scaling diffeomorphism
\begin{eqnarray*}
g_k: F \times \D^2(\sqrt{k}\cdot\varepsilon) & \longrightarrow & F \times \D^2(\varepsilon) \\
(p, r, \theta) & \longmapsto & (p, r/\sqrt{k}, \theta).
\end{eqnarray*}
Then the trivializing diffeomorphism $\psi_\varepsilon\circ g_k$ satisfies $(\psi_\varepsilon \circ g_k)^* \alpha_f^k = \alpha_v + r^2 d\theta$. Choose $k\in\N$ such that $\sqrt{k}\cdot \varepsilon \geq R$ to conclude the statement.
\end{proof}

\noindent To ease notation, we can refer to the contact structures resulting either of Proposition \ref{propmodel} or Corollary \ref{coro:model} as $\xi_f$. Since the latter has better properties than the former, $\xi_f$ refers to that in Corollary \ref{coro:model}.

\begin{rks}
Suppose that the contact manifold $(F,\xi_v)$ is overtwisted, then the contact structure $\xi_f$ contains a plastikstufe. Confer \cite{Ni},\cite{Pr2}. It can be constructed as follows.\\

\noindent Restrict the contact form $\alpha_f^k$ to $\{ (p, v, \theta) \in F \times \sS^2: v=1/2 \}\cong F\times\sS^1$. This is a contact bundle over the $\sS^1$--factor. The induced contact connection satisfies that $\pi^* \partial_{\theta}= \partial_\theta$ and thus the parallel transport is the identity. In particular, the parallel transport of the overtwisted disk on the fibre generates a plastikstufe.\\

\noindent The contact manifold $F \times\D^2(1/2)$ will be contact embedded in our initial manifold $(M,\xi)$, is $PS$--overtwisted. Note that Section \ref{sec:vertical} forces $(F,\xi_v)$ to be overtwisted contact structures. Hence the contact structures constructed in Theorem \ref{main} are $PS$--overtwisted.
\end{rks}

\subsection{The proof.} In this subsection we conclude the proof of \ref{thm:band}. The essential geometric ideas have been introduced in Proposition \ref{propmodel}. The necessary details to conclude are provided.\\

\noindent Let us introduce a definition. It is given in order to stress the relevance of the size in a neighborhood.
\begin{definition}
Let $(F,\xi_v=\ker \alpha_v)$ be a contact manifold. For $A \in\R^+$, the manifold $F \times [-A,A] \times\sS^1$ with the contact structure $\alpha_A= \alpha_v + td\theta$ is called the $A$--standard contact band associated to $(F,\ker\alpha_v)$.
\end{definition}

\noindent The role of this definition is elucidated in the following lemma.

\begin{lemma} \label{lem:emb}
Let $(F,\xi_F)$ be a contact manifold, $\xi_F=\ker\alpha_F$. Consider a contact manifold $(F \times [0, 1] \times \sS^1,\xi)$ with contact form $\alpha_F + Hd\theta$, $H\in C^\infty(F\times[0,1]\times\sS^1)$.\\

\noindent Suppose that $|H|< A$, for some $A\in\R^+$. Then, there exists a strict contact embedding of $(F \times [0, 1] \times \sS^1,\alpha)$ in the $A$--standard contact band associated to $(F, \alpha_F)$.
\end{lemma}
\begin{proof}
Consider the embedding defined as
\begin{eqnarray*}
\Psi_A: F \times [0, 1] \times \sS^1 & \longrightarrow & F \times [-A,A] \times \sS^1 \\
\left(p, t, \theta\right) & \longrightarrow & \left(p, H(p,t,\theta), \theta\right).
\end{eqnarray*}
This is a diffeomorphism onto its image because the form $\alpha_F+Hd\theta$ is a contact form, or equivalently $\partial_tH>0$.
\end{proof}

\noindent The remaining ingredient for the proof of Theorem \ref{thm:band} is the subsequent lemma.\\

\noindent Let $l\in\R^+$ be a constant, $l>1$. Consider a smooth function $\kappa_l:[0,2l+1]\longrightarrow[0,l]$ with
$$\kappa_l(r)=0\mbox{ for }r\in [0,l],\quad \kappa_l(r)=r-l-1\mbox{ for }r\in [2l,2l+1].$$
Consider $(r,\theta)\in\D^2_l$ to be polar coordinates for the $2$--disk $\D_l^2$ of radius $2l+1$. Suppose that $F$ is a manifold, the subset $F\times\{a\leq r\leq b\}$ of the product $F\times\D^2$ will be denoted $F\times [a,b]\times\sS^1$. Similarly, $F\times(a,b]\times\sS^1$ refers to the subset $F\times\{a<r\leq b\}\times\sS^1$.

\begin{lemma} \label{lemma:model}
Let $(F,\xi_v)$ be a contact $3$--manifold with $c_1(\xi_v)=0$, $\xi_v=\ker\alpha_v$, $l\in(1,\infty)$ and $L$ a transverse link. Consider the standard area $\omega_{\D}$ on the $2$--disk $\D^2_l$ and the almost contact structure on $F\times\D^2_l$ described as
$$(\xi,\omega)=(\ker(\alpha_v+ \kappa_l(r)d\theta), d\alpha_v +\omega_{\D}).$$
Then there exists a contact structure $\xi_1=\ker\alpha_1$ on $F\times\D^2_l$ such that:
\begin{enumerate}
\item[A.] The region $F\times[1,2l+1]\times\sS^1$ is an $l$--standard contact band for $(F,\ker\alpha_v)$:
$$\alpha_1|_{F\times[1,2l+1]\times\sS^1} = \alpha_v + (r-l-1)d\theta.$$

\item[B.] Consider the inclusion $i_L:L\times\D^2_l=\bigsqcup(\sS^1\times\D^2_l) \longrightarrow F \times\D^2_l$. Then the contact form $i_L^* \alpha_f$ defines a small neighborhood of a full Lutz twist on each $\sS^1\times\D^2_l$.\\

\item[C.] $(\xi,\omega)$ and $(\xi_1,d\alpha_1)$ are homotopic relative to the boundary $F\times\partial\D^2_l$.
\end{enumerate}
\end{lemma}

\begin{proof}
Consider Property d. in Proposition \ref{propmodel} and Corollary \ref{coro:model} with radius $R=\sqrt{l}$. Let $(F\times\sS^2,\xi_f=\ker\alpha_f)$ be the contact manifold obtained in Corollary \ref{coro:model}. Then there exists a contact neighborhood $U_\infty$ of the fibre $F_{\infty}$ and a trivializing diffeomorphism
\begin{equation*}
\psi: F\times\D^2(\sqrt{l})\longrightarrow U_\infty\mbox{ such that }\psi^*\alpha_f = \alpha_v + r^2 d\theta.
\end{equation*}
The diffeomorphism $\psi$ also identifies $\psi:F\times(0,\sqrt{l}]\times\sS^1\longrightarrow U_\infty\setminus F_\infty$.\\

\noindent Define the following map
$$m: F\times [-l,0) \times \sS^1\longrightarrow F\times (0,\sqrt{l}] \times \sS^1,\quad m(p,x,\theta)= (p, \sqrt{-x}, -\theta).$$
It satisfies $(\psi\circ m)^* \alpha_f = \alpha_v +rd\theta$. This form extends to the region $F\times [-l,l] \times \sS^1$ with the same expression. Then the manifold $F\times\D^2_l$ is obtained by gluing the annular region $F\times[0,l]\times\sS^1$ to the annular region
$$F\times(0,\sqrt{l}]\times\sS^1\cong F\times[-l,0)\times\sS^1\mbox{ identified via }m,$$
and using the contactomorphism $\psi$ restricted to $F\times(0,\sqrt{l}]\times\sS^1$ to perform the gluing construction in $(F\times\sS^2)\setminus F_\infty$. The construction implies that Property A holds. Properties B and C follow from Properties b and c in Corollary \ref{coro:model} since the manifold $(F\times\sS^2)\setminus F_\infty$ satisfies them.
\end{proof}

\noindent {\bf Proof of Theorem \ref{thm:band}.}
Let $\varepsilon>0$ be a small constant. The function $H$ is $C^0$--bounded on the compact manifold $\F=F\times\D^2(1)$. Let $l\in(1,\infty)$ be an upper bound such that $\|H\|_{C^0}<l-\varepsilon/4$. Consider coordinates $(p,r,\theta)\in\F$ and a smooth function $h\in C^\infty(\F)$ such that
\begin{itemize}
\item[-] $h(p,r,\theta)=0$ for $r\in[0,1-2\varepsilon]$,
\item[-] $h(p,r,\theta)=r-l-(1-\varepsilon)$ for $r\in[1-\varepsilon,1-3\varepsilon/4]$,
\item[-] $\partial_rh>0$ for $r\in[1-3\varepsilon/4,1-\varepsilon/2]$,
\item[-] $h(p,r,\theta)=H(p,r,\theta)$ for $r\in[1-\varepsilon/2,1]$.\\
\end{itemize}

\noindent The almost contact structure $(\xi,\omega)$ is homotopic relative to the boundary to the almost contact structure defined by
$$(\xi_h,\omega_h)=(\ker(\alpha_v+h(p,r,\theta)),d\alpha_v+(1-\tau(r))\cdot rdr\wedge d\theta+\tau(r)dh\wedge d\theta).$$
\noindent The homotopy is provided by a relative homotopy between the functions $h(p,r,\theta)$ and $H(p,r,\theta)$  and Lemma \ref{lem:split}. Hence the departing almost contact structure can be considered to be $(\xi_h,\omega_h)$ .\\

\noindent The neighborhood $F\times(1-\varepsilon,1]\times\sS^1$ of the boundary $F\times\partial\D^2(1)\subset\F$ is a contact manifold. By Lemma \ref{lem:emb}, $F\times(1-\varepsilon,1]\times\sS^1$ contact embeds in an $l$--standard contact band $F\times[-l,l]\times\sS^1$. Denote this embedding by $\phi$. It depends on the Hamiltonian $h\in C^\infty(\F)$ in the interval $(1-\varepsilon,1]$. Observe that $\phi(F\times\{1-\varepsilon\}\times\sS^1)=F\times\{-l\}\times\sS^1$ since $h(p,1-\varepsilon,\theta)=-l$.\\

\noindent Consider the almost contact manifold $(F\times\D^2_l,\xi_1=\ker\alpha_1)$ in the statement of Lemma \ref{lemma:model}. Property A implies the existence of a contactomorphism
$$\iota:F\times[-l,l]\times\sS^1\longrightarrow F\times[1,2l+1]\times\sS^1\subset (F\times\D^2_l,\xi_1),\quad\iota(p,r,\theta)=(p,r+(l+1),\theta)$$
embedding the $l$--standard contact band in a neighborhood of size $2l$ of the boundary of $F\times\D^2_l$. Consider the composition
$$j=\iota\circ\phi:F\times(1-\varepsilon,1]\times\sS^1\longrightarrow F\times\D^2_l.$$
In particular it satisfies $j(F\times\{1-\varepsilon\}\times\sS^1)=F\times\{1\}\times\sS^1\subset F\times\D^2_l$ and embeds a neighborhood of the boundary $F\times\{1-\varepsilon\}\times\sS^1$ via $$j:F\times(1-\varepsilon,1-7\varepsilon/8)\times\sS^1\subset\F\longrightarrow F\times[1,2l+1]\times\sS^1\subset F\times\D^2_l,\quad j(p,r,\theta)=(p,r+\varepsilon,\theta).$$

\noindent The required contact structure in the statement of Theorem \ref{thm:band} is obtained by extending $j$ to the interior of the manifold $F\times\D^2(1-\varepsilon)\subset\F$ and pulling--back the contact structure from $(F\times\D^2_l,\ker\alpha_1)$. Indeed, consider $\widetilde{j}$ a smooth embedding such that
$$\widetilde{j}:F\times\D^2(1)\longrightarrow F\times\D^2_l,\quad\widetilde{j}|_{F\times(\D^2(1)\setminus\D^2(1-\varepsilon))}=j.$$
For instance one can consider the extension to be
$$\widetilde j|_{F\times\D^2(1-\varepsilon)}:F\times\D^2(1-\varepsilon)\longrightarrow F\times\D^2(1),\quad (p,r,\theta)\longmapsto (p,c(r),\theta),$$
where $c:[0,1-\varepsilon]\longrightarrow[0,1]$ is a smooth function such that
\begin{itemize}
\item[-] $c(t)=t$ near $t=0$,
\item[-] $c(t)=t+\varepsilon$ near $t=1-\varepsilon$,
\item[-] $c'(t)>0$ for $t\in[0,1]$.
\end{itemize}
Then $\widetilde{j}^*(\xi_1)$ is the required contact structure. Property B in Lemma \ref{lemma:model} and the fact that the function $H$ is positive in a neighborhood of $L$ imply Property b in the Theorem. \\

\noindent Let us justify that the obtained contact structure is homotopic to the initial almost contact structure relative to the boundary $F\times\partial\D^2(1)$. The homotopy obstruction appears in the $2$--skeleton and therefore it is enough to find the homotopy at a disk $\{ p\} \times \D^2(1) \subset \F$. An analogous computation to the one detailed in the proof of Property c. of Proposition \ref{propmodel} yields the same result. Hence the resulting contact structure $\xi_1$ is homotopic as an almost contact structure to the initial almost contact structure $(\xi,\omega)$ relative to the boundary. \hfill $\Box$

\begin{remark}
The central ingredient in this construction is the existence of a contact structure $\xi$ on $F \times \sS^2$ with the following two properties:
\begin{itemize}
\item[-] It restricts to a given contact structure $(F, \xi_F)$ on a fibre $F \times \{ p \}$,
\item[-] The contact structure $\xi$ is homotopic to the almost contact structure $\xi_F \oplus T\sS^2$.
 \end{itemize}
The use of the space of contact elements space forces the fibre to have vanishing Chern class and part of Section \ref{sec:blowup} is invested to achieve this hypothesis. Since the submission of this article, the articles \cite{BCS,HW} provide a contact structure on $F\times\sS^2$ conforming the above properties. Their use would simplify Subsection \ref{ssec:gace}.\\
\end{remark}


\section{Horizontal Deformation II} \label{sec:end}
The arguments in the previous sections are gathered to conclude the proof of Theorem \ref{main}.
\subsection{Contact Structure in the fibration}
\begin{theorem} \label{thm:filling}
Let $(M,\xi,\omega)$ be an almost contact structure and $(f,C,E)$ a good ace fibration adapted to it. Suppose that $(\xi,\omega)$ is vertical with respect to $(f,C)$ and $T$ is an adapted family such that $\xi$ is a contact structure over a regular neighborhood of $|T|$. Then $(\xi,\omega)$ is homotopic to a contact structure $\xi'$ and the restriction of $\xi'$ to the exceptional $3$--spheres in $E$ induces the homotopically standard overtwisted contact structure.
\end{theorem}

\noindent The standard overtwisted structure is the unique overtwisted contact structure on $\sS^3$ homotopic to the standard contact structure $\xi_{std}$.\\

\noindent A neighborhood of the intersection of an exceptional $3$--sphere with a fibre of $f$ is diffeomorphic to $\sS^1\times\D^2\times\D^2$. Let $(z,r,\theta,\rho,\phi)$ be coordinates for such a neighborhood, the triple $(z,\rho,\phi)$ belong to the fibre. It can be considered as a trivial fibration over the first pair of factors
$$\pi:\sS^1\times\D^2\times\D^2\longrightarrow\sS^1\times\D^2,\quad (z,r,\theta,\rho,\phi)\longmapsto (z,r,\theta).$$
There also exists a contact structure given by the contact form $\alpha=dz+r^2d\theta+\rho^2d\phi$ on the neighborhood. This induces a contact connection $A_{\pi}$ for the fibration $\pi$. Let $\delta\in\R^+$ and suppose the horizontal $2$--disk $(\rho,\phi)\in\D^2(\delta)$ is of radius $\delta$.

\begin{lemma} \label{lem:parallel}
Consider the contact manifold $(\sS^1\times\D^2\times\D^2(\delta),\ker(dz+r^2d\theta+\rho^2d\phi))$, $\pi$ the projection onto the first pair of factors and $A_\pi$ the associated contact connection. The flow of the lift of $\partial_r$ to $A_\pi$ preserves the submanifold $\{(z,r,\theta, \rho, \phi)\in X: \rho=\delta/2 \}.$
\end{lemma}
\begin{proof}
The vector field $\partial_r$ belongs to the contact distribution. The vertical directions are generated by $\partial_\rho,\partial_\phi$ and the symplectic form pairs them via $\rho\cdot d\rho\wedge d\phi$. Hence $\partial_r$ is itself the lift to $A_\pi$. The statement follows.
\end{proof}

\noindent {\bf Proof of Theorem \ref{thm:filling}.} The complement of a regular neighborhood of $|T|$ in $\CP^1$ is a disjoint collection $\{B_1,\ldots,B_a\}$ of $2$--disks. The distribution $\xi$ is a contact structure in the fibres of $f$ close to the boundary of $B_1\cup\ldots\cup B_a$. The restriction of $f$ to the preimages of each $\mathcal{B}\in\{B_i\}$ is a smooth fibration since the critical values of $f$ lie in the complement of the set $B_1\cup\ldots\cup B_a$. In order to conclude the statement of the Theorem we produce a deformation over each ball $\mathcal{B}$ supported away from the boundary and resulting in a contact structure.\\

\noindent The proof of the statement now uses the results in Section \ref{sec:bands}. Let us precise the necessary details regarding the trivializations. Choose a ball $\mathcal{B}\in\{B_1,\ldots,B_a\}$ and a local chart $\varphi:\mathcal{B}\longrightarrow B^2(1)$. Consider the map $g=\varphi\circ f: f^{-1}(\mathcal{B}) \longrightarrow B^2(1)$. For $\varepsilon>0$ a small constant, we may assume that $g^{-1}(B^2(1)\backslash B^2 (1-\varepsilon))$ is an open set where the distribution $\xi$ is a contact structure.\\

\noindent Consider an exceptional divisor $E$. According to the local model used in Section \ref{sec:blowup}, there exists a neighborhood $\mathcal{E}$ of $E$ and a contactomorphism
$$\varphi_E: (\sS^3 \times\D^2(\delta),\alpha_{std}+ \rho^2 d\phi) \longrightarrow\mathcal{E}.$$
The composition $f \circ \varphi_E:  \sS^3 \times\D^2(\delta) \longrightarrow \sS^2$ restricts to the Hopf fibration at $\sS^3\times\{0\}$. Restricting to the region $f^{-1}(\mathcal{B})\cap\mathcal{E}$ we obtain a fibration
$$\varphi\circ f\circ\varphi_E:\sS^1\times B^2(1) \times\D^2(\delta) \longrightarrow B^2(1)$$
over the 2--ball. Lemma \ref{lem:parallel} implies that the contact parallel transport along the neighborhoods of the boundary is tangent to it. Lemma \ref{lem:paralleltrans} allows us to radially trivialize and express the contact structure as
$$\xi=\ker (\alpha_v + H d\theta).$$
\noindent Observe that the contact fibration is a contact structure in the neighborhood $\mathcal{E}$, therefore $\partial_r H\geq 0$ is satisfied on $\mathcal{E}$. Since $H(p,0,0)=0$, we also conclude that $H\geq 0$ over $\mathcal{E}$. \\

\noindent This setup satisfies the hypotheses of Theorem \ref{thm:band}. It applies producing a homotopy $\xi_t$ of almost contact structures over $f^{-1}(\mathcal{B})$ relative to its boundary such that $\xi_0=\xi$ and $\xi_1$ is a contact structure. The exceptional divisors are contact submanifolds of $\xi_1$ and their induced contact structure is the standard contact structure $\xi_{std}$ with a full Lutz twist performed. The construction is made relative to the pre--image of a neighborhood of the boundary of the ball $\mathcal{B}$. The argument successively applies to the elements of $\{B_1,\ldots,B_a\}$. This concludes the statement.\hfill $\Box$

\subsection{Interpolation at the exceptional divisors.}
Let $(M,\xi,\omega)$ be an almost contact manifold. The argument for proving Theorem \ref{main} begins with a good almost contact pencil $(f,C,E)$. Section \ref{sec:blowup} provides a good ace fibration in a modified manifold $(\widetilde M,\widetilde\xi,\widetilde\omega)$. The results in Sections \ref{sec:vertical}, \ref{sec:skeleton} and \ref{sec:bands} confer good ace fibrations. These exist not on the manifold $(M,\xi,\omega)$ but in $(\widetilde M,\widetilde\xi,\widetilde\omega)$. In the previous subsection a contact structure has been obtained in the almost contact manifold $(\widetilde M,\widetilde\xi,\widetilde\omega)$ such that a neighborhood of the exceptional spheres has remained contact. It is left to obtain a contact structure in the initial manifold $M$.\\

\noindent The exceptional spheres in $(\widetilde{M},\widetilde{\xi})$ have the standard tight contact structure $(\sS^3,\xi_{std})$ at the beginning of the argument. In the deformation performed in Section \ref{sec:bands} the exceptional spheres become overtwisted and we cannot directly obtain a contact structure on $M$. This has a simple solution, we deform the contact distribution on a neighborhood of the exceptional spheres to the standard one. This is the content of the following

\begin{theorem} \label{thm:dishant}
Let $(\sS^3 \times B^2(4),\xi_0)$ have the contact form
\begin{equation}
\eta= \alpha_{ot} + \delta\cdot r^2 d\theta, \label{eq:ot_ne}
\end{equation}
where $\delta\in\R^+$ is a constant and $\alpha_{ot}$ is any contact form associated to an overtwisted contact structure homotopic to the standard contact structure on $\sS^3$.\\

\noindent Let $\xi_{std}$ be a tight contact structure on $\sS^3$. Then there exists a deformation $\xi_1$ of $\xi_0$ supported in $\sS^3 \times B^2(3)$ such that the $\xi_1$ is a contact structure and $\sS^3 \times \{ 0 \}$ inherits the contact structure $\xi_{std}$.
\end{theorem}
\noindent This result is a consequence of Lemma 3.2 in \cite{EP}. Let us give an alternative argument, pointed out to us by Y. Eliashberg. \\

\noindent {\bf Proof of Theorem \ref{thm:dishant}.} Let us begin with the tight contact structure on the $3$--sphere $(\sS^3,\xi_{std})$. Performing a Lutz twist along a given transverse trivial knot $K$ produces an overtwisted contact structure $\xi^1_{ot}$ in $\sS^3$ homotopic to $\xi_{std}$ as an almost contact distribution.  The contact structure $\xi^1_{ot}$ is isotopic to the contact structure $\xi^2_{ot}=\ker\alpha_{ot}$. Consider both a trivial Legendrian knot $L\subset(\sS^3,\xi_{std})$ whose positive transverse push--off is $K$, and its Legendrian push--off $L'$ with two additional zig--zags. According to \cite{DGS} a Lutz twist along $K$ is tantamount to a contact (+1)--surgery along $L$ and $L'$. Hence, given $(\sS^3,\xi^1_{ot})$ there exists a ($-$1)--surgery on $(\sS^3,\xi^1_{ot})$ producing $(\sS^3,\xi_{std})$. Such surgery provides a Liouville cobordism $(W,\lambda)$ from $(\sS^3,\xi^1_{ot})$ to $(\sS^3,\xi_{std})$.\\

\noindent The cobordism obtained by a (+1)--surgery along $L$ and $L'$ can be made smoothly trivial, see \cite{DGS}. Consider $\theta\in\sS^1$ and $\eta^1=\lambda+\mu\cdot d\theta$, for a constant $\mu\in\R^+$. Then the contactization $(W\times\sS^1,\eta^1)$ of the exact symplectic manifold $(W,\lambda)\cong(\sS^3\times[0,1],\lambda)$ is diffeomorphic to $\sS^3\times[0,1]\times\sS^1$. We have obtained a contact structure on the $3$--sphere times the annulus such that the inner boundary $\sS^3\times\{0\}$ has fibres $(\sS^3,\xi_{std})$, and $(\sS^3,\xi^1_{ot})$ are the fibres of the outer bundary $\sS^3\times\{1\}$. The inner part is a convex boundary and it can be filled with the contact manifold
$$(\sS^3\times\D^2,\ker(\alpha_{std}+r^2d\theta))$$
in order to obtain a contact structure on $\sS^3\times\D^2$ with $(\sS^3,\xi_{std})$ as central fibre. For a choice of $\mu$ small enough, there exists a small constant $\delta\in\R^+$ such that in a neighborhood $\sS^3\times(1-\varepsilon,1]\times\sS^1$ of the outer boundary the contact structure can be expressed as
$$\eta^1=\alpha^1_{ot}+\delta\cdot r^2d\theta.$$
The contact forms $\alpha^1_{ot}$ and $\alpha^2_{ot}=\alpha_{ot}$ are isotopic via a family of contact forms $\{\alpha^r_{ot}\}$, $r\in[1,2]$. On the manifold $\sS^3\times[1,4]\times\sS^1$ consider the 1--form
$$\eta^2=\widetilde{\alpha}_{ot}+\delta\cdot r^2d\theta\mbox{ for }r\in[1,2]\mbox{ and } \eta^2=\alpha^2_{ot}+\delta\cdot r^2d\theta \mbox{ for }r\in[2,4]$$
where $\widetilde{\alpha}_{ot}(p,r,\theta)=\alpha^r_{ot}(p)$. The form $\eta^2$ is a contact form because the form $r^2d\theta$ does not depend on the point $p\in\sS^3$. The gluing of the contact forms $\eta^1$ and $\eta^2$ is the required contact structure $\xi_1$ on $\sS^3\times B^2(4)$. \hfill$\Box$\\

\noindent Notice that this deformation gives a homotopy of almost contact structures.

\subsection{Proof of Theorem \ref{main}.}

\noindent Let $(M,\xi,\omega)$ be an almost contact structure. Applying Lemma \ref{lem:almost-quasi} we suppose that $(\xi,\omega)$ is an exact quasi--contact structure. Proposition \ref{prop:good_ac} allows us to construct a good almost contact pencil for an homotopic almost contact structure also referred to as $(\xi,\omega)$. Then Theorem \ref{thm:blow-up} provides a good ace fibration $(f,C,E)$ on an almost contact manifold $(\widetilde{M},\widetilde{\xi},\widetilde{\omega})$, a contact neighborhood $\SN(B)$ of $B$ and a diffeomorphism $\Pi:\widetilde M\setminus E\longrightarrow M\setminus B$ such that $(\Pi_*\widetilde\xi,\Pi_*\widetilde\omega)= (\xi,\omega)$.\\

\noindent Theorems \ref{thm:vert_defor}, \ref{prop:pencil_skeleton} and \ref{thm:filling} subsequently applied to this almost contact manifold and good ace fibration yield a contact structure $\widetilde\xi_c$ on $\widetilde M$. It induces the standard overtwisted structure on the exceptional spheres since a sequence of full Lutz twists are performed. Apply Theorem \ref{thm:dishant} to deform the contact structure to be the initial tight contact structure near each of the exceptional spheres. Then, maybe after a small deformation, it coincides with  $(\widetilde\xi,\widetilde\omega)$ in a tubular neighborhood $\SN(E)$ of $E$. Let us still refer to this contact structure as $\widetilde\xi_c$. The distribution $\Pi_*\widetilde\xi_c$ defines a contact structure on $M\setminus\SN(B)$. It coincides with $(\Pi_*\widetilde\xi,\Pi_*\widetilde\omega)=(\xi,\omega)$ in the submanifold $\Pi(\SN(E)\setminus E)$. The almost contact structure $(\xi,\omega)$ is a contact structure in a neighborhood of $\SN(B)$. In consequence $\Pi_*\widetilde\xi_c$ can be extended to a contact structure $\xi_c$ on $M$. This concludes the proof of the existence of a contact structure $\xi_c$ in the manifold $M$.\\

\noindent Let us prove that $\xi$ and $\xi_c$ are homotopic. There exists a homotopy between $(\widetilde\xi,\widetilde\omega)$ and $\widetilde\xi_c$ over $\widetilde M$. This homotopy restricts to a homotopy over the open submanifold $\widetilde M \setminus E$. Then, the diffeomorphism $\Pi$ yields a homotopy between $(\xi,\omega)$ and $\xi_c$ in the open manifold $M\setminus\SN(B)$. Let us consider a cell decomposition of the manifold $M$ such that $\SN(B)$ does not intersect the 2--skeleton. Such decomposition exists because $B$ is $1$--dimensional, $M$ is $5$--dimensional and the genericity of transversality. Thus $(\xi,\omega)$ and $\xi_c$ are homotopic over the 2--skeleton of this cell decomposition. Then Lemma \ref{lem:2sk} implies that the almost contact structures $(\xi,\omega)$ and $\xi_c$ are also homotopic over $M$.\hfill$\Box$

\subsection{Uniqueness}

The uniqueness of a contact structure in every homotopy class of almost contact structures does not hold in a $5$--manifold. There are many examples in the literature, for instance \cite{Pr2} provides two non--contactomorphic contact structures in the same almost contact homotopy class.\\

\noindent The construction described in this article requires a fair amount of choices. Though, the dependence of the contact structure with respect to them may be understood. The three main ingredients are the stabilization procedure of almost contact pencils, in the same spirit than Giroux's stabilization for a contact open book decomposition ~\cite{Co, Ko}, the addition of fake curves in the triangulation increasing the amount of holes filled with the local model and the surgery procedure.

 \section{Non--coorientable case} \label{sec:non-coorientable}
 
 \subsection{Definitions}
 
 Let $M$ be a ($2n+1$)--dimensional closed manifold, not necessarily orientable. In order to state the Theorem \ref{main} in the non--coorientable setting, we need to give a definition of a non--coorientable almost contact structure. This is a distribution with a suitable reduction of the structure group along with a property requiring a relation between the normal bundle and the distribution. First we introduce the Lie group $\mathfrak{A}(n)$ defined as
 $$\mathfrak{A}(n)=\{A\in O(2n):\quad AJ=\pm JA\},\quad\mbox{where }
 J=\left(\begin{array}{cc}
 0 & Id_n\\
 -Id_n & 0
 \end{array}\right)$$
 
 Notice the following properties:
 \begin{itemize}
 \item[1.] The group $\mathfrak{A}(n)$ has two connected components. It is homeomorphic to $U(n)\times\Z_2$.
 \item[2.] Its group structure is isomorphic to a semidirect product $U(n)\rtimes_\rho\Z_2$. More precisely, let $\mathbb{I}=\left(\begin{array}{cc}
 Id_n & 0\\
 0 & -Id_n
 \end{array}\right)$, then the action
 $$\rho:\Z_2\longrightarrow Aut(U(n)),\quad a\longmapsto (U\longmapsto \mathbb{I}^aU\mathbb{I}^a)$$
 induces the semidirect product structure in the usual way.
 \item[3.] There is a natural group morphism $\mathfrak{s}:\mathfrak{A}(n)\longrightarrow\Z_2$ defined as
 $$\mathfrak{s}(A)=tr(JAJ^{-1}A^{-1})/(2n),$$
 i.e. under the previous isomorphism, $\mathfrak{s}$ is the projection onto the second factor of $U(n)\rtimes_\rho\Z_2$.
 \end{itemize}
 
 \noindent Let us deduce some topological implications of the existence of a contact structure. Let $\xi \subset TM$ be a possibly non--coorientable contact structure on $M$ with a fixed set $\{U_i\}$ of trivializing contractible charts. Choose $\alpha_i$ as a local equation for $\xi|_{U_i}$, then
 $$\alpha_i=a_{ij}\alpha_j,\qquad\mbox{with }a_{ij}:U_i\cap U_j\longrightarrow \{\pm1\}.$$
 This implies that $\{a_{ij}\}$ are the transition function of the normal line bundle $TM/\xi$. Further,  $(d\alpha_i)_{|\xi}=a_{ij}(d\alpha_j)_{|\xi}$. In particular, we may choose a family of compatible complex structures $\{J_i\}$ for the bundle $\xi$ satisfying $J_i=a_{ij}J_j$.\\
 
 \noindent First, note that there is a group injection
 $$\mathfrak{A}(n)\longrightarrow O(2n+1),\quad A\longmapsto\left(\begin{array}{cc}
 A & 0\\
 0 & \mathfrak{s}(A)
 \end{array}\right)$$
 and thus the structure group of $M$ reduces to $\mathfrak{A}(n)$. And second, a $\mathfrak{A}(n)$--bundle $E$ induces via the morphism $\mathfrak{s}$ a real line bundle $\mathfrak{s}(E)$. This construction applied to $\xi$ gives the line bundle $TM/\xi$ in the case above. These two properties will be the ones required in the following:
 
 \begin{definition}
 An almost contact structure on a manifold $M$ is a codimension $1$ distribution $\xi \subset TM$ such that the structure group of $\xi$ reduces to $\mathfrak{A}(n)$ and $\mathfrak{s}(\xi)\cong TM/\xi$.
 \end{definition}
 
 \noindent Observe that the definition for a cooriented almost contact distribution coincides with the one previously given. There are some immediate topological consequences of the existence of such a $\xi$. Indeed:\\
 
 \begin{enumerate}
 \item If $n$ is an even integer, then $\mathfrak{A}(n)\subset SO(2n)$. Thus the distribution $\xi$ is oriented.
 \item If $n$ is an even integer, there is an isomorphism
 \begin{equation}
 TM/\xi\cong det(TM). \label{eq:normdet}
 \end{equation}
 Hence, any almost contact structure in an orientable $5$--dimensional manifold is cooriented. Conversely, any non--orientable $5$--manifold can only admit non--corientable almost contact structures.
 \item If $n$ is an odd integer, then $\mathfrak{s}=det$ as morphisms from $\mathfrak{A}(n)$ to $\Z_2$. Therefore $M$ is orientable since
 $$det(TM)\cong det(\xi\oplus (TM/\xi))\cong det(\xi)\otimes\mathfrak{s}(\xi)\cong det(\xi)^2\cong\R$$
 \end{enumerate}
 
 \noindent Let $M^{2n+1}$ be a non--orientable manifold with $n$ an even integer. Then there exists a canonical $2:1$ cover
$$\pi_2:M_2\longrightarrow M$$
satisfying the following properties:
\begin{itemize}
\item[1.] $M_2$ is an orientable manifold.
\item[2.] Any almost contact structure $\xi$ on $M$ lifts to an almost contact structure $\pi_2^*\xi$ on $M_2$. Moreover, such a distribution is cooriented because of equation (\ref{eq:normdet}).
\end{itemize}

\subsection{Statement of the main result}
Let us state the equivalent of Theorem \ref{main} in the non--coorientable setting:
\begin{theorem}
Let $M$ be a non--orientable closed $5$--dimensional manifold. Let $\xi$ be an almost contact structure. Then there exists a contact structure $\xi_c$ homotopic to $\xi$.
\end{theorem}
 
 \begin{proof}
 Let $\pi_2:(M_2,\pi_2^*\xi)\longrightarrow (M, \xi)$ be an orientable double cover. The constructions developed in this article can be performed in a $\Z_2$--invariant manner. Let us discuss it:
 \begin{enumerate}
 \item An almost contact pencil $(f,B,C)$  can be made $\Z_2$--invariant. To be precise, the loci $B$ and $C$ are $\Z_2$--invariant subsets and $f$ is a $\Z_2$--invariant as a map. In particular the action preserves the fibres. This is because the approximately holomorphic techniques can be developed in that setting. See \cite{IMP} for the details of the construction in the $\Z_2$--invariant setting.\\
 
 \item The deformations performed in Section \ref{sec:defor_local} can easily be done in a $\Z_2$--invariant way. Also, the surgery along a $\Z_2$--invariant loop can be built to preserve that symmetry.\\
 
 \item Subsection \ref{subsec:otdisks} is also prepared for the $\Z_2$--invariant setting. Instead of having a single pair of overtwisted disks, we require two pairs of overtwisted disks. Each pair in the image of the other through the $\Z_2$--action.\\
 
 \item Eliashberg's construction is not $\Z_2$--invariant. Therefore we proceed by quotienting the whole manifold by the $\Z_2$--action, we then obtain an almost contact pencil over the quotient. The fibres are oriented since they are $3$--dimensional almost contact manifolds. The induced almost contact distribution on them is non--coorientable. However, there is no hypothesis on the coorientability in the results of ~\cite{El}. Once the procedure described in Section \ref{sec:vertical} is applied, we consider the orienting double cover.\\
 
 \item Section \ref{sec:skeleton} is trivially adapted to the $\Z_2$--invariant setting if a serious increase of notation is allowed.\\
 
 \item Filling the $2$--cells as in Section \ref{sec:bands} and \ref{sec:end}. We need to produce a $\Z_2$--invariant standard model over $M \times \sS^2$, with $(M,\alpha_0)$ a contact manifold with a $\Z_2$--invariant action. The only required ingredient is to ensuring that the framing $\{\alpha_0,  \alpha_1, \alpha_2 \}$ is chosen $\Z_2$--invariant. The rest of the proof works through up to notation details.\\
 
 \item The arguments in Section \ref{sec:end} are still $\Z_2$--invariant if the previous choices have been done $\Z_2$--invariantly. Therefore, we obtain a $\Z_2$--invariant contact structure $\xi^2_c$ on $M_2$. Its quotient produces a contact structure $\xi_c$ on $M$.
 \end{enumerate}
 This proves the existence part of the statement. The statement concerning the homotopy follows since the homotopies can be easily made $\Z_2$--invariant.
 \end{proof}

\end{document}